\documentclass[aos]{imsart}
\RequirePackage{natbib}
\usepackage{amsmath,amstext,amssymb,amsfonts,amscd,color}
\usepackage{graphicx}
\usepackage{epsfig, natbib}
\usepackage{threeparttable}
\usepackage{cases,color}
\usepackage{multirow}
\usepackage{verbatim}
\usepackage{booktabs}
\usepackage{multicol}
\usepackage{multicol}
\usepackage{bbm}
\usepackage{amsthm}
\usepackage{bm}
\usepackage{subcaption}
\usepackage{chngcntr}
\usepackage[colorlinks,citecolor=blue,urlcolor=blue,linkcolor=blue]{hyperref}

\usepackage{xparse}
\usepackage{tikz}
\usetikzlibrary{matrix,backgrounds}
\pgfdeclarelayer{myback}
\pgfsetlayers{myback,background,main}
\tikzset{mycolor/.style = {dashed,rounded corners,line width=1bp,color=#1}}%
\tikzset{myfillcolor/.style = {draw,fill=#1}}%

\NewDocumentCommand{\highlight}{O{blue!40} m m}{%
	\draw[mycolor=#1] (#2.north west)rectangle (#3.south east);
}
\newcommand{\rmnum}[1]{\romannumeral #1}

\def\cov{{\mbox{cov}}}
\def\var{{\mbox{var}}}

\newtheorem{assumption}{Assumption}
\newenvironment{myassumption}[1]
{\renewcommand\theassumption{#1}\assumption}
{\endassumption}

\def\tr{{\mbox{tr}}}
\newtheorem{thm}{Theorem}[subsection]
\newtheorem{corollary}{Corollary}[subsection]
\newtheorem{example}{Example}[section]
\newtheorem{remark}{Remark}[subsection]
\newtheorem{lemma}{Lemma}

\newtheorem{proposition}{Proposition}[subsection]

\begin{document}
%\title{\bf }
%\date{}
%\author{ { Changbo Zhu, Shun Yao, Xianyang Zhang, Xiaofeng Shao \thanks{ Address correspondence to Xiaofeng Shao (xshao@illinois.edu), Professor, Department of Statistics,  University of Illinois at Urbana-Champaign. Changbo Zhu (changbo2@illinois.edu) is a Ph.D. student in Department of Statistics, University of Illinois at Urbana-Champaign. Shun Yao (shunyao2@illinois.edu) is currently Quantitative Analyst at Goldman Sachs, New York City; Xianyang Zhang (zhangxiany@stat.tamu.edu) is Assistant Professor of Statistics at Texas A\&M University.}}}	
%\maketitle

\begin{frontmatter}
	\title{Distance-based and RKHS-based Dependence Metrics in High Dimension}
	\runtitle{Dependence Metrics in High Dimension}
	%\thankstext{T1}{Footnote to the title with the `thankstext' command.}
	
	\begin{aug}
		\author{\fnms{Changbo} \snm{Zhu}\corref{}\thanksref{m1}
			\ead[label=e1]{changbo2@illinois.edu}}
		\author{\fnms{Shun} \snm{Yao}\thanksref{m2}
			\ead[label=e2]{shunyao2@illinois.edu}}
		\author{\fnms{Xianyang} \snm{Zhang}\thanksref{m3}
			\ead[label=e3]{zhangxiany@stat.tamu.edu}}
		
		\and
		\author{\fnms{Xiaofeng} \snm{Shao}\thanksref{m1,t2}
		\ead[label=e4]{xshao@illinois.edu}}%
		%\ead[label=u1,url]{http://www.foo.com}}
		
		\thankstext{t2}{Address correspondence to Xiaofeng Shao (xshao@illinois.edu), Professor, Department of Statistics,  University of Illinois at Urbana-Champaign. Changbo Zhu (changbo2@illinois.edu) is a Ph.D. student in Department of Statistics, University of Illinois at Urbana-Champaign. Shun Yao (shunyao2@illinois.edu) is currently Quantitative Analyst at Goldman Sachs, New York City; Xianyang Zhang (zhangxiany@stat.tamu.edu) is Assistant Professor of Statistics at Texas A\&M University.}
		
		\runauthor{C. Zhu, S. Yao, X. Zhang and X. Shao}
		
		\affiliation{University of Illinois at Urbana-Champaign\thanksmark{m1}, Goldman Sachs at New York City\thanksmark{m2} and Texas A\&M University\thanksmark{m3}}
		
%		\address{C. Zhu \\
%			 Department of Statistics \\
%			 University of Illinois at Urbana-Champaign \\
%			 Champaign, IL 61820. \\
%			\printead{e1}}
%		
%		\address{S. Yao \\
%			Goldman Sachs at New York City. \\
%		\printead{e2}}
%		
%		\address{X. Zhangb\\
%			Department of Statistics \\
%			Texas A\&M University \\
%			College Station, TX 77843.\\
%			\printead{e3}}
%		
%		\address{X. Shao \\
%			Department of Statistics \\
%			University of Illinois at Urbana-Champaign \\
%			Champaign, IL 61820. \\
%			\printead{e4}}
	\end{aug}
	
	\begin{abstract}
		In this paper, we study distance covariance, Hilbert-Schmidt covariance (aka Hilbert-Schmidt independence criterion [\cite{gretton2007}]) and related independence tests under the high dimensional scenario. We show that the sample distance/Hilbert-Schmidt covariance between two random vectors can be approximated by the sum of squared componentwise sample cross-covariances up to an asymptotically constant factor, which indicates that the distance/Hilbert-Schmidt covariance based test can only capture  linear dependence in high dimension. As a consequence,  the distance correlation based $t$ test developed by \cite{szekely2013} for independence is shown to  have trivial limiting  power when the two random vectors are nonlinearly dependent but component-wisely uncorrelated. This new and surprising phenomenon, which seems to be discovered for the first time,  is  further confirmed in our simulation study. As a remedy, we propose tests based on an aggregation of marginal sample distance/Hilbert-Schmidt covariances and show their superior power behavior against their joint counterparts in simulations. We further extend the distance correlation based $t$ test to those based on Hilbert-Schmidt covariance and marginal distance/Hilbert-Schmidt covariance. A novel unified approach is developed to analyze the studentized sample distance/Hilbert-Schmidt covariance as well as the studentized sample marginal  distance covariance under both null and alternative hypothesis. Our theoretical and simulation results shed light on the limitation of distance/Hilbert-Schmidt covariance when used jointly in the high dimensional setting and suggest the aggregation of marginal distance/Hilbert-Schmidt covariance as a useful alternative.
	\end{abstract}
	
	

\end{frontmatter}

\section{Introduction}
Testing for independence between two random vectors $X\in \mathbb{R}^p$ and $Y\in \mathbb{R}^q$ is a fundamental problem in statistics. There is a huge literature in the low dimensional context. Here we  mention  rank correlation coefficients based tests and  nonparametric Cram\'er-von Mises type statistics  in \cite{hoeffding1948non}, \cite{blum1961}, \cite{de1980cramer};  tests based on signs or empirical characteristic functions, see   \cite{sinha1977}, \cite{deheuvels1981}, \cite{csorgHo1985}, \cite{hettmansperger1994}, \cite{gieser1997},  \cite{taskinen2003}, \cite{stepanova2003} among others; tests based on recently developed nonlinear dependence metrics that target at non-linear and non-monotone dependence include distance covariance [\cite{szekely2007}], Hilbert-Schmidt independence criterion (HSIC) [\cite{gretton2007}] (aka Hilbert-Schmidt covariance in this work) and sign covariance [\cite{bergsma2014}].

In the high dimensional setting, the literature is scarce. \cite{szekely2013} extended  the distance correlation proposed in \cite{szekely2007} to the problem of testing independence of two random vectors under the setting that the dimensions $p$ and $q$ grow while sample size $n$ is fixed. This setting is known as high dimension, low sample size (HDLSS) in the literature and has been adopted in \cite{hall2005}, \cite{ahn2007}, \cite{jung2009}, and \cite{wei2016} etc. A closely related asymptotic framework is the high dimension medium sample size (HDMSS) [\cite{aoshima2018survey}], where $n \wedge p \wedge q \rightarrow \infty$ with $p,q$ growing more rapidly. Among the recent work that is related to independence testing in the high dimensional setting, \cite{pan2014} proposed tests of independence among a large number of high dimensional random vectors using insights from random matrix theory; \cite{yang2015} proposed a new statistic based on the sum of regularized sample canonical correlation coefficients of $X$ and $Y$, which is limited to testing for uncorrelatedness due to the use of canonical correlation. \cite{leung2018testing} proposed to test for mutual independence of high dimensional vectors using sum of pairwise rank correlations and sign covariances;
\cite{yao2017testing} addressed the mutual independence testing problem in the high dimensional context by using sum  of pairwise squared sample distance covariances;
\cite{zhang2018conditional} proposed a $L^2$ type test for conditional mean/quantile dependence of a univariate response variable given a high dimensional covariate vector based on martingale difference divergence [\cite{shao2014martingale}], which is an extension of distance covariance to quantify conditional mean dependence.

Distance covariance/correlation was first introduced in \cite{szekely2007} and has received much attention since then. Owing to its notable ability to quantify any types of dependence including non-monotone, non-linear dependence and also the flexibility to be applicable to two random vectors in arbitrary, not necessarily equal dimensions, a lot of research work has been done to extend and apply distance covariance into many modern statistical problems; see e.g.  \cite{kong2012}, \cite{li2012}, \cite{zhou2012}, \cite{lyons2013}, \cite{szekely2014}, \cite{dueck2014}, \cite{shao2014martingale}, \cite{park2015}, \cite{matteson2016}, \cite{zhang2018conditional} , \cite{edelmann2017}, \cite{yao2017testing} among others.
In this paper, we shall revisit the test proposed by \cite{szekely2013}, which seems to be the only test in the high dimensional setting that captures nonlinear and nonmonotonic dependence. Unlike the positive finding reported in \cite{szekely2013}, we obtained some negative and shocking results that show the limitation of distance covariance/correlation in the high dimensional context.

%For example, in the micro-array data analysis, we can measure thousands of gene expression levels via high throughput technology. Nonetheless the available patients or subjects are very limited and it is often very expensive to increase the sample size. Hence the asymptotic under HDLSS is most suitable for this type of data.

%Here the high dimensionality in  is viewed as a ``blessing'' instead of a ``curse'' as the authors claimed the test performs very well especially for high dimensional data. In this paper, we want to revisit the problem of whether distance covariance/correlation can still capture all kinds of non-linear dependence in high dimension. This is motivated by \cite{reddi2014}, where the authors noticed that even though the estimation error of the statistics itself is independent of dimensionality, it is possible that the power degrades as dimension grows. They demonstrated numerically that the decresing power of distance covariance/correlation based test in high dimension for various data generating processes.

Specifically, we show that for two random vectors $X =(x_1,...,x_p)$ $ \in \mathbb R^p $ and $Y=(y_1,...,y_q) \in \mathbb R^q$ with finite component-wise second moments,  as $p, q \rightarrow \infty$ and $n$ can either be fixed or grows to infinity at a slower rate,
\begin{align}
\label{intro:dcov}
dCov^2_n(\mathbf X, \mathbf Y)  \approx  \frac{1}{\tau} \sum_{i=1}^p \sum_{j=1}^q cov_n^2(\mathcal X_i, \mathcal Y_j),
\end{align}
where $X_{k} \overset{d}{=} X$ and $Y_{k} \overset{d}{=} Y $ are independent samples, $\mathcal{X}_{i}$ and $\mathcal{Y}_{j}$ are the component-wise samples, $\mathbf{X} = ( X_{1},X_{2}, \cdots, X_{n} )^{T} = (\mathcal X_{1}, \mathcal X_{2}, \cdots, \mathcal X_{p} )$ and $\mathbf{Y} = ( Y_{1},Y_{2}, \cdots, Y_{n} )^{T} = (\mathcal Y_{1}, \mathcal Y_{2}, \cdots, \mathcal Y_{q} )$ denote the sample matrices, $dCov^2_n(\mathbf X, \mathbf Y)$ is the unbiased sample distance covariance, $\tau$ is a constant quantity depending on the marginal distributions of $X$ and $Y$ as well as $p$ and $q$, $cov_n^2(\mathcal X_i, \mathcal Y_j)$ is an unbiased sample estimate of $cov^2(x_i, y_j)$ to be defined later.  To the best of our knowledge, this is the first work in the literature uncovering the connection between sample distance covariance and sample covariance, the latter of which can only measure the linear dependence between two random variables.
This approximation suggests that the distance covariance can only measure linear dependence in the high dimensional setting although it is well-known to be capable of capturing non-linear dependence in the fixed dimensional case.

\cite{gretton2007} proposed Hilbert-Schmidt independence criterion (aka Hilbert-Schmidt covariance in this paper), which can be seen as a generalization of distance covariance by kernelizing the $L^2$ distance as shown by \cite{sejdinovic2013}. Despite the kernelization process, we show that the Hilbert-Schmidt covariance ($hCov$) enjoys similar approximation property under high dimension low/medium sample size setting, i.e.
\begin{align}
\label{intro:hcov}
 hCov^2_{n}(\mathbf X, \mathbf Y) \approx A_pB_q \times  \frac{1}{\tau^2} \sum\limits_{i=1}^{p} \sum\limits_{j=1}^{q} cov_{n}^2 (\mathcal X_{i}, \mathcal Y_{j}),
\end{align}
where $ hCov^2_{n}(\mathbf X, \mathbf Y) $ is the unbiased sample Hilbert-Schmidt covariance, $A_p$ and $B_q$ both converge in probability to constants  that depend on the pre-chosen kernels. This aproximation also suggests that when the dimension is high, the Hilbert-Schmidt covariance ($hCov$) applied to the whole components of the vectors also exhibits the loss of power when $X$ and $Y$ are non-linearly dependent, but component-wisely uncorrelated or weakly correlated.
%High dimension, low sample size (HDLSS) asymptotics is well studied in the literature.
%Among others, \cite{hall2005} demonstrated that under suitable conditions the data tend to lie deterministically at the vertices of a regular simplex under HDLSS; \cite{ahn2007} further weakened the conditions for such geometric representation;
% \cite{jung2009} studied the different types of consistency for the Principal Component directions under HDLSS;
%\cite{wei2016} proposed the direction-projection-permutation to assess the validity of binary linear classifiers for HDLSS data.

As a natural  remedy, we  propose a distance covariance based marginal test statistic, i.e.,
\[ mdCov_n^{2}(\mathbf X, \mathbf Y)  = \sqrt{\binom{n}{2}}\sum_{i=1}^p \sum_{j=1}^q dCov_n^2(\mathcal X_i, \mathcal Y_j). \]
%where $\hat{S}^2$ is an estimate of $S^2  = ( \sum_{i,i'=1}^p   dCov^2(x_i,x_{i'}) )( \sum_{j,j'=1}^q dCov^2(y_j,y_{j'}) ).$
This test statistic is an aggregate of the componentwise sample distance covariances and  captures the component by component nonlinear dependence. Similarly, the marginal Hilbert-Schmidt covariance ($mhCov$) is defined as
\[ mhCov_n^2 (\mathbf X, \mathbf Y)  = \sqrt{\binom{n}{2}}\sum_{i=1}^p \sum_{j=1}^q hCov_n^2( \mathcal X_i,\mathcal Y_j). \]

The distance covariance, Hilbert-Schmidt covariance, marginal distance covariance and marginal Hilbert-Schmidt covariance based tests can be carried out by standard permutation procedures. The superiority of $mdCov$ and $mhCov$ based tests over its joint counterparts in power is demonstrated in the simulation studies.
%We further showed that the test statistic has a standard normal limit under the null of independence, as sample size and dimension both grow to infinity.
On the other hand, \cite{szekely2013}  discussed the distance correlation ($dCor$) based $t$-test under HDLSS and derived the limiting null distribution of the test statistic under suitable assumptions. We consider the same $t$-test statistic and further extends to Hilbert-Schmidt covariance ($hCov$), marginal distance covariance ($mdCov$) and marginal Hilbert-Schmidt covariance ($mhCov$). To derive the asymptotic distribution of studentized version of  $dCov, hCov, mdCov $ and $ mhCov $ under both the null  of independence (for HDLSS and HDMSS setting) and some specific alternative classes (for HDLSS setting), we develop a novel unified approach. In particular, we define a unified quantity ($uCov$) based on the bivariate kernel $k$ and show that under HDLSS setting, properly scaled $dCov_n^{2}$, $ hCov_n^{2}$ and $mdCov_n^{2}$ are all asymptotically equal to $uCov_n^2$ up to different choices of kernels, i.e.
\begin{align}
\label{eq:approx}
\begin{array}{ll}
\left.
\begin{array}{l}
dCov_n^{2}(\mathbf X, \mathbf Y) \approx a \times  uCov_n^2(\mathbf X, \mathbf Y)    \\
hCov_n^{2}(\mathbf X, \mathbf Y)  \approx A_pB_q \times  uCov_n^2(\mathbf X, \mathbf Y)
\end{array}  \right\rbrace
 & \text{when } k(x,y)= |x-y|^2, \\
\\
\left.
\begin{array}{l}
mdCov_n^{2}(\mathbf X, \mathbf Y)  = b \times  uCov_n^2(\mathbf X, \mathbf Y)
\end{array} \;\;\;\; \right\rbrace & \text{when } k(x,y)= |x-y| ,
\end{array}
\end{align}
where $a,b$ are constants and $A_p,B_p$ both converge in probability to constants. Next, we show that
\begin{align*}
 \left\lbrace
\begin{array}{ll}
uCov_n^2(\mathbf X, \mathbf Y) \overset{d}{\rightarrow} \frac{2}{n(n-3)}\mathbf{c}^T \mathbf{M} \mathbf{d}, & \text{under HDLSS}, \\
C_{n,p,q} uCov_n^2(\mathbf X, \mathbf Y) \overset{d}{\rightarrow} N(0,1), & \text{under HDMSS} ,
\end{array} \right.
\end{align*}
where $\mathbf{c}, \mathbf{d}$ are jointly Gaussian, $\mathbf{M}$ is a projection matrix and $C_{n,p,q}$ is a normalizing constant. Thus, we can easily apply the above results to $dCov, hCov$ and $mdCov$-based $t$-test statistics using \eqref{eq:approx}. The unified approach still works for $mhCov$-based $t$-test if we consider the bandwidth parameters appeared in the kernel distance to be fixed constants. However, we encounter technical difficulties if the bandwidth parameters along each dimension depends on the whole component-wise samples, since this  makes the pair-wise sample distance correlated with each other and complicates the asymptotic analysis.

We obtain the same limiting null distribution as \cite{szekely2013} and further show that this test statistic has a trivial power against the alternative where $X$ and $Y$ are non-linearly dependent, but component-wisely uncorrelated. This clearly demonstrates that the distance covariance/correlation based joint independence test (i.e., treating all components of a vector as a whole jointly) fails to capture the non-linear dependence in high dimension. This phenomenon is new and was not reported in \cite{szekely2013}. It shows that there might be some intrinsic difficulties for distance covariance to capture the non-linear dependence when the dimension is high and provide a cautionary note on the use of distance covariance/correlation directly to the whole components of high dimensional data. Besides, we have the following additional contributions relative to \cite{szekely2013}: (i) we relax the component-wise i.i.d. assumption used for asymptotic analysis; (ii) the limiting distributions are derived under both the null and certain classes of alternative hypothesis for the HDLSS framework; (iii) our unified approach holds for any bivariate kernel that has continuous second order derivative in a neighborhood containing 1; (iv) the  limiting null distribution is also derived under the HDMSS setting.

%The rest of the paper is organized as follows. In Section \ref{sec:population}, we analyze the population distance covariance as dimension grows; following that Section \ref{sec:sample} shows the decomposition for sample distance covariance and also the asymptotic analysis under both the null of independence and certain alternative classes, we further present the connection between our results and the ones obtained in \cite{szekely2013}; in Section \ref{sec:indep}, we provide the asymptotic analysis for the ``joint'' test statistic using the decomposition we developed, and also introduce the ``marginal'' test of independence; Section \ref{sec:sim} presents the simulation results, which include the HSIC based test; Section \ref{sec:conclusion} summarizes the paper and discusses some related topics for future research. All the technical proofs can be found in the Appendix.

\subsection{Notations}
In this paper, random data samples are denoted as, for each $i = 1, 2, \cdots, n$, $X_{i} \overset{d}{=} X= (x_{1}, \cdots, x_{p})^T \in \mathbb{R}^{p}$, $Y_{i} \overset{d}{=} Y= (y_{1}, \cdots, y_{q})^T \in \mathbb{R}^{q}$. Next, let $\mathbf{X} = ( X_{1},X_{2}, \cdots, X_{n} )^{T}$ and $\mathbf{Y} = ( Y_{1},Y_{2}, \cdots, Y_{n} )^{T}$ denote the random sample matrices. In addition, the random component-wise samples are denoted as $\mathcal{X}_{1}, \cdots, \mathcal{X}_{p}$ and $\mathcal{Y}_{1}, \cdots, \mathcal{Y}_{q}$, which are illustrated in the following table,
\begin{equation*}
\begin{tikzpicture}[baseline=-\the\dimexpr\fontdimen22\textfont2\relax ]
\matrix (m)[matrix of math nodes]
{
	&  & \mathcal{X}_{1} & \mathcal{X}_{2} & \cdots & \mathcal{X}_{p} & &\\
	&  & \color{blue}{\Downarrow} &  &  & & & \\
	X_{1}^{T} & \color{red}{\Rightarrow} & x_{11} & x_{12} & \cdots & x_{1p} & &\\
	X_{2}^{T} &  & x_{21} & x_{22} & \cdots & x_{2p} & &\\
	\vdots &  & \vdots & \vdots &  & \vdots & \color{green}{\Leftarrow} & \mathbf{X}\\
	X_{n}^{T} & & x_{n1} & x_{n2} & \cdots & x_{np} & &\\
};
\begin{pgfonlayer}{myback}
\highlight[red]{m-3-3}{m-3-6}
\highlight[blue]{m-3-3}{m-6-3}
\highlight[green]{m-3-3}{m-6-6}
\end{pgfonlayer}
\end{tikzpicture}
\qquad
\begin{tikzpicture}[baseline=-\the\dimexpr\fontdimen22\textfont2\relax ]
\matrix (m)[matrix of math nodes]
{
	& & \mathcal{Y}_{1} & \mathcal{Y}_{2} & \cdots & \mathcal{Y}_{q} & & \\
	& & &  &  & \color{blue}{\Downarrow} &  & \\
	&  & y_{11} & y_{12} & \cdots & y_{1q} &  & Y_{1}^{T} \\
	\mathbf{Y}& \color{green}{\Rightarrow} & y_{21} & y_{22} & \cdots & y_{2q} &  &  Y_{2}^{T} \\
	& & \vdots & \vdots &  & \vdots &  & \vdots \\
	& & y_{n1} & y_{n2} & \cdots & y_{nq} & \color{red}{\Leftarrow} & Y_{n}^{T} \\
};
\begin{pgfonlayer}{myback}
\highlight[red]{m-6-3}{m-6-6}
\highlight[blue]{m-3-6}{m-6-6}
\highlight[green]{m-3-3}{m-6-6}
\end{pgfonlayer}
\end{tikzpicture}
\end{equation*}
Furthermore, matrices are denoted by  upper case boldface letters (e.g. $\mathbf{A}$, $\mathbf{B}$). For any matrix $\mathbf A = (a_{st}) \in \mathbb{R}^{n \times n}$, we use $\widetilde{ \mathbf A} = (\tilde{a}_{st}) \in \mathbb{R}^{n \times n}$ to denote the $\mathcal U$-centered version of $\mathbf A$, i.e.,
\begin{align*}
\tilde{a}_{st} = \left\lbrace  \begin{array}{ll}
a_{st} -\frac{1}{n-2} \sum_{v=1}^n a_{sv} - \frac{1}{n-2} \sum_{u=1}^n a_{ut}  + \frac{1}{(n-1)(n-2)} \sum_{u,v =1}^{n} a_{uv},&  s \ne t \\
0 , & s=t
\end{array} \right.
\end{align*}
%constant vectors are denoted by lower case boldface letters (e.g. $\mathbf{a}$, $\mathbf{b}$) with corresponding $i$th element as lower case letters (e.g. $a_{i}, b_{i}$),
Following \cite{szekely2014}, the inner product between two $\mathcal U$-centered matrices $\widetilde{\mathbf{A}}= (\tilde{a}_{st}) \in \mathbb{R}^{n \times n} $ and $ \widetilde{\mathbf{B}}= (\tilde{b}_{st}) \in \mathbb{R}^{n \times n} $ is defined as
\begin{align*}
(\widetilde{\mathbf{A}} \cdot \widetilde{\mathbf{B}} ) := \frac{1}{n(n-3)} \sum\limits_{s \neq t} \tilde{a}_{st} \tilde{b}_{st}.
\end{align*}
Next, we use $\mathbf{1}_{n}$ to denote the $n$ dimensional column vector whose entries are all equal to 1. Similarly, we use $\mathbf{0}_{n}$ to denote the $n$ dimensional column vector whose entries are all equal to 0. Finally, we use $ | \cdot | $ to denote the $L^{2}$ norm of a vector, $(X',Y')$ and $(X'',Y'')$ to be independent copies of $(X,Y)$ and $X \perp Y$ to indicate that $X$ and $Y$ are independent.

We  utilize the order in probability notations such as stochastic boundedness $O_{p}$ (big O in probability), convergence in probability $o_{p}$ (small o in probability) and equivalent order $\asymp_p$, which is defined as follows: for a sequence of random variables $\{Z_s\}_{s \in \mathbb Z}$ and a sequence of numbers $\{a_s\}_{s \in \mathbb Z}$, $Z_s \asymp_p a_s$  if and only if $Z_s/a_s = O_p(1) $ and $a_s/Z_s = O_p(1) $ as $s \rightarrow \infty$. For more details about these notations, please see \cite{dasgupta2008asymptotic}.

\section{High Dimension Low Sample Size}
The analyses in this section are conducted under the HDLSS setting, i.e., the sample size $n$ is fixed and the dimensions $p \wedge q \rightarrow \infty$.

\subsection{Distance Covariance and Variants}
\label{statistics}
In this section, we introduce the following test statistics based on distance covariance ($dCov$), marginal distance covariance ($mdCov$), Hilbert-Schmidt covariance ($hCov$) and marginal Hilbert-Schmidt covariance ($mhCov$). In addition, their asymptotic behaviors under the HDLSS setting are derived. The following moment conditions will be used throughout the paper.

%Notice that in the above definition, $p,q$ are arbitrary positive integers. Therefore distance covariance is applicable in the high dimensional setting, where we allow $p,q\rightarrow \infty$. However, it is unclear whether this metric can still retain the power to detect the nonlinear dependence or not when the dimension is high.

\begin{myassumption}{D1}\label{D1}
% \textbf{ \upshape (Bounded second moments)}
For any $p,q$, the variance and the second moment of any coordinate of $X = (x_{1}, x_{2}, \cdots, x_{p})^T$ and $Y = (y_{1}, y_{2}, \cdots, y_{q})^T$ is uniformly bounded below and above, i.e.,
	\begin{align*}
	 &0 < a \leq \inf_{i} var(x_i) \leq \sup_i E (x_i^2 ) \leq b < \infty, \\
	 & 0 < a' \leq \inf_{j} var(y_j) \leq \sup_j E (y_j^2 ) \leq b' < \infty ,
	\end{align*}
	for some constants $a,b,a',b'$.
\end{myassumption}

%\begin{myassumption}{D2}
%Finite fourth moments:
%\label{D2}
%$$
%\sup_p \sup_{1 \le j \le p} E (x_j^4 )< \infty , ~\sup_q \sup_{1 \le j \le q} E (y_j^4 )< \infty
%$$
%\end{myassumption}
%To derive the asymptotic distributions of the test statistics under HDLSS setting, the most commonly used assumption, see e.g., \cite{szekely2013}, is that the components of random vectors $X$ and $Y$ are independent, under which the classic multivariate central limit theorem can be applied directly.
%\begin{myassumption}{D2} \textbf{ \upshape (Component-wise IID)}
%	The components of $X = (x_{1}, x_{2}, \cdots, x_{p})^T$ are i.i.d. as well as the components of $Y = (y_{1}, y_{2}, \cdots, y_{q})^T$, i.e. $x_{i} \perp x_{j} $ for all $1 \leq i,j \leq p$ and $ y_{k} \perp y_{l} $ for all $1 \leq k,l \leq q$.
%\end{myassumption}
%
%As we will see later the component-wise independence is not essential in our technical arguments, what we need is a CLT, which can also hold when there are weak dependence among $X$ (and $Y$).

Next, denote $\tau_X^2 =E|X-X'|^2 $, $\tau_Y^2 =E|Y-Y'|^2 $ and $\tau^2 := \tau_X^2\tau_Y^2= E|X-X'|^2 E|Y-Y'|^2$. Notice that under assumption \ref{D1}, it can be easily seen that
\begin{align*}
\tau_X \asymp \sqrt{p}, \tau_{Y} \asymp \sqrt{q} \text{ and } \tau \asymp \sqrt{pq}.
\end{align*}

The statistics we study in this work use the pair-wise $L^2$ distance between data points. The following proposition presents an expansion formula on the normalized $L^2$ distance when the dimension is high, which plays a key role in our theoretical analysis.
\begin{proposition}
	\label{prop:taylor}
	Under Assumption \ref{D1}, we have
	\begin{align*}
	 \frac{|X-X'|}{\tau_X} & = 1 + \frac{1}{2} L_X(X,X') + R_X(X,X'),
	\end{align*}
	where
	\begin{align*}
	 L_X(X,X') : = \frac{|X-X'|^2 - \tau_{X}^2 }{\tau_{X}^2},
	\end{align*}
	and $R_X(X,X') $ is the remainder term. If we further assume that as $p \wedge q \rightarrow \infty$, $L_X(X,X') =o_p(1)$, then $ R_X(X,X')  = O_p ( L_X(X,X')^2)$. Similar result holds for $Y$.
\end{proposition}

%\begin{remark}
%	It can be easily seen from Chebyshev's inequality,
%	$$
%	E[L_X(X,X')^2]=o(1) \Rightarrow L_X(X,X') =o_p(1).
%	$$
%	Let $\Sigma_X=\cov(X)$, by a straightforward calculation, we obtain
%	$|X-X'|^2=\sum_{j=1}^{p}(x_j-x_j')^2$, $E |X-X'|^2 =2\sum_{j=1}^{p} \var(x_j) = 2\tr(\Sigma_X) $, and
%	\begin{align*}
%	E[L_{X}(X,X')^2]  = \frac{\sum_{j,j'}^{p} [ \cov (x_{j}^2, x_{j'}^{2}) + 2 \cov^2(x_{j}, x_{j'}) ] }{ 2 \text{tr}^2(\Sigma_{X}) }.
%	\end{align*}
%	Therefore, it requires that the component-wise dependence within $X$ is not too strong.
%\end{remark}
%\begin{remark}
%	For example, $E[L_X(X,X')^2]=o(1)$ holds if $X$ has independent identically distributed components or $m$-dependent components, that is, $x_i \perp x_j $ if $|i-j|>m$, where $m$ is a bandwidth that satisfies $m/p \rightarrow 0$.
%\end{remark}

In order for the approximations in equations (\ref{intro:dcov}) and (\ref{intro:hcov}) to work well, it is required that $ L_{X}(X_{s}, X_{t}) $ and $ L_{Y}(Y_{s}, Y_{t}) $ should decay relatively fast as $p \wedge q \rightarrow \infty$. The following assumption specifies the order of $ L_{X}(X_{s}, X_{t}) $ and $ L_{Y}(Y_{s}, Y_{t}) $.
\begin{myassumption}{D2}\label{D2}
	$
	L_{X}(X, X') = O_p (a_{p}) \text{ and } L_{Y}(Y, Y') = O_p (b_{q}),$
	where $a_{p}, b_{q}$ are sequences of numbers such that
	\begin{align*}
	\begin{array}{c}
	a_{p} = o(1), b_{q} = o(1), \\
	\tau_{X}^2 a_p^3 = o(1), \tau_{Y}^2 b_{q}^3 = o(1), \tau a_{p}^2 b_{q} = o(1), \tau a_{p} b_{q}^{2} = o(1).
	\end{array}
	\end{align*}
\end{myassumption}

\begin{remark}
	\label{remark:mainthm}
	A sufficient condition for $ L_X(X,X') =o_p(1) $ is that $ E[L_X(X,X')^2]=o(1)$. Let $\bm{\Sigma}_X=\cov(X)$. By a straightforward calculation, we obtain
	$|X-X'|^2=\sum_{j=1}^{p}(x_j-x_j')^2$, $E |X-X'|^2 =2\sum_{j=1}^{p} \var(x_j) = 2\tr(\bm{\Sigma}_X) $, and
	\begin{align*}
	E[L_{X}(X,X')^2]  = \frac{\sum_{j,j'=1}^{p} [ \cov (x_{j}^2, x_{j'}^{2}) + 2 \cov^2(x_{j}, x_{j'}) ] }{ 2 \text{tr}^2(\bm{\Sigma}_{X}) }.
	\end{align*}
	Therefore, $E[L_X(X,X')^2]=o(1)$ holds if the component-wise dependence within $X$ is not too strong. To illustrate this point, we consider the factor model,
		$$
		X_{p \times 1} = \mathbf{A}_{p \times s_{1}} U_{s_{1} \times 1} + \Phi_{p \times 1}, Y_{q \times 1} = \mathbf{B}_{q \times s_{2}} V_{s_{2} \times 1} + \Psi_{q \times 1},
		$$
		where $\mathbf{A}, \mathbf{B}$ are constant matrices such that $\| \mathbf{A} \|_{F}^{2} = O(p^{1/2})$ and $\| \mathbf{B} \|_{F}^{2} = O(q^{1/2}) $, where $\| \cdot \|_{F}$ is the Frobenius norm. In addition, the components in $U = (u_{1}, \cdots, u_{s_{1}})^{T}$, $V = (v_{1}, \cdots, v_{s_{2}})^{T}$ are independent, $\Phi = ( \phi_{1}, \cdots, \phi_{p} )^T$ is independent of $U$ and $\Psi = (\psi_{1}, \cdots, \psi_{q})^T$ is independent of $V$. Furthermore, the $4$th moment of each component of $ U,V, \Phi, \Psi $ are bounded, i.e.
		$$
		\max \left\lbrace  \sup\limits_{s} E[u_{s}^{4}] , \sup\limits_{t} E[v_{t}^{4}] , \sup\limits_{i} E[\phi_{i}^{4}] , \sup\limits_{j} E[ \psi_{j}^{4}]\right\rbrace< \infty.
		$$
Under Assumption \ref{D1}, the above factor model satisfies Assumption \ref{D2} with $a_{p} = 1/ \sqrt{p} \text{ and } b_q = 1/ \sqrt{q},$  see Section \ref{App:proofRemark} of Appendix for more details.
\end{remark}

\subsubsection{Distance Covariance}
\label{dCov}
Distance covariance was first introduced by \cite{szekely2007} to measure the dependence between two random vectors of arbitrary dimensions. For two random vectors $X \in \mathbb{R}^p$ and $Y\in \mathbb{R}^q$, the (squared) distance covariance is defined as
\[
dCov^2(X,Y) = \int_{\mathbb{R}^{p+q}} \frac{|\phi_{X,Y}(t,s)-\phi_X(t)\phi_Y(s)|^2}{c_pc_q|t|^{1+p} |s|^{1+q}}dtds,
\]
where $c_p=\pi^{(1+p)/2}/\Gamma((1+p)/2)$, $|\cdot| $ is the (complex) Euclidean norm defined as $|x| = \sqrt {\bar x^T x}$ for any vector $x$ in the complex vector space (
$\bar x$ denotes the conjugate of $x$), $\phi_X$ and $\phi_Y$ are the characteristic functions of $X$ and $Y$ respectively, $\phi_{X,Y}$ is the joint characteristic function. According to Theorem 7 of \cite{szekely2009},  an alternative definition of distance covariance is given by
\begin{multline}
\label{def:alt}
dCov^2(X,Y)  =  E|X-X'||Y-Y'| \\ + E|X-X'|E|Y-Y'| - 2 E|X- X'||Y-Y''|,
\end{multline}
where $(X',Y')$ and $(X'',Y'')$ are independent copies of $(X,Y)$. It has been shown that $dCov^2(X,Y)=0$ if and only if $X$ and $Y$ are independent. Therefore, it is able to measure any type of dependence including non-linear and non-monotonic dependence between $X$ and $Y$, whereas the commonly used Pearson correlation can only measure the linear dependence and the rank correlation coefficients (Kendall's $\tau$ and Spearman's $\rho$) can only capture the monotonic dependence.

Notice that in the above setting, $p,q$ are arbitrary positive integers. Therefore, distance covariance is applicable to the high dimensional setting, where we allow $p,q \rightarrow \infty.$ However, it is unclear whether this metric can still retain the power to detect the nonlinear dependence or not when the dimension is high. Distance correlation ($dCor$) is the normalized version of distance covariance, which is defined as
\begin{align*}
dCor^2(X,Y) = \left\lbrace  \begin{array}{ll}
\frac{dCov^2(X,Y)}{\sqrt{dCov^2(X,X) dCov^2(Y,Y) }}, &  dCov^2(X,X) dCov^2(Y,Y) >0, \\
0,	&  dCov^2(X,X) dCov^2(Y,Y) = 0.
\end{array} \right.
\end{align*}

%We now focus on the sample distance covariance.

Following \cite{szekely2014}, we introduce the $\mathcal U$-centering based unbiased sample distance covariance ($dCov_n^2$) as follows.
\begin{align*}
dCov_n^2(\mathbf{X},\mathbf{Y}) = (\widetilde{\mathbf{A}} \cdot \widetilde{\mathbf{B}}  ),
\end{align*}
where $\widetilde{\mathbf{A}}, \widetilde{\mathbf{B}}$ are the $\mathcal U$-centered versions of $\mathbf{A} = (a_{st})_{s,t=1}^n, \mathbf{B} = (b_{st})_{s,t=1}^n$ respectively and $a_{st}= |X_s - X_t|, b_{st}= |Y_s - Y_t|$ for $s,t=1, \cdots ,n$. Correspondingly, the sample distance correlation ($dCor^2_n$) is given as
\begin{align*}
dCor^2_n(\mathbf{X},\mathbf{Y}) = \left\lbrace  \begin{array}{ll}
\frac{dCov^2_n(\mathbf{X},\mathbf{Y})}{\sqrt{dCov^2_n(\mathbf{X},\mathbf{X}) dCov^2_n(\mathbf{Y},\mathbf{Y}) }}, &  dCov^2_n(\mathbf{X},\mathbf{X}) dCov^2_n(\mathbf{Y},\mathbf{Y}) >0, \\
0,	&  dCov^2_n(\mathbf{X},\mathbf{X}) dCov^2_n(\mathbf{Y},\mathbf{Y}) = 0.
\end{array} \right.
\end{align*}
Here we can apply the approximation in Proposition \ref{prop:taylor}, that is
\begin{align}
\label{l2:decomp}
 & \frac{a_{st}}{\tau_X} = 1 + \frac{1}{2} L_{X}(X_s, X_t ) + R_{X}(X_s, X_t ),\\
 & \frac{b_{st}}{\tau_Y} = 1 + \frac{1}{2} L_{Y}(Y_s, Y_t ) + R_{Y}(Y_s, Y_t ), \label{l2:decomp2}
\end{align}
where
\begin{align*}
L_{X}(X_s, X_t ) = \frac{|X_s - X_t|^2 -\tau_{X}^2 }{\tau_{X}^2 }, \;
L_{Y}(Y_s, Y_t ) = \frac{|Y_s - Y_t|^2 -\tau_{Y}^2 }{\tau_{Y}^2 },
\end{align*}
and $R_{X}$, $R_{Y}$ are the remainder terms from the approximation. The approximation of the pair-wise $L^2$ distance in Equations \eqref{l2:decomp} and \eqref{l2:decomp2} is our building block to decompose the unbiased sample (squared) distance covariance ($dCov_n^2$) into a leading term plus a negligible remainder term under the HDLSS setting. The following main theorem summarizes the decomposition properties of sample distance covariance ($dCov_n^2$).
\begin{thm}
	\label{thm:decomp}
	Under Assumption \ref{D1}, we can show that
	\begin{itemize}
		\item[(i)] \begin{align}
		\label{eq:decomp}
		dCov^2_n(\mathbf{X},\mathbf{Y})  =  \frac{1}{\tau} \sum_{i=1}^p \sum_{j=1}^q \cov_n^2 \left(\mathcal X_i, \mathcal Y_j \right)  + {\cal R}_n.
		\end{align}
		Here
		\[
		\cov_n^2 \left( \mathcal X_i, \mathcal Y_j \right) =  \frac{ 1 }{\binom{n}{4}}  \sum_{s< t< u < v  }  h(x_{si} , x_{ti}, x_{ui} ,  x_{vi} ; y_{sj} , y_{tj} , y_{uj} , y_{vj} ),
		\]
		and the kernel $h$ is defined as
		\begin{multline*}
		h(x_{si} , x_{ti}, x_{ui} ,  x_{vi} ;  y_{sj},  y_{tj} , y_{uj} , y_{vj})  \\ =  \frac{1}{4!} \sum_{  * }^{(s,  t , u, v)}    \frac{1}{4}  (x_{si} - x_{ti}) (y_{sj} - y_{tj})(x_{ui} - x_{vi}) (y_{uj} - y_{vj}),
		\end{multline*}
		where the summation $\sum_{  * }^{(s,  t , u, v )}$ is over all permutations of the 4-tuples of indices $(s,t,u,v)$ and ${\cal R}_n $ is the remainder term. $\cov_n^2\left(\mathcal X_i, \mathcal Y_j \right) $
		is a fourth-order U-statistic and is an unbiased estimator for the squared covariance between $x_{i}$ and $y_{j}$, i.e., $E[ \cov_n^2\left(\mathcal X_i, \mathcal Y_j \right) ] = cov^2(x_{i}, y_{j})$.
		\item[(ii)] Further suppose Assumption \ref{D2} holds. Then
		\begin{align*}
		& \frac{1}{\tau} \sum_{i=1}^p \sum_{j=1}^q \cov_n^2 \left(\mathcal X_i, \mathcal Y_j \right) = O_p( \tau a_p b_q ), \\
		& {\cal R}_n = O_p \left( \tau a_p^2 b_q + \tau a_p b_q^2 \right)        = o_p (1),
 		\end{align*}
		thus the remainder term is of smaller order compared to the leading term and therefore is asymptotically  negligible.
	\end{itemize}
\end{thm}

Equation (\ref{eq:decomp}) in Theorem \ref{thm:decomp} shows that the leading term for sample distance covariance is the sum of all component-wise squared sample cross-covariances scaled by $\tau$, which depends on the marginal variances of $X$ and $Y$. This theorem suggests  that in the HDLSS setting, the sample distance covariance can only measure the component-wise linear dependence between the two random vectors.

%\begin{remark}
%	$\cov_n^2\left(\mathcal X_i, \mathcal Y_j \right) $
%	is a fourth-order U-statistic and is an unbiased estimator for the squared covariance between $x_{i}$ and $y_{j}$, i.e., $E[ \cov_n^2\left(\mathcal X_i, \mathcal Y_j \right) ] = cov^2(x_{i}, y_{j})$.
%\end{remark}

\begin{remark}
	It is worth mentioning that the distance correlation and RV coefficient, introduced by \cite{escoufier1970echantillonnage} (see also \cite{josse2013measures}), are asymptotically equal in the HDLSS setting, where RV coefficient is another metric for quantifying the association between two random vectors and is defined as,
	\begin{align*}
	RV(X,Y) =  \frac{\sum_{i=1}^{p} \sum_{j=1}^{q} \text{cov}^2 \left( x_{i}, y_{j} \right) }{\sqrt{ \left(   \sum_{i,j=1}^{p} \text{cov}^2 ( x_{i},  x_{j})  \right) \left(   \sum_{i,j=1}^{q} \text{cov}^2 ( y_{i},  y_{j})  \right) }  }.
	\end{align*}
	If we use $cov_{n}^2 (\mathcal X_i, \mathcal Y_j)$ to estimate $cov^2(x_i, y_j)$, its sample version can be written as
	\begin{align*}
	RV_n(\mathbf X, \mathbf Y) = \frac{\sum_{i=1}^{p} \sum_{j=1}^{q} \text{cov}_{n}^2 \left( \mathcal X_{i}, \mathcal Y_{j} \right) }{\sqrt{ \left(  \sum_{i,j=1}^{p} \text{cov}_{n}^2 (\mathcal X_{i}, \mathcal X_{j})  \right) \left(  \sum_{i,j=1}^{q} \text{cov}_{n}^2 (\mathcal Y_{i}, \mathcal Y_{j})  \right) }  }.
	\end{align*}
	By taking the limit with respect to $p \wedge q$, under Assumptions \ref{D1}, \ref{D2}, \ref{D3} (which will be introduced later) and using Theorem \ref{thm:decomp}, we have
	\begin{align*}
	&  \text{dCor}_{n}^{2}(\mathbf X, \mathbf Y) \\
	= &  \frac{ \frac{1}{\tau} \sum_{i=1}^{p} \sum_{j=1}^{q} \text{cov}_{n}^2 (\mathcal X_{i}, \mathcal Y_{j}) \left(1+ o_p(1) \right)}{\sqrt{ \left( \frac{1}{\tau_{X}^2}  \sum_{i,j=1}^{p} \text{cov}_{n}^2 (\mathcal X_{i}, \mathcal X_{j}) \right) \left( \frac{1}{\tau_{Y}^2}  \sum_{i,j=1}^{q} \text{cov}_{n}^2 ( \mathcal Y_{i}, \mathcal Y_{j}) \right) }  } \\
	= &  \frac{  \sum_{i=1}^{p} \sum_{j=1}^{q} \text{cov}_{n}^2 (\mathcal X_{i}, \mathcal Y_{j}) (1 + o_p(1)) }{\sqrt{ \left(   \sum_{i,j=1}^{p} \text{cov}_{n}^2 (\mathcal X_{i}, \mathcal X_{j})  \right) \left(   \sum_{i,j=1}^{q} \text{cov}_{n}^2 (\mathcal Y_{i}, \mathcal Y_{j})  \right) }  }  \\
	= &  RV_n(\mathbf X, \mathbf Y)(1 + o_p(1)),
	\end{align*}
	which shows that the squared sample distance correlation and sample RV coefficient are approximately equal when the dimension is high. Consequently, they have the same limiting distribution as $p\wedge q$ goes to infinity. Since it is known from \cite{szekely2013} that the studentized version of sample distance covariance has a limiting $t$-distribution, it is expected that the studentized RV coefficient has the same limiting $t$-distribution as well.
\end{remark}

As argued previously, sample distance covariance ($dCov_n^2$) based tests suffer from power loss when $X$ and $Y$ are component-wisely non-linear dependent but uncorrelated. To remedy this drawback, we can consider the following aggregation of marginal sample distance covariances,
\[ mdCov_n^2 (\mathbf X,\mathbf Y)  = \sqrt{\binom{n}{2}}\sum_{i=1}^p \sum_{j=1}^q dCov_n^2(\mathcal X_i, \mathcal Y_j), \]
where $dCov_n^2(\mathcal X_i, \mathcal Y_j) = ( \widetilde{\mathbf{A}}(i) \cdot \widetilde{\mathbf{B}}(j) )$, $\widetilde{\mathbf{A}}(i)$ and $\widetilde{\mathbf{B}}(j)$ are the $\mathcal U$-centered versions of $\mathbf{A}(i) = (a_{st}(i))_{s,t=1}^n, \mathbf{B}(j) = (b_{st}(j))_{s,t=1}^n$ respectively and $a_{st}(i) = |x_{si}-x_{ti}|$, $b_{st}(j) = |y_{sj}-y_{tj}|$.

Note that $mdCov_n^2$ captures the pairwise low dimensional nonlinear dependence, which can be viewed as the main effects of the dependence between two high dimensional random vectors. It is natural in many fields of statistics to test for main effects first before proceeding to high order interactions. See \cite{chakraborty2018} for some discussions on main effects and high order effects in the context of joint dependence testing.
In the testing of mutual independence of a high dimensional vector, \cite{yao2017testing} also approached the problem by testing the pairwise independence using distance covariance and demonstrated that there may be intrinsic difficulty to capture the effects beyond main effects (pairwise dependence in the mutual independence testing problem), as the tests that target joint dependence do not perform well in the high dimensional setting.

\subsubsection{Hilbert-Schmidt Covariance}
\label{hCov}
A generalization of the Distance Covariance ($dCov$) is Hilbert-Schmidt Covariance ($hCov$), first proposed and aka Hilbert-Schmidt independence criterion ($HSIC$) by \cite{gretton2007}. In particular, the (squared) Hilbert-Schmidt Covariance ($hCov$) is obtained by kernelizing the Euclidean distance in equation \eqref{def:alt}, i.e.,
\begin{multline*}
hCov^{2} (X,Y) = E [ K(X, X') L(Y,Y') ] \\ + E [ K(X, X')] E [ L(Y,Y') ] - 2 E [ K(X,X') L(Y,Y'') ],
\end{multline*}
where $(X',Y'), (X'',Y'')$ are independent copies of $(X,Y)$ and $K, L$ are user specified kernels. Following the literature, we consider the following widely used kernels
$$ \left.
\begin{array}{l}
\text{Gaussian kernel: } K(x,y) = \exp \left(- \frac{|x-y|^2}{ 2\gamma^2} \right), \\
\text{Laplacian kernel: }  K(x,y) = \exp \left( -\frac{|x-y|}{ \gamma} \right),
\end{array} \right.
$$
where $\gamma$ is a bandwidth parameter. For later convenience, we focus on the kernels that can be represented compactly as $ K(x,y) =  f \left( |x-y|/ \gamma \right)$ for some continuously differentiable function $ f$. For example, the Gaussian and Laplacian kernel can be defined by choosing different function ${f}$,
$$ \left.
\begin{array}{l}
\text{Gaussian kernel: } {K}(x,y) =  f \left( \frac{|x-y|}{ \gamma} \right),   f(a) = \exp \left( - \frac{a^2}{ 2 } \right), \\
\text{Laplacian kernel: }  {K}(x,y) =  f \left( \frac{|x-y|}{ \gamma} \right),   f(a) = \exp \left( - a \right).
\end{array} \right.
$$
In practice, the bandwidth parameter is usually set as the median of pair-wise sample $L^2$ distance. Thus, a natural estimator for $hCov^2 (X,Y)$ is defined as
%\begin{align*}
%\text{hCov}_{n}^2(X,Y) = \frac{1}{\binom{n}{2}} \frac{1}{2 !} \sum\limits_{(i,j) \in i_{2}^{n}} k_{ij}l_{ij}+ \frac{1}{\binom{n}{4}} \frac{1}{4!} \sum\limits_{(i,j,k,l) \in i_{4}^{n}} k_{ij}l_{kl} - \frac{2}{\binom{n}{3}} \frac{1}{3!} \sum\limits_{(i,j,k) \in i_{3}^{n}} k_{ij} l_{ik},
%\end{align*}
\[
hCov_n^2(\mathbf X, \mathbf Y) = ( \widetilde{\mathbf R} \cdot \widetilde{\mathbf H}),
\]
where $\widetilde{\mathbf{R}}$ and $\widetilde{\mathbf{H}}$ are the $\mathcal U$-centered versions of $\mathbf{R} = (r_{st})_{s,t=1}^n, \mathbf{H} = (h_{st})_{s,t=1}^n$ respectively and
\begin{align*} \left\lbrace
\begin{array}{l}
 r_{st} =  K(X_{s}, X_{t}, \mathbf X) =  f \left( \frac{|X_{s} - X_{t}|}{ \gamma_{\mathbf X}} \right), \gamma_{\mathbf X} = \text{median} \{ |X_{s} - X_{t}|, s \neq t \},\\
 h_{st} = L(Y_{s}, Y_{t}, \mathbf{Y}) = g \left( \frac{|Y_{s} - Y_{t}|}{\gamma_{\mathbf Y}} \right), \gamma_{\mathbf Y} = \text{median} \{ |Y_{s} - Y_{t}|, s \neq t \}.
\end{array}  \right.
\end{align*}
Similar to the definition of distance correlation, the Hilbert-Schmidt Correlation ($hCor$) is defined as
\begin{align*}
hCor^2(X,Y) = \left\lbrace  \begin{array}{ll}
\frac{hCov^2(X,Y)}{\sqrt{hCov^2(X,X) hCov^2(Y,Y) }}, &  hCov^2(X,X) hCov^2(Y,Y) >0, \\
0,	&  hCov^2(X,X) hCov^2(Y,Y) = 0,
\end{array} \right.
\end{align*}
and the sample Hilbert-Schmidt Correlation ($hCor_n^2$) is defined in the same way by
replacing $hCov^2$ with the corresponding sample version.
Next, we can extend the decomposition results for sample distance covariance ($dCov_n^2$) to sample Hilbert-Schmidt covariance ($hCov_n^2$) as shown in the following theorem.
\begin{thm}
	\label{thm:decompHsic}
	Under Assumption \ref{D1}, we have
	\begin{itemize}
		\item[(i)] 	
		\begin{multline}
		\label{eq:decompHSIC}
		\tau \times hCov^2_{n}(\mathbf X, \mathbf Y) \\ =  f^{(1)} \left( \frac{\tau_{X}}{\gamma_{\mathbf X}} \right)  g^{(1)} \left( \frac{\tau_{Y}}{ \gamma_{\mathbf Y}} \right) \frac{\tau_{X}}{\gamma_{\mathbf X}}  \frac{\tau_{Y}}{\gamma_{ \mathbf Y}}  \frac{1}{\tau} \sum\limits_{i=1}^{p} \sum\limits_{j=1}^{q} \text{cov}_{n}^2 (\mathcal X_{i}, \mathcal Y_{j}) + \mathcal{R}_{n},
		\end{multline}
		where $ \text{cov}_{n}^2 $ is defined the same as in Theorem \ref{thm:decomp} and $\mathcal{R}_{n}$ is the remainder term.
		\item[(ii)] Further suppose Assumption \ref{D2} holds. Then
		\begin{align*}
		&  f^{(1)} \left( \frac{\tau_{X}}{\gamma_{\mathbf X}} \right)  g^{(1)} \left( \frac{\tau_{Y}}{ \gamma_{\mathbf Y}} \right)   \frac{\tau_{X}}{\gamma_{\mathbf X}} \frac{\tau_{Y}}{\gamma_{ \mathbf Y}}  \asymp_p 1, \\
		& \frac{1}{\tau} \sum_{i=1}^p \sum_{j=1}^q \cov_n^2 \left(\mathcal X_i, \mathcal Y_j \right) = O_p (\tau a_p b_q), \\
		& {\cal R}_n =  O_p(\tau a_p^2b_q + \tau a_pb_q^2) = o_p (1).
		\end{align*}
		Thus the remainder term is of smaller order compared to the leading term and is therefore asymptotically  negligible.
	\end{itemize}
\end{thm}
Notice that different from the decomposition of $dCov_n^2(\mathbf X,\mathbf Y)$ as in Theorem \ref{thm:decomp}, here we decompose $hCov_n^2$ multiplied by $\tau=\tau_{X} \tau_{Y}$. This is expected, since in $hCov_n^2$, each pair-wise distance is normalized by $\gamma_{\mathbf X}$ or $\gamma_{\mathbf Y}$, which has asymptotically the same magnitude as $\tau_{X}$, $\tau_{Y}$ respectively.
In the high dimensional case, the expansion \eqref{eq:decompHSIC} suggests that $hCov$-based tests also suffer from power loss when $X$ and $Y$ are component-wisely uncorrelated but nonlinearly dependent.

To analyze the asymptotic property of sample Hilbert-Schmidt covariance, most literature would assume the bandwidth parameters to be fixed constants, see e.g. \cite{gretton2007}. In  contrast, our approach can handle the case where these bandwidth parameters are selected to be the median of pairwise sample distance, which is random and whose magnitude increases with dimension.

Similar to the marginal distance covariance introduced in Section \ref{dCov}, we can also aggregate the marginal Hilbert-Schmidt Covariance ($mhCov$), which is defined as
\[ mhCov_n^2 (\mathbf X, \mathbf Y)  = \sqrt{\binom{n}{2}}\sum_{i=1}^p \sum_{j=1}^q hCov_n^2( \mathcal X_i,\mathcal Y_j) \]
where $hCov_n^2(\mathcal X_i, \mathcal Y_j) = (\widetilde{\mathbf R}(i) \cdot \widetilde{\mathbf H}(j))$, $\widetilde{\mathbf R}(i)$ and $\widetilde{\mathbf H}(j)$  are $\mathcal U$-centered version of $\mathbf R (i)= (r_{st}(i))_{s,t=1}^n$, $\mathbf H(j) = (h_{st}(j))_{s,t=1}^n$ respectively and
\begin{align*} \left\lbrace
\begin{array}{l}
 r_{st}(i) =  K(x_{si}, x_{ti}, \mathcal{X}_i) =  f \left( \frac{|x_{si} - x_{ti}|}{ \gamma_{\mathcal{X}_i}} \right), \gamma_{\mathcal X_{i}} = \text{median} \{ |x_{si} - x_{ti}|, s \neq t \},\\
 h_{st}(j) =  L(y_{sj}, y_{tj}, \mathcal{Y}_j) =  g \left( \frac{|y_{sj} - y_{tj}|}{\gamma_{\mathcal Y_j}} \right), \gamma_{\mathcal Y_{j}} = \text{median} \{ |y_{sj} - y_{tj}|, s \neq t \}.
\end{array}  \right.
\end{align*}

\subsection{Studentized Test Statistics}
\label{student}
In this section, we provide studentized version of the statistics introduced in Section \ref{statistics}. It is worth mentioning that we provide a unified approach to the asymptotic analysis of studentized $dCov, mdCov$ and further extend them to the analysis of studentized $hCov$.

\subsubsection{Unified Approach}
\label{sec:uni}
Firstly, we will present results that will be useful for deriving the studentized version of the interested statistics, i.e. distance covariance ($dCov$), marginal distance covariance ($mdCov$), Hilbert-Schmidt Covariance ($hCov$), marginal Hilbert-Schmidt Covariance ($mhCov$). It can be shown later that many previously mentioned statistics are asymptotically equal to the unified quantity $uCov^2_n(\mathbf X, \mathbf Y)$ multiplied by some normalizing factor. Here, $uCov^2_n(\mathbf X, \mathbf Y)$ is defined as
\begin{align*}
uCov^2_{n}(\mathbf X, \mathbf Y) = \frac{1}{\sqrt{pq}}\sum^{p}_{i=1}\sum^{q}_{j=1} (  \widetilde{\mathbf K}(i) \cdot \widetilde{\mathbf L}(j) ),
\end{align*}
where $\widetilde{\mathbf K}(i)$ and $\widetilde{\mathbf L}(j)$ are the $\mathcal U$-centered versions of $\mathbf K(i) = (k_{st}(i))_{s,t=1}^n ,\mathbf L (j) = (l_{st}(j))_{s,t=1}^n $ respectively and $k_{st}(i)$, $l_{st}(i)$ are the double centered kernel distances, i.e., for bivariate kernels $k$ and $l$,
\begin{align*}
&k_{st}(i)=k(x_{si},x_{ti})-E[k(x_{si},x_{ti})|x_{si}]-E[k(x_{si},x_{ti})|x_{ti}]+E[k(x_{si},x_{ti})], \\
&l_{st}(i)=l(y_{si},y_{ti})-E[l(y_{si},y_{ti})|y_{si}]-E[l(y_{si},y_{ti})|y_{ti}]+E[l(y_{si},y_{ti})].
\end{align*}
The advantage of using the double centering kernel distance is that we can have 0 covariance between $k_{st}(i)$ and $k_{uv}(j)$ ($l_{st}(i)$ and $l_{uv}(j)$) for $ \{ s,t \} \neq \{ u,v   \}$ as shown in the following proposition.
\begin{proposition}
\label{prop:ind}
For all $1 \leq i, i' \leq q, 1\leq j, j' \leq p, \text{ if } \{ s, t \} \neq \{ u, v \}$, then
\begin{align*}
E[ k_{st}(i) k_{uv}(i') ] = E[ l_{st}(j) l_{uv}(j') ] =  E[ k_{st}(i) l_{uv}(j) ] = 0.
\end{align*}
\end{proposition}

To derive the limiting distribution of the unified quantity, we need the following assumptions.
\begin{myassumption}{D3} \label{D3} For fixed $n$, as $p \wedge q \rightarrow \infty$,
	\begin{align*}
	\left(
	\begin{array}{c}
	p^{-1/2}\sum^{p}_{i=1}k_{st}(i)   \\
	q^{-1/2}\sum^{q}_{j=1}l_{uv}(j)
	\end{array} \right)_{  s < t ,  u < v }  \overset{d}{\rightarrow} \left(
	\begin{array}{c}
	c_{st} \\
	d_{uv}
	\end{array} \right)_{ s < t ,  u < v },
	\end{align*}
	where $\{ c_{st}, d_{uv} \}_{s<t,u<v}$ are jointly Gaussian. Naturally, we further assume the existence of the following constants that show up in the covariance matrix of $\{ c_{st}, d_{uv} \}$,
	\begin{align*}
	\begin{array}{rl}
	var[c_{st}] := & \sigma_{x}^2 \\
	 =&  \lim\limits_{p} \frac{1}{p} \sum\limits_{i,j=1}^p cov[ k_{st}(i), k_{st}(j)]  \\  = & \left\lbrace
	\begin{array}{ll}
	\lim\limits_{p} \frac{\sum\limits_{i,j=1}^p dCov^2(x_{i}, x_{j})}{p} , & \text{ if } k(x,y)= l(x, y) = |x-y|, \\
	\lim\limits_{p} \frac{\sum\limits_{i,j=1}^p 4 cov^2 (x_{i}, x_{j})}{p} , & \text{ if } k(x,y) = l(x,y)= |x-y|^2 ,
	\end{array} \right.  \\
	 var[d_{st}] : = & \sigma_{y}^2 \\
	 = &  \lim\limits_{q} \frac{1}{q} \sum\limits_{i,j=1}^q cov[ l_{st}(i), l_{st}(j)]   \\
	 = & \left\lbrace
	\begin{array}{ll}
	\lim\limits_{q} \frac{\sum\limits_{i,j=1}^q  dCov^2(y_{i}, y_{j})}{q} , & \text{ if } k(x,y)= l(x,y)= |x-y|, \\
	\lim\limits_{q} \frac{\sum\limits_{i,j=1}^q 4 cov^2 (y_{i}, y_{j})}{q} , & \text{ if } k(x,y)=l(x,y)= |x-y|^2 ,
	\end{array} \right. \\
	 cov[c_{st}, d_{st}] := & \sigma_{xy}^2 \\
	 = & \lim\limits_{p,q} \frac{1}{\sqrt{pq}} \sum\limits_{i=1}^p \sum\limits_{j=1}^q cov[ k_{st}(i), l_{st}(j)]
 \\	 = & \left\lbrace
	\begin{array}{ll}
	\lim\limits_{p,q} \frac{\sum\limits_{i=1}^p \sum\limits_{j=1}^q dCov^2(x_{i}, y_{j})}{\sqrt{pq}} , & \text{ if } k(x,y)=l(x,y)= |x-y|, \\
	\lim\limits_{p,q} \frac{\sum\limits_{i=1}^p \sum\limits_{j=1}^q 4 cov^2 (x_{i}, y_{j})}{\sqrt{pq}} , & \text{ if } k(x,y)= l(x,y)= |x-y|^2.
	\end{array} \right.
	\end{array}
	\end{align*}
\end{myassumption}
\begin{remark}
	Notice that when $ \{ s,t \} \neq \{ u,v \}$, we do not assume the form of $cov[c_{st}, c_{uv}]$, $cov[d_{st}, d_{uv}]$, $cov[c_{st}, d_{uv}]$ in Assumption \ref{D3}, since it follows easily from Proposition \ref{prop:ind} that $cov[c_{st}, c_{uv}] =0, cov[d_{st}, d_{uv}] =0$ and $cov[c_{st}, d_{uv}] =0$ if $ \{ s,t \} \neq \{ u,v   \}.$
\end{remark}
\begin{remark}
	The above Central Limit Theorem (CLT) result can be derived under suitable moment and weak dependence assumptions for the components of $X$ and $Y$. We refer the reader to \cite{doukhan2008} for a relatively recent survey of weak dependence notions and the CLT results under such weak dependence.
\end{remark}
The following theorem is our main result, which shows that the unified quantity converges in distribution to a quadratic form of random variables.
\begin{thm}
	\label{thm:key}
	Fixing $n$ and let $p \wedge q \rightarrow \infty$, under Assumptions \ref{D1} and \ref{D3},
	\begin{align*}
	& uCov^2_{n}(\mathbf X,\mathbf Y) \overset{d}{\rightarrow}  \frac{1}{v} \mathbf{c}^T \mathbf{M} \mathbf{d} ,\\
	& uCov^2_{n}(\mathbf X,\mathbf X) \overset{d}{\rightarrow}  \frac{1}{v} \mathbf{c}^T \mathbf{M} \mathbf{c}  \overset{d}{=} \frac{\sigma^2_x}{v}  \chi^2_{v},\\
	& uCov^2_{n}(\mathbf Y, \mathbf Y) \overset{d}{\rightarrow}  \frac{1}{v} \mathbf{d}^T \mathbf M \mathbf{d} \overset{d}{=} \frac{\sigma^2_y}{v}  \chi^2_{v},
	\end{align*}
	where $v:=  n(n-3)/2$, $ \mathbf M$ is a projection matrix of rank $v$ and
	\begin{align*} \left(
	\begin{array}{c}
	\mathbf{c} \\
	\mathbf{d}
	\end{array} \right)  \overset{d}{=} N \left(\mathbf 0, \left(
	\begin{array}{cc}
	\sigma_{x}^2 \mathbf{I}_{n(n-1)/2} & \sigma_{xy}^2 \mathbf{I}_{n(n-1)/2}  \\
	\sigma_{xy}^2 \mathbf{I}_{n(n-1)/2} & \sigma_{y}^2 \mathbf{I}_{n(n-1)/2}
	\end{array} \right)
	\right).
	\end{align*}
\end{thm}
\begin{remark}
		For the exact form of $\mathbf M$, see the proof of Theorem \ref{thm:key} in the Appendix.
\end{remark}
Next, we define the quantity $T_u$ as
\begin{align*}
T_u  = \sqrt{v-1} \frac{uCor_n^2(\mathbf X, \mathbf Y)}{ \sqrt{1 - (uCor_n^2(\mathbf X, \mathbf Y))^{2} } },
\end{align*}
where
\begin{align*}
uCor_n^2(\mathbf X,\mathbf Y) = \frac{ uCov_n^2(\mathbf X,\mathbf Y) }{\sqrt{ uCov_n^2(\mathbf X,\mathbf X) uCov_n^2(\mathbf Y,\mathbf Y)} }.
\end{align*}
We then define the constants $v$ and $\phi$ that appear in the limiting distribution of $T_u$. Set $v = n(n-3)/2$ and $ \phi = \sigma_{xy}^2 / \sqrt{\sigma_{x}^2\sigma_{y}^2}  $ such that
\begin{align*}
\phi = \phi_{1}\mathbb{I}_{\{ k(x,y)= l(x,y)= |x-y| \}}  + \phi_{2}\mathbb{I}_{\{ k(x,y)= l(x,y)= |x-y|^2 \}},
\end{align*}
where
\begin{align*}
\phi_{1} := & \lim_{p,q} \frac{ \sum_{i=1}^p \sum_{j=1}^q dCov^2(x_{i}, y_{j})}{ \sqrt{  \sum_{i,j=1}^p dCov^2(x_{i}, x_{j})   \sum_{i,j=1}^q dCov^2(y_{i}, y_{j}) }}, \\
\phi_{2} := & \lim_{p,q}  \frac{  \sum_{i=1}^p \sum_{j=1}^{q}  cov^2 (x_{i}, y_{j}) }{ \sqrt{  \sum_{i,j=1}^p  cov^2 (x_{i}, x_{j})   \sum_{i,j=1}^q  cov^2 (y_{i}, y_{j}) } }.
\end{align*}
%\begin{multline*}
%\phi  =\left\lbrace  \begin{array}{lc}
%\phi_{1} := \lim\limits_{p,q} \frac{ \sum\limits_{i=1}^p \sum\limits_{j=1}^q dCov^2(x_{i}, y_{j})}{ \sqrt{  \sum\limits_{i,j=1}^p dCov^2(x_{i}, x_{j})   \sum\limits_{i,j=1}^q dCov^2(y_{i}, y_{j}) }}, & k(x,y)= l(x,y)= |x-y|,   \\
%\phi_{2} := \lim\limits_{p,q}  \frac{  \sum\limits_{i=1}^p \sum\limits_{j=1}^{q}  cov^2 (x_{i}, y_{j}) }{ \sqrt{  \sum\limits_{i,j=1}^p  cov^2 (x_{i}, x_{j})   \sum\limits_{i,j=1}^q  cov^2 (y_{i}, y_{j}) } }, & k(x,y)= l(x,y) = |x-y|^2.
%\end{array} \right.
%\end{multline*}
%and
%$$
%c = \frac{\sigma_{xy}^2}{\sqrt{ \sigma_{x}^2 \sigma_{y}^2 - \sigma_{xy}^4}} = \left\lbrace  \begin{array}{lc}
%c_1 := \sqrt{\frac{\phi_1^2}{1 - \phi_1^2} }, & r=1,   \\
%c_2 := \sqrt{\frac{\phi_2^2}{1 - \phi_2^2} } , & r=2.
%\end{array} \right.
%$$
The limiting distribution of $T_u$ is derived under both null ($H_0$) and alternative $(H_{A})$ hypothesis, i.e.,
$$
\begin{array}{rl}
\text{null hypothesis}: & H_0 = \left\lbrace  (X,Y) \; | \; X \perp Y \right\rbrace  , \\
\text{alternative hypothesis}: & H_A = \{ (X,Y) \; | \; X \not\perp Y \}.
\end{array}
$$
In addition, we also consider the local alternative hypothesis $H_{A_l} \subset H_A$, i.e.,
\begin{align*}
H_{A_{l}} = \left\lbrace  (X,Y) \; \left| \; X \not\perp Y, \phi = \frac{ \phi_{0}}{\sqrt{v}} \right. \right\rbrace ,
\end{align*}
where $v = n(n-3)/2$, $\phi_{0} = \phi_{0,1} \mathbb{I}_{\{ k(x,y)= l(x,y)= |x-y| \}} + \phi_{0,2} \mathbb{I}_{\{k(x,y)= l(x,y)= |x-y|^2\}}$ and $ 0< \phi_{0,1}, \phi_{0,2}< \infty\ $ are constants with respect to $n$. It is also insteresting to compare the asymptotic power under the following class of alternatives $H_{A_{s}} \subset H_A$, i.e.,
\begin{align*}
H_{A_{s}} = \{ (X,Y) \; | \; x_i \not\perp y_j, cov(x_i, y_j) = 0 \text{ for all } 1\leq i \leq p, 1 \leq j \leq q \}.
\end{align*}
In summary, the following table illustrates the value of $\phi$ under different cases we are considering,
$$
\begin{array}{c|cccc}
\phi & H_0 & H_A & H_{A_{l}} & H_{A_{s}} \\ \hline
k(x,y)= l(x,y)= |x-y| & 0 & \phi_1 &\frac{\phi_{0,1}}{\sqrt{v}} & \phi_1 \\
k(x,y)= l(x,y)= |x-y|^2 & 0 & \phi_2 & \frac{\phi_{0,2}}{\sqrt{v}} & 0
\end{array}
$$
Next, let $t_a$ denote the student $t$-distribution with degrees of freedom $a$, $t_{a}^{(\alpha)}$ denotes the $(1-\alpha)$th percentile of $t_a$, $t_{a,b}$ denotes the non-central $t$-distribution with degrees of freedom $a$ and non-central parameter $b$. The asymptotic distribution of $T_u$ is stated in the following proposition.
\begin{proposition}\label{prop:exactT}
	Fix $n$ and let $p \wedge q \rightarrow \infty$. If Assumptions \ref{D1} and \ref{D3} hold, then for any fixed $t \in \mathbb{R}$,
	\begin{align*}
	& P_{H_0}(T_u \leq t ) \rightarrow	 P(t_{v-1} \leq t), \\
	& P_{H_A}(T_u \leq t ) \rightarrow E \left[ P \left( t_{v-1, W} \leq t \right)\right],
	\end{align*}
	where $W \sim \sqrt{\frac{\phi^2}{1 - \phi^2}\chi_{v}^2} $ and $\chi_{v}^2$ is the chi-square distribution with degrees of freedom $v$.
\end{proposition}
\begin{remark}
	For the explicit form of $ E \left[ P \left( t_{v-1, W} \leq t \right)\right] $, see Lemma \ref{lem:exact} in the Appendix.
\end{remark}
Below we derive the large sample approximation of the limiting distribution $ E \left[ P \left( t_{v-1, W} \leq t \right)\right] $ under the local alternative hypothesis ($H_{A_l}$).
\begin{proposition}
	\label{prop:LarA}
	 Under $H_{A_l}$, if we allow $n$ to grow and $t$ is bounded as $n \rightarrow \infty$, $E \left[ P \left( t_{v-1, W} \leq t \right)\right]$ can be approximated as
	 \begin{align*}
	 E_{H_{A_l}} \left[ P \left( t_{v-1, W} \leq t \right)\right] = P \left( t_{v-1, \phi_{0}} \leq t  \right) + O \left(\frac{1}{v} \right),
	 \end{align*}
	 where $\phi_0 = \phi_{0,1}\mathbb{I}_{\{ k(x,y)= l(x,y)= |x-y| \}}  + \phi_{0,2}\mathbb{I}_{\{k(x,y)= l(x,y)= |x-y|^2\}}$. In particular, the result still holds if we replace $t$ with $t_{v-1}^{(\alpha)}$.
\end{proposition}

\subsubsection{Studentized Tests}
For testing the null, permutation test can be used to determine the critical value of the distance covariance ($dCov$), Hilbert-Schmidt covariance ($hCov$), marginal distance covariance ($mdCov$) and marginal Hilbert-Schmidt covariance ($mhCov$) respectively. If $dCov_n^2$, $hCov_n^2$, $mdCov_n^2$ or $mhCov_n^2$ is larger than the corresponding critical value, which can be determined by the empirical permutation distribution function, we reject the null. Alternatively, similar to the construction of $ T_u $, we transform each of $dCov_n^2, hCov_n^2, mdCov_n^2$ and $ mhCov_n^2  $ into a statistic that has asymptotic $t$-distribution under the null. Thus, instead of using permutation test, which can be quite computationally expensive, we can determine the critical value using this asymptotic $t$-distribution. For each $R \in \{ dCov, hCov, mdCov, mhCov \}$,  the studentized test statistic $T_R$ is defined as
\begin{align*}
T_{R} & = \sqrt{v-1} \frac{ R^*(\mathbf X, \mathbf Y)}{ \sqrt{1 - ( R^*(\mathbf X,\mathbf Y))^{2} } },
\end{align*}
where
\begin{align*}
R^*(\mathbf X,\mathbf Y) = \frac{ R_n^2(\mathbf X,\mathbf Y) }{\sqrt{ R_n^2(\mathbf X, \mathbf X) R_n^2(\mathbf Y,\mathbf Y)} }.
\end{align*}
The way to derive the asymptotic distribution of $T_{R}$ is to show that for each $R \in \{dCov, hCov, mdCov \} $, $R_n^2(\mathbf{X}, \mathbf{Y})$ and $ uCov^2_n(\mathbf{X}, \mathbf{Y}) $ are asymptotically equal up to an asymptotically constant factor, as shown below.
\begin{proposition}
	\label{prop:uni}
	Under Assumption \ref{D1},
	\begin{itemize}
		\item[(i)]  When $k(x,y)= l(x,y)= |x-y|^2$,
		\begin{align*}
		& dCov^2_n(\mathbf X,\mathbf Y) = \frac{1}{ 4 } \frac{\sqrt{pq}}{\tau} uCov^2_n(\mathbf X,\mathbf Y) +  {\cal R}'_n, \\
		& \tau \times hCov^2_n(\mathbf X, \mathbf Y)  =\frac{\sqrt{pq}}{ 4 \gamma_{\mathbf X}\gamma_{\mathbf Y} }  f^{(1)} \left( \frac{\tau_{X}}{\gamma_{\mathbf X}} \right)  g^{(1)} \left( \frac{\tau_{Y}}{ \gamma_{\mathbf Y}} \right)    uCov^2_n(\mathbf X, \mathbf Y) +  {\cal R}''_n,
		\end{align*}
		where ${\cal R}'_n, {\cal R}''_n$ are the remainder terms. Further suppose Assumption \ref{D2} holds. Then
		\begin{align*}
		& uCov^2_n(\mathbf X,\mathbf Y) = O_p (\tau a_p b_q), \\
		& {\cal R}'_n =   O_p(\tau a_p^2b_q+ \tau a_pb_q^2) = o_p (1), \\
		& {\cal R}''_n =   O_p(\tau a_p^2b_q+ \tau a_pb_q^2) = o_p (1).
		\end{align*}
		Thus the remainder term is of smaller order compared to the leading term and therefore is asymptotically  negligible.
		\item[(ii)] When $k(x,y)= l(x,y)= |x-y|$,
		\begin{align*}
		mdCov_n^2 (\mathbf X,\mathbf Y) = \sqrt{pq} \sqrt{\binom{n}{2}}  uCov^2_{n}(\mathbf X,\mathbf Y).
		\end{align*}
	\end{itemize}
\end{proposition}
As shown in Proposition \ref{prop:uni}, $k(x,y)= l(x,y)= |x-y|$ would correspond to the $mdCov$-based $t$-test
and $k(x,y)= l(x,y) = |x-y|^2$ would correspond to the$\{dCov, hCov\}$-based $t$-tests. Then, for each $R \in \{ dCov, hCov, mdCov \}$ the asymptotic distribution of $T_{R}$ is given in the following Corollary.
\begin{corollary} \label{cor:uni}
	If Assumptions \ref{D1}, \ref{D2} and \ref{D3} hold, for any fixed $t$ and each $R \in \{ dCov, hCov, mdCov \}$, we have	
	\begin{align*}
	 P_{H_0}(T_R \leq t ) & \rightarrow	 P(t_{v-1} \leq t), \\
	 P_{H_A}(T_R \leq t ) & \rightarrow E \left[ P \left( t_{v-1, W} \leq t \right)\right], \text{ where } W \sim \sqrt{\frac{\phi^2}{1 - \phi^2}  \chi_{v}^2}.
	\end{align*}
\end{corollary}
%\begin{remark}
%Large sample approximation (under $H_{A_l}$) is presented in Proposition \ref{prop:LarA} and the exact form of $E \left[ P \left( t_{v-1, W} \leq t \right)\right]$ is given in Lemma \ref{lem:exact} of Appendix.
%\end{remark}

After knowing the asymptotic distribution of $T_R$ under the null, i.e. $t$-distribution with degrees of freedom $v-1$, we can set critical value as $t_{v-1}^{(\alpha)}$.
Then, from Proposition \ref{prop:exactT}, under the alternative, the asymptotic power of testing the null can be written as a function of $\phi$, i.e.,
\begin{align*}
\label{power2}
Power_{n}(\phi) : =  E \left[ P \left( t_{v-1, W} > t_{v-1}^{\alpha} \right)\right],
\end{align*}
and under $H_{A_l}$, if we allow $n$  to grow
\begin{align*}
Power_{\infty}(\phi_0) : = \lim\limits_{n \rightarrow \infty} Power_{n}\left( \frac{\phi_0}{\sqrt{v}} \right) = P \left( t_{v-1, \phi_{0}} > t_{v-1}^{(\alpha)}  \right).
\end{align*}
We then plot $Power_n(\phi)$ under different combinations of $\alpha$ and $n$, which are shown in Figure \ref{fig1}. It can be seen from Figure \ref{fig1} that larger $\phi$ results in better power and $\phi=0$ corresponds to trivial power. Next, we can actually bound the ratio of $\phi_1$ and $\phi_2$ for standard normal random variables.
\begin{figure}[h]
	\begin{subfigure}[b]{0.5\linewidth}
		\centering
		\includegraphics[width=0.9\linewidth]{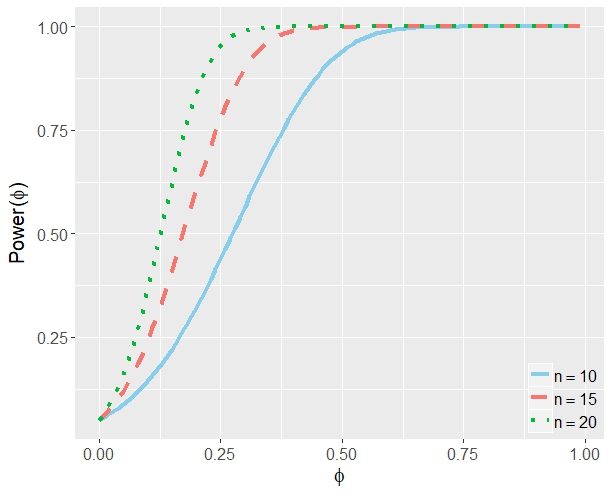}
		\caption{$\alpha=0.05$}
		\label{fig7:a}
	\end{subfigure}%%
	\begin{subfigure}[b]{0.5\linewidth}
		\centering
		\includegraphics[width=0.9\linewidth]{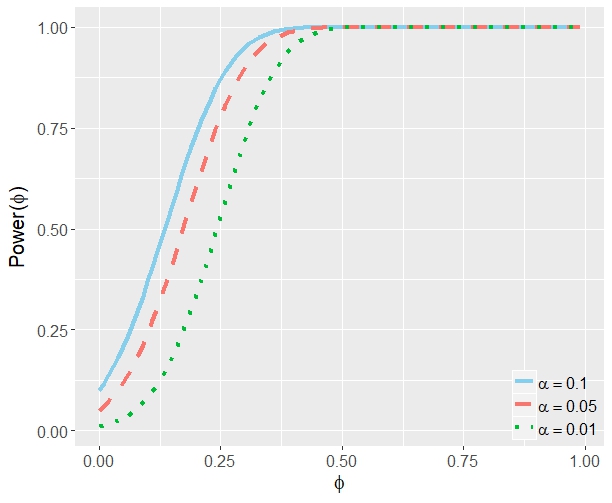}
		\caption{$n=15$}
		\label{fig7:b}
	\end{subfigure}
	\caption{Plot of $Power_n(\phi)$ as a function of $\phi$ under different combinations of $\alpha$ and $n$.}
	\label{fig1}
\end{figure}

\begin{proposition}
	\label{signal:normal}
	Suppose that
	\begin{align*} \left(
	\begin{array}{c}
	X \\
	Y
	\end{array} \right)  \overset{d}{=} N \left(\mathbf 0, \left(
	\begin{array}{cc}
	\mathbf{I}_{p} & \bm{\Sigma}_{XY}  \\
	\bm{\Sigma}_{XY}^T  & \mathbf{I}_{q}
	\end{array} \right)
	\right),
	\end{align*}
	where $\bm{\Sigma}_{XY} = cov(X,Y)$. We have
	\begin{align*}
	0.89^{2} \phi_2 \leq \phi_{1} \leq \phi_2.
	\end{align*}
\end{proposition}
It will be shown later that $\phi_1$ corresponds to the $mdCov$-based test, whereas $\phi_{2}$ corresponds to the $dCov$ and $hCov$-based tests. Thus considering models described in Proposition \ref{signal:normal}, we expect a power loss for the $mdCov$-based test comparing to the $dCov$ and $hCov$-based tests. On the other hand, since $ \phi_{1}$ is bounded below by $ 0.89^{2} \phi_2 $, the power loss is expected to be moderate.

Using Corollary \ref{cor:uni}, we can theoretically compare the power of these $t$-tests under different cases and the results are summarized in the following table
\begin{align*}
\begin{array}{c|cc}
Power &T_{mdCov}&  T_{dCov}, T_{hCov}  \\ \hline
\text{under } H_{A} & Power_n(\phi_1) & Power_n(\phi_2) \\
\text{under } H_{A_{l}}, \text{ allow }n \text{ growing to infinity} & Power_{\infty}(\phi_{0,1}) &  Power_{\infty}(\phi_{0,2}) \\
\text{under } H_{A_{s}} & Power_n(\phi_1) & \alpha \\
\end{array}
\end{align*}
For the studentized version of $mhCov$, if we consider the bandwidth parameters to be fixed constants, then we can use the unified approach to get the limiting $t$-distribution of the transformed $mhCov_n^2$. On the other hand, if $\gamma_{\mathcal X_{i}}$ and $\gamma_{\mathcal Y_{j}}$ are treated to be median of sample distance along each dimension and are thus random, we encounter technical difficulties to derive the limiting distribution, as in this case the kernelized pair-wise distance along each dimension are correlated with each other. This is due to the choice of the bandwidth parameter and the high dimensional approximation used for $hCov_n^2$ can not be directly applied, since $\gamma_{\mathcal X_{i}}$ and $\gamma_{\mathcal Y_{j}}$ are calculated component-wisely. Nevertheless, we shall examine the testing efficiency using $t$-distribution approximation when the bandwidth parameters are chosen to be the median of sample distance in simulation.

\section{High Dimension Medium Sample Size}
Another type of asymptotics closely related to HDLSS is the high dimension medium sample size (HDMSS) setting [\cite{aoshima2018survey}], where $p \wedge q \rightarrow \infty$ and $n \rightarrow \infty$ at a slower rate comparing to $p,q$. The HDMSS setting has been studied by \cite{fan2008sure} and \cite{yata2010effective}, among others.

From the previous sections, we know that the distance/Hilbert-Schmidt covariance can only detect linear dependencies between pair-wise components when $p \wedge q \rightarrow \infty$ and $n$ fixed. In this section, we show that this surprising phenomenon still holds under the high dimension medium sample size setting. Consequently, a unified approach is used to show that $T_{R}$ converges in distribution to standard norml under the null hypothesis, but the technical details of handling the leading term and controlling the remainder are totally different from the fixed $n$ case.

\subsection{Distance Covariance and Variants}
We first state the following assumption which can be seen as an extension of Assumption \ref{D2}.
%\begin{myassumption}{D4}\label{D4}
%	Set
%	$
%	E[L_{X}(X, X')^2] = \alpha_{p}^2, E[L_{Y}(Y, Y')^2] = \beta_{q}^2$,
%	where $\alpha_{p}, \beta_{q}$ are sequences of numbers such that as $n \wedge p \wedge q \rightarrow \infty$
%	\begin{align*}
%	\begin{array}{c}
%	n \alpha_{p} = o(1) \text{, }  n \beta_{q} = o(1),
%	n \tau (\alpha_{p} \beta_{q}^2 + \alpha_{p}^2  \beta_{q} ) = o(1).
%	\end{array}
%	\end{align*}
%\end{myassumption}
%
%\begin{remark}
%	Under Assumption \ref{D1}, it can be shown that for the same time series and factor model stated in Remark \ref{remark:mainthm}, we have $\alpha_{p} = 1/ \sqrt{p} \text{ and } \beta_q = 1/ \sqrt{q},$  see Section \ref{App:proofRemark} of Appendix for more details. Then, for these two models, Assumption \ref{D4} holds if we futher require $n/\sqrt{p}=o(1)$ and $n/\sqrt{q}=o(1)$, i.e., $n$ grows more slowly than $p,q$.
%\end{remark}
\begin{myassumption}{D4}\label{D4} Denote
	$
	E[L_{X}(X, X')^2] = \alpha_{p}^2, E[L_{Y}(Y, Y')^2] = \beta_{q}^2$, $ E[L_{X}(X, X')^4] = \gamma_{p}^2 \text{ and } E[L_{Y}(Y, Y')^4] = \lambda_{q}^2,$
	where $\alpha_{p}, \beta_{q}, \gamma_{p}, \lambda_{q}$ are sequences of numbers such that as $ n\wedge p \wedge q \rightarrow \infty$
	\begin{align*}
	\begin{array}{c}
	n \alpha_{p} = o(1) \text{, }  n \beta_{q} = o(1),  \\
	\tau_{X}^2 (\alpha_{p}\gamma_{p} + \gamma_{p}^2 ) = o(1), \tau_{Y}^2(  \beta_{q}\lambda_{q} + \lambda_{q}^2)  = o(1), \tau (\alpha_{p} \lambda_{q} + \gamma_{p} \beta_{q} +  \gamma_{p} \lambda_{q})  = o(1).
	\end{array}
	\end{align*}
\end{myassumption}

\begin{remark}
\label{rem:proof2}
For the $m$-dependence structure, i.e., $ x_{i} \perp x_{j}$ if $|i-j|>m$ and $ y_{i'} \perp y_{j'}$ if $ |i' - j'| >m'$, where $\sup_{i} E(x_{i}^{8}) < \infty $ and $ \sup_{i} E(y_{i}^{8}) < \infty $, we can show that $\alpha_p = O(\sqrt{m/p})$, $\beta_{q} = O(\sqrt{ m'/q})$, $\gamma_{p}=O(m/p)$ and $\lambda_q=O(m'/q)$.
%\begin{align*}
%\gamma_{p} =O \left( \sqrt{ \frac{m^3}{p^3} + \frac{1}{p^2} } \right), \lambda_{q} = O \left( \sqrt{ \frac{m'^3}{q^3} + \frac{1}{q^2} } \right).
%\end{align*}
Thus, Assumption \ref{D4} holds under the $m$-dependence model if $n$ and $m, m'$ satisfies
\begin{equation*}
\begin{array}{c}
n^2 m = o(p),~~n^2 m' = o(q), \\
m^3=o(p), ~~ m'^3 = o(q), ~~m'm^2=o(p), ~~ m m'^{2}=o(q).
\end{array}
\label{eqcondii}
\end{equation*}
%In particular, \eqref{eqcondii} holds if $m = O(1), m' = O(1)$ and $n = o(\min \{\sqrt{p},\sqrt{q} \})$. For the proof of this remark, see Section \ref{App:proofRemk2} of %Appendix.
\end{remark}
The following theorem shows that the decomposition property \eqref{eq:decomp} for distance covariance still holds under high dimension medium sample size setting.
\begin{thm}
	\label{thm:decomp2}
	Under Assumption \ref{D1}, we can show that	
	\begin{itemize}
		\item[(i)]
	\begin{equation}
	\label{eq:decomp2}
	dCov^2_n(\mathbf{X},\mathbf{Y})  =  \frac{1}{\tau} \sum_{i=1}^p \sum_{j=1}^q \cov_n^2 \left(\mathcal X_i, \mathcal Y_j \right)  + {\cal R}_n.
	\end{equation}
	Here $\cov_n^2$ is defined the same as in Theorem \ref{thm:decomp} and $\mathcal{R}_{n}$ is the remainder term.
%	\item[(i)] Under $H_{0}$ and further suppose Assumption \ref{D4} holds, we have
%	\begin{align*}
%	& \frac{1}{\tau} \sum_{i=1}^p \sum_{j=1}^q \cov_n^2 \left(\mathcal X_i, \mathcal Y_j \right) = O_p( \tau \alpha_{p} \beta_{q} ), \\
%	& {\cal R}_n =  O_p ( \tau \alpha_{p} \beta_{q}^2 + \tau \alpha_{p}^2  \beta_{q}) = o_p(1).
%	\end{align*}
%	Thus the remainder term is of smaller order compared to the leading term and is therefore asymptotically  negligible.
	\item[(ii)] Further suppose Assumption \ref{D4} holds. Then we have
	\begin{align*}
	& \frac{1}{\tau} \sum_{i=1}^p \sum_{j=1}^q \cov_n^2 \left(\mathcal X_i, \mathcal Y_j \right) = O_p( \tau \alpha_{p} \beta_{q} ), \\
	& {\cal R}_n =  O_p (\tau \alpha_{p} \lambda_{q} + \tau \gamma_{p} \beta_{q} + \tau \gamma_{p} \lambda_{q}) = o_p(1).
	\end{align*}
	\end{itemize}	
\end{thm}

Similarly, as shown in the following, $hCov$ also has the decomposition property under HDMSS.
\begin{thm}
	\label{thm:decompHsic2}
	Under Assumption D1, we have
	\begin{itemize}
	\item [(i)]
	\begin{multline}
	\label{eq:decompHSIC2}
	\tau \times hCov^2_{n}(\mathbf X, \mathbf Y) \\ =  f^{(1)} \left( \frac{\tau_{X}}{\gamma_{\mathbf X}} \right)  g^{(1)} \left( \frac{\tau_{Y}}{ \gamma_{\mathbf Y}} \right) \frac{\tau_{X}}{\gamma_{\mathbf X}}  \frac{\tau_{Y}}{\gamma_{ \mathbf Y}}  \frac{1}{\tau} \sum\limits_{i=1}^{p} \sum\limits_{j=1}^{q} \text{cov}_{n}^2 (\mathcal X_{i}, \mathcal Y_{j}) + \mathcal{R}_{n},
	\end{multline}
	where $ \text{cov}_{n}^2 $ is defined the same as in Theorem \ref{thm:decomp} and $\mathcal{R}_{n}$ is the remainder term.	
%	\item[(i)] Under $H_{0}$ and further suppose Assumption \ref{D4} holds, we have
%	\begin{align*}
%	&  f^{(1)} \left( \frac{\tau_{X}}{\gamma_{\mathbf X}} \right)  g^{(1)} \left( \frac{\tau_{Y}}{ \gamma_{\mathbf Y}} \right)   \frac{\tau_{X}}{\gamma_{\mathbf X}} \frac{\tau_{Y}}{\gamma_{ \mathbf Y}}  \asymp_p 1, \\
%	& \frac{1}{\tau} \sum_{i=1}^p \sum_{j=1}^q \cov_n^2 \left(\mathcal X_i, \mathcal Y_j \right) = O_p( \tau \alpha_{p} \beta_{q} ), \\
%	& {\cal R}_n =  O_p ( \tau \alpha_{p} \beta_{q}^2 + \tau \alpha_{p}^2  \beta_{q}) = o_p(1).
%	\end{align*}
%	Thus the remainder term is of smaller order compared to the leading term and is therefore asymptotically  negligible.
	\item[(ii)] Further suppose Assumption \ref{D4} holds. Then
	\begin{align*}
	&  f^{(1)} \left( \frac{\tau_{X}}{\gamma_{\mathbf X}} \right)  g^{(1)} \left( \frac{\tau_{Y}}{ \gamma_{\mathbf Y}} \right)   \frac{\tau_{X}}{\gamma_{\mathbf X}} \frac{\tau_{Y}}{\gamma_{ \mathbf Y}}  \asymp_p 1, \\
	& \frac{1}{\tau} \sum_{i=1}^p \sum_{j=1}^q \cov_n^2 \left(\mathcal X_i, \mathcal Y_j \right) = O_p( \tau \alpha_{p} \beta_{q} ), \\
	& {\cal R}_n =  O_p (\tau \alpha_{p} \lambda_{q} + \tau \gamma_{p} \beta_{q} + \tau \gamma_{p} \lambda_{q}) = o_p(1).
	\end{align*}
	\end{itemize}
\end{thm}
From Equations \eqref{eq:decomp2} and \eqref{eq:decompHSIC2}, we can see that under the HDMSS setting, it is still true that distance/Hilbert-Schmidt covariance can only detect the linear dependence between the components of $X$ and $Y$.

\subsection{Studentized Test Statistics}
Similar to Section \ref{student}, we provide a unified approach to analyze the studentized $dCov, hCov, mdCov$. Since now the sample size is growing, the element-wise argument used to prove the results in Section \ref{student} will no longer work. Inspired by \cite{zhang2018conditional} and \cite{yao2017testing}, we derive the asymptotic distribution by constructing a martingale sequence and using martingale CLT.

\subsubsection{Unified Approach}
For notational convenience, we first define the following metrics,
\begin{align*}
U(X_{s}, X_{t}) :=  \frac{1}{\sqrt{p}}\sum^{p}_{i=1} k_{st}(i), \quad  V(Y_{s},  Y_{t}) := \frac{1}{\sqrt{q}}\sum^{q}_{i=1} l_{st}(i),
\end{align*}
where $k_{st}(i)$ and $l_{st}(i)$ are defined in Section \ref{sec:uni}. To show that the studentized test statistic converges to standard normal, we essentially use the martingale CLT [\cite{hall2014martingale}] and the following assumptions are used to guarantee the conditions in martingale CLT.
\begin{myassumption}{D5}\label{D5}
	\begin{align}	
	& \frac{E \left[  U(X, X')^4  \right]}{ \sqrt{n} (E[ U(X, X')^2 ])^2 } \rightarrow 0, \label{eq:d51} \\
	& \frac{E \left[  U(X, X') U(X', X'')U(X'', X''')U(X''',X ) \right] }{ (E[ U(X, X')^2 ])^2 } \rightarrow 0, \label{eq:d52}
	\end{align}
	and similar assumptions hold for $Y$.
\end{myassumption}

\begin{remark}
When $k(x,y)= l(x,y)= |x-y|$, Assumption \ref{D5} has been studied in Propositions 2.1 and 2.2 of \cite{zhang2018conditional}.
\end{remark}

\begin{remark}
	\label{rem:D7}
	When $k(x,y)= l(x,y)= |x-y|^2$, Equations \eqref{eq:d51} and \eqref{eq:d52} can be simplified to
	\begin{align*}
%	\begin{array}{c}
	&\frac{\sum\limits_{i,j,r,w=1}^p E^2 \left[ (x_{i} - E[x_{i}]) (x_{j} - E[x_{j}]) (x_{r} - E[x_{r}]) (x_{w} - E[x_{w}]) \right] }{ \sqrt{n} Tr^2(\bm{\Sigma}_{X}^2) } \rightarrow 0, \\
	 & \frac{Tr ( \bm{\Sigma}_{X}^4)  }{   Tr^2(\bm{\Sigma}_{X}^2)  } \rightarrow 0, \quad \text{where } \bm{\Sigma}_{X} = cov (X, X).
%	\end{array}	
	\end{align*}
	Notice that $Tr(\bm{\Sigma}_{X}^2) = \sum_{i=1}^{p} \sum_{j=1}^{p} cov^2(x_{i}, x_{j})$. Consider the $m$-dependence model in Remark \ref{rem:proof2}. Assuming $\sup_{i} E(x_{i}^{4}) < \infty $, we have $ Tr ( \bm{\Sigma}_{X}^4) =O(m^3p) $ and
	$$
	 \sum_{i,j,r,w=1}^p E^2 \left[ (x_{i} - E[x_{i}]) (x_{j} - E[x_{j}]) (x_{r} - E[x_{r}]) (x_{w} - E[x_{w}]) \right]  = \\ O(m^2p^2).
	$$
	Consequently, it can be seen that the $m$-dependence model in Remark \ref{rem:proof2} also satisfies Equations \eqref{eq:d51} and \eqref{eq:d52} by controlling the orders of $n, m, m'$.
\end{remark}

Then, we can show that the normalized $uCov_{n}^{2}(\mathbf{X}, \mathbf{Y})$ converges to standard normal distribution under the high dimension medium sample size regime.
\begin{thm}
	\label{thm:key2}
	 Let $n \wedge p \wedge q \rightarrow \infty$. Under $H_{0}$ and Assumption \ref{D5}, we have
		\begin{align*}
		\sqrt{\binom{n}{2}} \frac{uCov_{n}^{2}(\mathbf{X}, \mathbf{Y})}{\mathcal{S}} \overset{d}{\rightarrow} N(0,1), \text{ where } \mathcal{S}^{2} = E[ U(X, X')^{2}] E[ V(Y,  Y')^{2} ].
		\end{align*}
\end{thm}

Consequently, we have the following result.
\begin{proposition}\label{prop:exactT2}
	 Let $n \wedge p \wedge q \rightarrow \infty$. Under $H_{0}$ and Assumption \ref{D5}, we have
		\begin{align*}
		T_u \overset{d}{\rightarrow} N(0,1).
		\end{align*}
\end{proposition}

\subsubsection{Studentized Tests}
The following result shows that as $ n \wedge p \wedge q \rightarrow \infty $, scaled $dCov, hCov$ and $mdCov$ are all equal to $uCov$ up to an asymptotically constant factor.
\begin{proposition}
	\label{prop:uni2}
	Under Assumption \ref{D1},
	\begin{itemize}
		\item[(i)]  When $k(x,y)= l(x,y)= |x-y|^2$,
		\begin{align*}
		& dCov^2_n(\mathbf X,\mathbf Y) = \frac{1}{ 4 } \frac{\sqrt{pq}}{\tau} uCov^2_n(\mathbf X,\mathbf Y) +  {\cal R}'_n, \\
		& \tau \times hCov^2_n(\mathbf X, \mathbf Y)  =\frac{\sqrt{pq}}{ 4\gamma_{\mathbf X}\gamma_{\mathbf Y} }  f^{(1)} \left( \frac{\tau_{X}}{\gamma_{\mathbf X}} \right)  g^{(1)} \left( \frac{\tau_{Y}}{ \gamma_{\mathbf Y}} \right)  uCov^2_n(\mathbf X, \mathbf Y) +  {\cal R}''_n,
		\end{align*}
		where ${\cal R}'_n, {\cal R}''_n$ are the remainder terms.
%			\item[(a)] Under $H_{0}$, further suppose Assumption \ref{D4} holds. Then
%			\begin{align*}
%			& uCov^2_n(\mathbf X,\mathbf Y) = O_p (\tau \alpha_p \beta_q), \\
%			& {\cal R}'_n =   O_p(\tau \alpha_p^2 \beta_q+ \tau \alpha_p \beta_q^2) = o_p (1), \\
%			& {\cal R}''_n =   O_p(\tau \alpha_p^2 \beta_q+ \tau \alpha_p \beta_q^2) = o_p (1).
%			\end{align*}
%			Thus, the remainder term is of smaller order compared to the leading term and therefore is asymptotically  negligible.
			Further suppose Assumption \ref{D4} holds. Then
			\begin{align*}
			& uCov^2_n(\mathbf X,\mathbf Y) = O_p( \tau \alpha_{p} \beta_{q} ), \\
			& {\cal R}_n' =  O_p (\tau \alpha_{p} \lambda_{q} + \tau \gamma_{p} \beta_{q} + \tau \gamma_{p} \lambda_{q}) = o_p(1), \\
			& {\cal R}_n'' =  O_p (\tau \alpha_{p} \lambda_{q} + \tau \gamma_{p} \beta_{q} + \tau \gamma_{p} \lambda_{q}) = o_p(1).
			\end{align*}
			
		\item[(ii)] When $k(x,y)= l(x,y)= |x-y|$,
		\begin{align*}
		mdCov_n^2 (\mathbf X,\mathbf Y) = \sqrt{pq} \sqrt{\binom{n}{2}}  uCov^2_{n}(\mathbf X,\mathbf Y).
		\end{align*}
	\end{itemize}
\end{proposition}
Finally, by adopting a unified approach, we have the following Corollary.
\begin{corollary} \label{cor:uni2}
		Let $n \wedge p \wedge q \rightarrow \infty$. Under $H_{0}$ and Assumption \ref{D5}, we have
		\begin{itemize}
			\item[(i)]  \begin{align*}
			T_{mdCov} \overset{d}{\rightarrow} N(0,1).
			\end{align*}
			\item[(ii)] Further suppose Assumption \ref{D4} and
			\begin{align}
			\label{eq:last}
			\frac{n}{ \sqrt{  \frac{1}{p}  Tr ( \bm{\Sigma}_{X}^2 )  \frac{1}{q}  Tr ( \bm{\Sigma}_{Y}^2 ) } } \tau (\alpha_{p} \lambda_{q} + \gamma_{p}  \beta_{q} + \gamma_{p} \lambda_{q} ) = o(1).
			\end{align}
			Then, for each $R \in \{ dCov, hCov \}$, we have
			\begin{align*}
			T_{R} \overset{d}{\rightarrow} N(0,1).
			\end{align*}
		\end{itemize}		
\end{corollary}

\begin{remark}
The $m$-dependence model in Remark \ref{rem:proof2} can also satisfies Equation \eqref{eq:last} by controlling the orders of $n,m,m'$ based on the magnitude of $ Tr(\bm{\Sigma}_{X}^2)/p $ and $Tr ( \bm{\Sigma}_{Y}^2 )/q $.
\end{remark}

\section{Conclusion}

In this article, we investigate the behavior of the distance covariance and  Hilbert-Schmidt covariance in the high dimensional setting. Somewhat shockingly, we discover that the distance covariance and  Hilbert-Schmidt covariance, which are well-known to capture nonlinear dependence in low/fixed dimensional context, can only capture linear componentwise cross-dependence (to the first order). We believe that this is a new finding that may have significant implications to the design of tests for independence for high dimensional data.
 On one hand, we reveal  the limitation of distance covariance and variants in the high dimensional context, and suggest to use marginally aggregated (sample) distance covariance as a way out, where the latter targets the low dimensional nonlinear dependence.
  On the other hand, we speculate whether it is possible to capture all kinds of dependence between high dimensional vectors $X$ and $Y$, in a limited sample size framework. If the sample size is fixed, we would conjecture that an omnibus test does not exist; If the sample size can grow faster than the dimension, it seems possible but unclear to us how to develop an omnibus test in an asymptotic sense. We hope the results presented in this paper shed some light on the challenges in the high dimensional dependence testing and will motivate more work in this area.

% It is worth noting that there are a few related problems that might warrant immediate attention. It is expected that in the closely related two sample testing problem, the Energy distance and Maximum Mean Discrepancy.

\begin{supplement}
	\sname{Supplement to}\label{suppA}
	\stitle{``Distance-based and RKHS-based Dependence Metrics in High Dimension"}
	\slink[url]{}
	\sdescription{This supplement contains simulations and technical details of the results in the paper.}
\end{supplement}

{
	\bibliographystyle{imsart-nameyear}
	\bibliography{reference}
}

\clearpage
\title{Supplement to ``Distance-based and RKHS-based Dependence Metrics in High Dimension''}

\begin{aug}
	\author{\fnms{Changbo} \snm{Zhu}\corref{}\thanksref{m1}
		\ead[label=e1]{first@somewhere.com}}
	\author{\fnms{Shun} \snm{Yao}\thanksref{m2}
		\ead[label=e2]{second@somewhere.com}}
	\author{\fnms{Xianyang} \snm{Zhang}\thanksref{m3}
		\ead[label=e2]{second@somewhere.com}}
	
	\and
	\author{\fnms{Xiaofeng} \snm{Shao}\thanksref{m1}
		\ead[label=e3]{third@somewhere.com}%
		\ead[label=u1,url]{http://www.foo.com}}
	
	\runauthor{C. Zhu et al.}

	\affiliation{University of Illinois at Urbana-Champaign\thanksmark{m1}, Goldman Sachs at New York City\thanksmark{m2} and Texas A\&M University\thanksmark{m3}}
\end{aug}

\appendix

\section{Simulation Study}
Here, we consider some numerical examples to compare the ``joint'' tests, where the distance/Hilbert-Schmidt covariance is applied to whole components of data jointly,  with the ``marginal'' tests, where distance/Hilbert-Schmidt covariance is applied to one dimensional components and then being aggregated. To this end, we consider the following statistics
\begin{align*}
\text{``Joint"} &  \left\lbrace   \begin{array}{l}
dCov : \text{ distance covariance (permutation)} \\
T_{dCov} : \text{studentized distance covariance} \\
hCov : \text{ Hilbert-Schmidt covariance (permutation)} \\
T_{hCov} : \text{studentized Hilbert-Schmidt covariance}
\end{array} \right.  \\
\text{``Marginal"} &  \left\lbrace  \begin{array}{l}
mdCov : \text{marginal distance covariance (permutation)} \\
T_{mdCov} : \text{studentized marginal distance covariance}, \\
mhCov : \text{marginal Hilbert-Schmidt covariance (permutation)} \\
T_{mhCov} : \text{studentized marginal Hilbert-Schmidt covariance}
\end{array} \right.
\end{align*}
%where
%\begin{align*}
%T^{*} & = \frac{ (pq)^{-1/2}\sum^{p}_{i=1}\sum^{q}_{j=1}\sum_{k\neq
%		l}\tilde{u}_{kl}(i)\tilde{v}_{kl}(i) }{ \sqrt{ \left( (pq)^{-1/2}\sum^{p}_{i=1}\sum^{q}_{j=1}\sum_{k\neq
%			l} (\tilde{u}_{kl}(i))^2 \right) \left( (pq)^{-1/2}\sum^{p}_{i=1}\sum^{q}_{j=1}\sum_{k\neq
%			l} (\tilde{v}_{kl}(i))^2 \right) }}, \\
%T_{u} & = \sqrt{v-1} \frac{T^{*}}{ \sqrt{1 - (T^{*})^{2} } }
%\end{align*}
In the above display, $ dCov_n^2$ and $ hCov_n^2 $ are the two ``joint'' test statistics to measure the overall dependence between $X$ and $Y$,   $mdCov_n^2$ and $mhCov_n^2$ are the ``marginal'' test statistics, and these four test statistics are implemented as permutation tests; $T_{dCov}$ from \cite{szekely2013} is the studentized version of $dCov$, our proposed $t$-tests $T_{hCov}, T_{mdCov}, T_{mhCov}$ are the studentized version of $hCov, mdCov, mhCov$ respectively. All these four tests are implemented using the $t$-distribution based critical value. We examine both the Gaussian kernel and Laplacian kernel for the Hilbert-Schmidt covariance based tests.

For the permutation-based tests, we randomly shuffle the samples $\{X_1,\dots,$ $X_n\}$ and get $(X_{\pi(1)},\dots,X_{\pi(n)})$, where $\pi$ is the permutation map from $\{1,\dots,n\}$ to $\{1,\dots,n\}$. Then we calculate the test statistic based on the permuted sample $\{(X_{\pi(1)},\dots,X_{\pi(n)})$,  $(Y_1,\dots,Y_n)\}$. The $p$-value for permutation-based test is defined as the proportion of times that the test statistic based on the permuted samples is greater than the one based on the original sample.
All the numerical results from permutation-based tests are based on 200 permutations and the empirical rejection rate of the tests are based on 5000 Monte Carlo repetitions.

We first examine the size of the afore-mentioned tests.
\begin{example}
	\label{exp1}
	Generate i.i.d. samples from the following models for $i=1,\dots,n$.
	\begin{itemize}
		\item[(i)]
		$$
		\begin{array}{l}
		X_{i} = \left( x_{i1}, \dots, x_{ip} \right) \sim N(\mathbf 0, \mathbf I_{p}), \\
		Y_{i} = \left( y_{i1}, \dots, y_{ip} \right) \sim N(\mathbf 0, \mathbf I_{p}).
		\end{array} $$
		\item[(ii)] Let $AR(1)$ denotes the Gaussian autoregressive model of order 1 with parameter $\phi$,
		$$
		\begin{array}{l}
		X_{i} \sim AR(1), \phi = 0.5, \\
		Y_{i} \sim AR(1), \phi = -0.5.
		\end{array} $$
		\item[(iii)] Let $\bm \Sigma = (\sigma_{ij}) \in \mathbb{R}^{p \times p} \text{ and }  \sigma_{ij} = 0.7^{|i-j|}$,
		\begin{align*}
		& X_i =(x_{i1},\dots,x_{ip}) \sim N(\mathbf 0, \bm \Sigma), \\
		& Y_i =(y_{i1},\dots,y_{ip}) \sim N(\mathbf 0, \bm \Sigma).
		\end{align*}
	\end{itemize}
\end{example}
From Table \ref{tab1}, we can see that all the tests have quite accurate size. Although the $t$-tests are derived under the high dimensional scenario, they still have pretty accurate size even for relatively low dimension (e.g., $p=5$). In addition, for data samples from Example \ref{exp1} (i), we provide the density plots of the studentized test statistics in Figure \ref{figR2} as well as the density plots of $t_{v-1}$. As we can see, for all cases, the empirical densities are fairly close to that of $t_{v-1}$ and getting closer to $t_{v-1}$ as dimension increases.
\begin{figure}
	\begin{tikzpicture}
	\matrix[
	matrix of nodes,
	nodes={
		anchor=center
	}
	](m){ & n=30 & n=60 \\
		$T_{dCov}$	&  \includegraphics[width=5cm, height=2.7cm ]{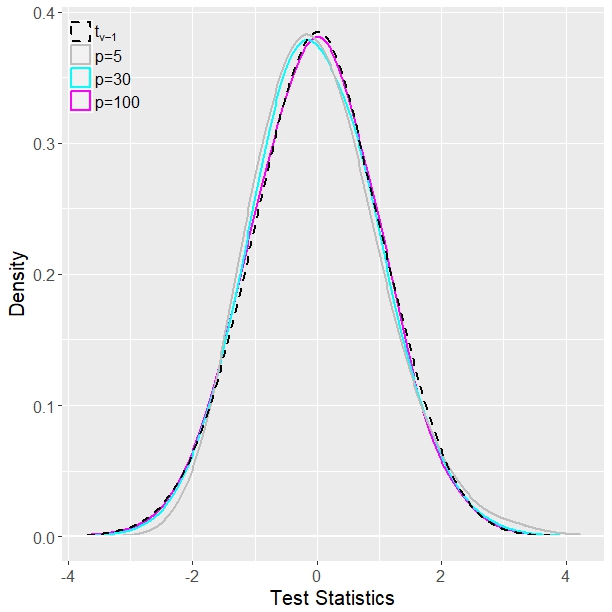} & \includegraphics[width=5cm, height=2.7cm]{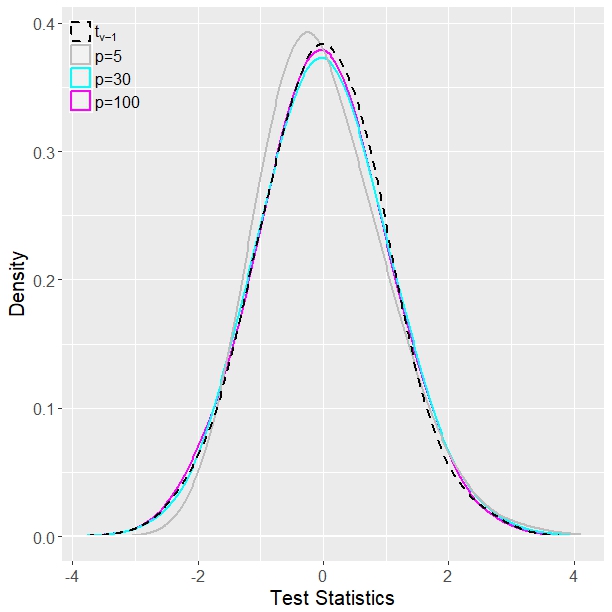}  \\
		$T_{mdCov}$	&  \includegraphics[width=5cm, height=2.7cm ]{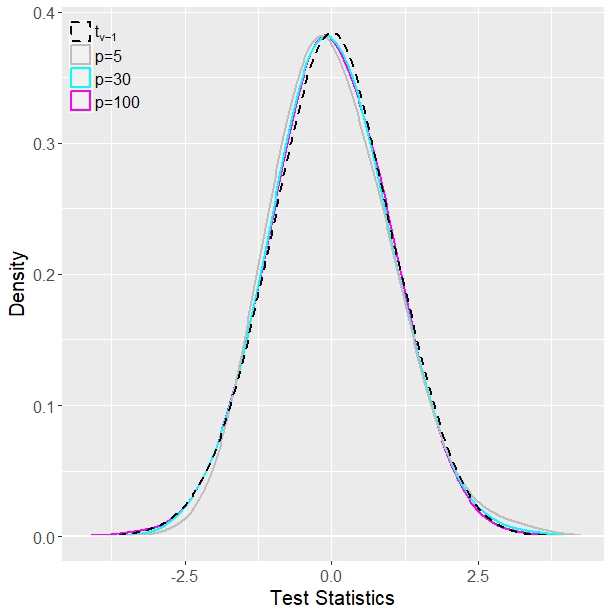} & \includegraphics[width=5cm, height=2.7cm]{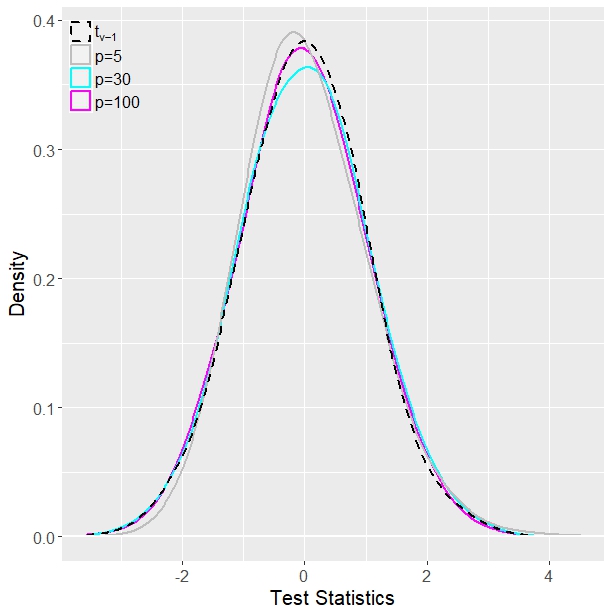}  \\
		$\underset{(\text{Gaussian kernel})}{T_{hCov}}$ 	&  \includegraphics[width=5cm, height=2.7cm ]{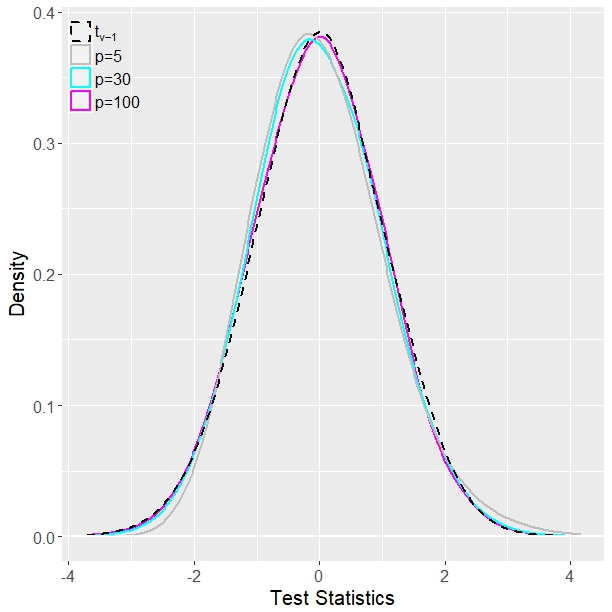} & \includegraphics[width=5cm, height=2.7cm]{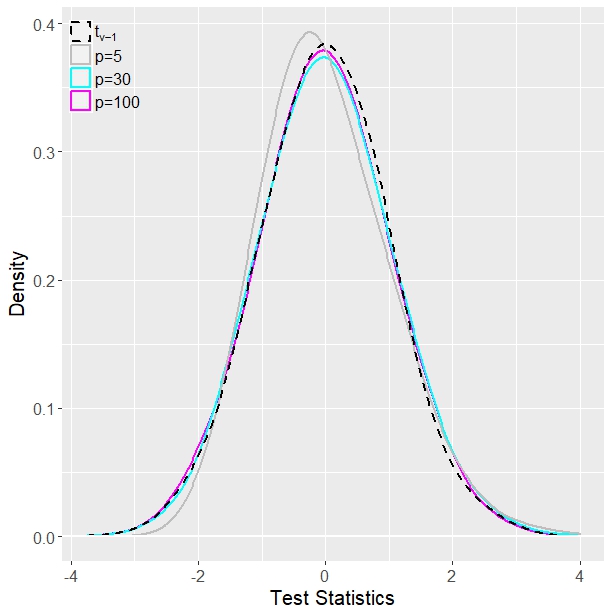}  \\
		$\underset{ (\text{Gaussian kernel}) }{ T_{mhCov} }$ 	&  \includegraphics[width=5cm, height=2.7cm ]{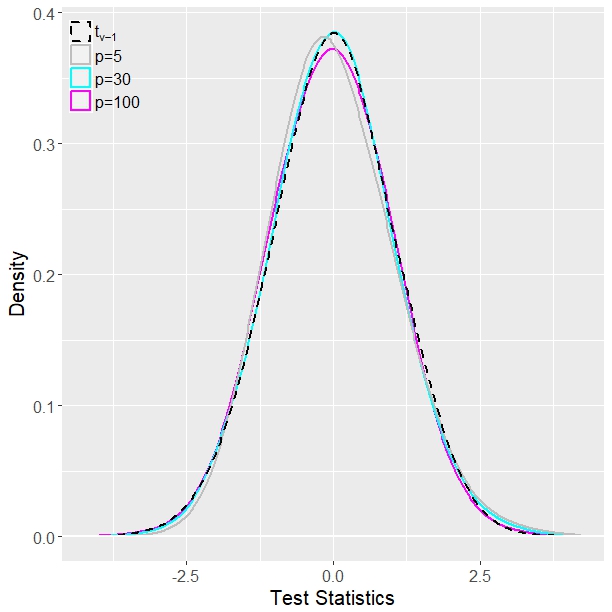} & \includegraphics[width=5cm, height=2.7cm]{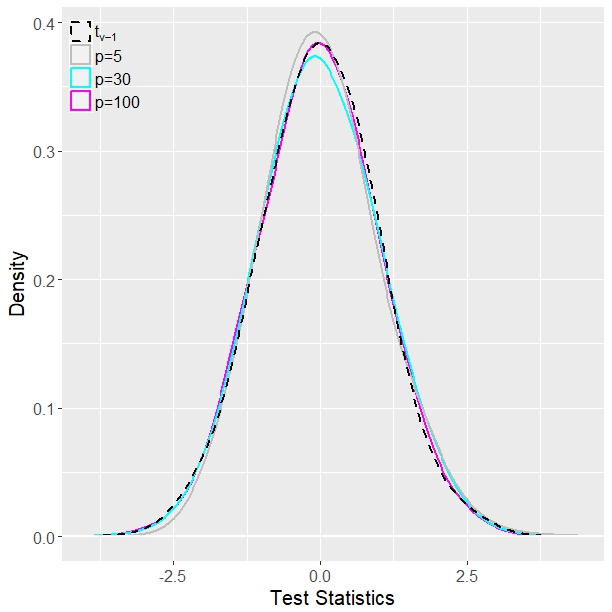}  \\
		$ \underset{ (\text{Laplacian kernel}) }{T_{hCov} }  $ 	&  \includegraphics[width=5cm, height=2.7cm ]{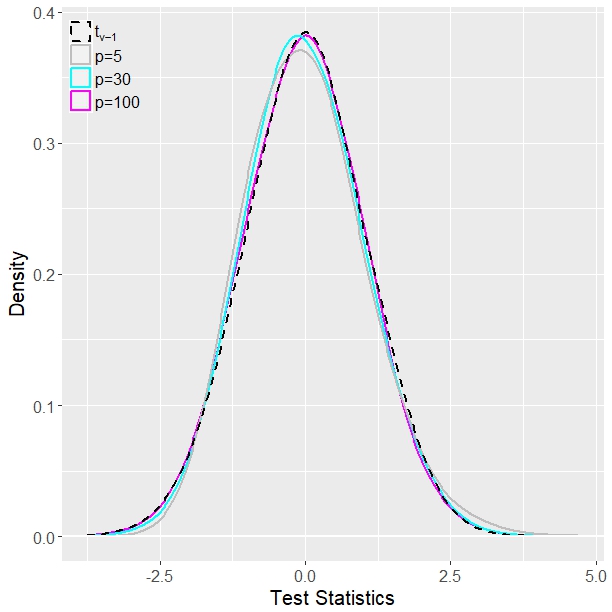} & \includegraphics[width=5cm, height=2.7cm]{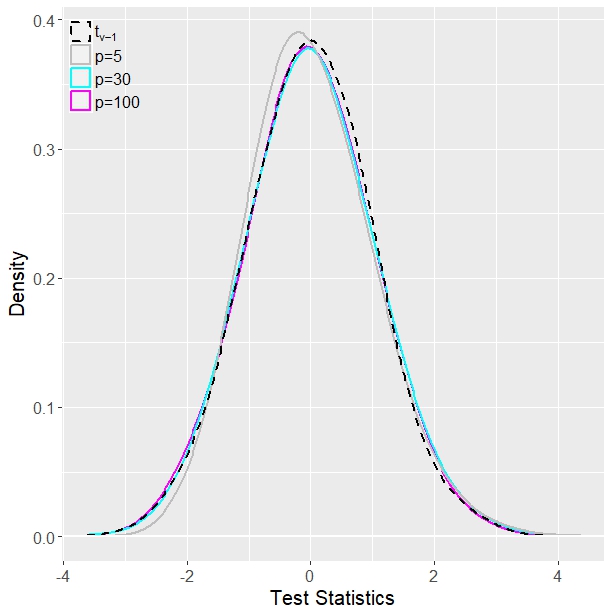}  \\
		$ \underset{ (\text{Laplacian kernel}) }{ T_{mhCov} } $ 	&  \includegraphics[width=5cm, height=2.7cm ]{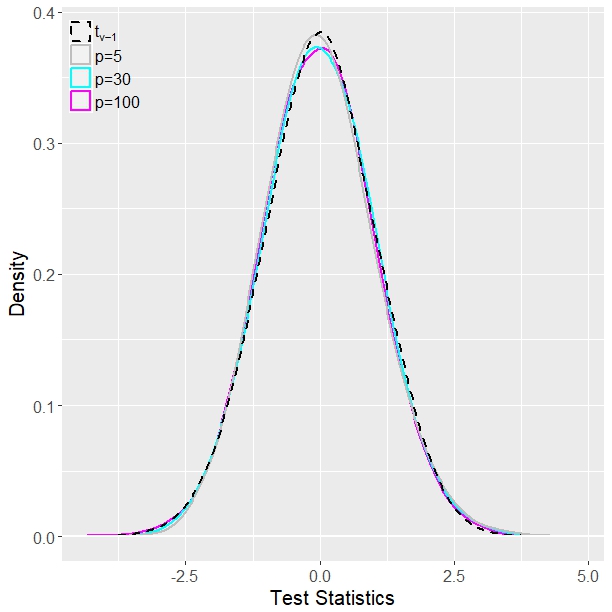} & \includegraphics[width=5cm, height=2.7cm]{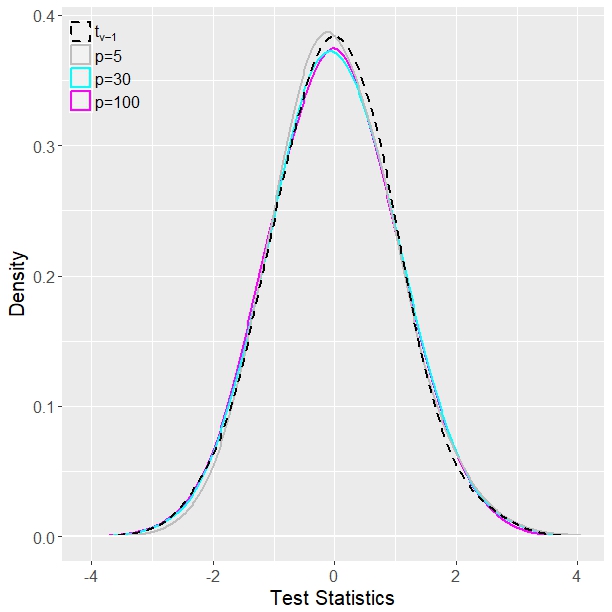} \\
	} ;
	\end{tikzpicture}
	\caption{Density plots of the studentized test statistics (solid colored lines) and $t_{v-1}$ (dashed black line).}
	\label{figR2}
\end{figure}

Notice that under the high dimensional case, the ``joint" tests can be seen as the aggregation of component-wise sample squared covariances. On the other hand, the ``marginal" tests are the accumulation of component-wise sample distance/Hilbert-Schmidt covariances. When $(X,Y)$ are generated from the model in Proposition \ref{signal:normal}, it is expected that there is power loss for $mdCov$ and $mhCov$ based permutation test comparing to $dCov$ and $hCov$ based permutation tests and similar phenomenon is expected for $mdCov$ and $mhCov$ based $t$-test comparing to $dCov$ and $hCov$ based $t$-tests. The following example demonstrates this phenomemon.
\begin{example}
	\label{exp2}
	Generate i.i.d. samples from the following models for $i=1,...,n$.
	\begin{itemize}
		\item[(i)] Let $\rho = 0.5$,
		\begin{align*}
		\begin{array}{l}
		Z_{i} = \left( z_{i1}, \cdots, z_{ip} \right) \sim N(\mathbf 0, \mathbf I_{p}),\\
		X_{i} = \left( x_{i1}, \cdots, x_{ip} \right) \sim N(\mathbf 0, \mathbf I_{p}),\\
		Y_{i} = \frac{\rho X_{i} + (1 - \rho)Z_{i}}{\sqrt{\rho^2 + (1 - \rho)^2}}.
		\end{array}
		\end{align*}	
		
		\item[(ii)] Let $ \rho = 0.7 $ and $(X_i,Y_i,Z_i)$ be defined in the same way as in (i).
		%\begin{align*}
		%\begin{array}{l}
		%Z_{i} = \left( z_{i1}, \cdots, z_{ip} \right) \sim N(\mathbf 0, \mathbf I_{p}), \\
		%X_{i} = \left( x_{i1}, \cdots, x_{ip} \right) \sim N(\mathbf 0, \mathbf I_{p}), \\
		%Y_{i} = \frac{\rho X_{i} + (1 - \rho)Z_{i}}{\sqrt{\rho^2 + (1 - \rho)^2}}.
		%\end{array}
		%\end{align*}
		
		\item[(iii)] Let $ \rho = 0.5 $ and $\otimes$ denote the Kronecker product. Define
		\begin{align*}
		\begin{array}{l}
		Z_{i} = \left( z_{i1}, \cdots, z_{ip} \right) \sim N(\mathbf 0, \mathbf I_{p}), \\
		X_{i} = \left( x_{i1}, \cdots, x_{ip} \right) \sim N(\mathbf 0, \mathbf I_{p}), \\
		Y_{i} = \frac{\rho \bm \Sigma X_i + (1 - \rho)Z_{i}}{\sqrt{\rho^2 + (1 - \rho)^2}},
		\end{array}
		\end{align*}
		where $\bm{\Sigma} = \mathbf{I} \otimes \mathbf{A}$ and $\mathbf{A}$ is an orthogonal matrix defined as
		\begin{align*}
		\mathbf{A} =   \left(
		\begin{array}{ccccc}
		0 & \sqrt{\frac{1}{4}}  & \sqrt{\frac{1}{5}}  & -\sqrt{\frac{1}{4}} & -\sqrt{\frac{3}{10}} \\
		\sqrt{\frac{1}{6}} & \sqrt{\frac{1}{4}} & \sqrt{\frac{1}{5}} & \sqrt{\frac{1}{4}} &  \sqrt{\frac{2}{15}} \\
		-\sqrt{\frac{2}{3}} & 0 & \sqrt{\frac{1}{5}} & 0 & \sqrt{\frac{2}{15}} \\
		\sqrt{\frac{1}{6}} & -\sqrt{\frac{1}{4}} & \sqrt{\frac{1}{5}} & -\sqrt{\frac{1}{4}} & \sqrt{\frac{2}{15}} \\
		0 & -\sqrt{\frac{1}{4}} & \sqrt{\frac{1}{5}} & \sqrt{\frac{1}{4}} & -\sqrt{\frac{3}{10}}		
		\end{array} \right).
		\end{align*}
		%Let $\bm \Sigma = (\sigma_{ij}) \in \mathbb{R}^{p \times p} \text{ and }  \sigma_{ij} = 0.7^{|i-j|}$,
		%		\begin{align*} \left(
		%		\begin{array}{c}
		%		X_{i} \\
		%		Y_{i}
		%		\end{array} \right)  \sim N \left(\mathbf 0, \left(
		%		\begin{array}{cc}
		%		\mathbf{I}_{p} &  \bm \Sigma  \\
		%		 \bm \Sigma  & \mathbf{I}_{p}
		%		\end{array} \right)
		%		\right).
		%		\end{align*}
	\end{itemize}
\end{example}
From Table \ref{tab2}, we can see that there is indeed a power loss for the ``marginal" tests compared to the ``joint" tests, but the loss of power appears fairly moderate, which is consistent with our theory. For Example \ref{exp2}, it can also be observed that the power decrease for the Hilbert-Schmidt covariance based tests is a bit more than the power decrease of distance covariance based tests. Moreover, the power drop is slightly smaller for Gaussian kernel comparing with Laplacian kernel.

As demonstrated in Theorem \ref{thm:decomp} and \ref{thm:decompHsic}, the leading term in (\ref{eq:decomp}) and (\ref{eq:decompHSIC}) can only measure the linear dependence as $p \wedge q \rightarrow \infty$, therefore we expect the ``joint'' test based on $dCov_n^2(\mathbf X,\mathbf Y)$ or $ hCov_n^2(\mathbf X, \mathbf Y) $ may fail to capture the non-linear dependence in high dimension. On the other hand, we consider the ``marginal'' test where we take the sum of pairwise sample distance/Hilbert-Schmidt covariances to measure the low dimensional dependence
for all the pairs as the test proposed in Sections  \ref{dCov} and \ref{hCov}. The ``marginal'' test statistic measures the dependence marginally in a low-dimensional fashion so that it can preserve the ability to capture component-wise non-linear dependence. In the following two examples, we demonstrate the superiority of ``marginal'' tests.
\begin{example}
	\label{exp3}
	Generate $i.i.d.$ samples from the following models for $i=1,...,n$.
	\begin{itemize}
		\item[(i)] 	
		\begin{align*}
		& X_i =(x_{i1},\dots,x_{ip}) \sim N(\mathbf 0, \mathbf I_p), \\
		& Y_i = (y_{i1},\dots,y_{ip}), \mbox{~where~} y_{ij} = x_{ij}^2 \mbox{~for~} j=1,\dots,p.
		\end{align*}
		\item[(ii)] Let $\bm \Sigma = (\sigma_{ij}) \in \mathbb{R}^{p \times p} \text{ and }  \sigma_{ij} = 0.7^{|i-j|}$,
		\begin{align*}
		& X_i =(x_{i1},\dots,x_{ip}) \sim N(\mathbf 0, \bm \Sigma), \\
		& Y_i = (y_{i1},\dots,y_{ip}), \mbox{~where~} y_{ij} = x_{ij}^2 \mbox{~for~} j=1,\dots,p.
		\end{align*}
		
		\item[(iii)] 	
		\begin{align*}
		& X_i =(x_{i1},\dots,x_{ip}) \sim N(\mathbf 0, \mathbf I_p), \\
		& Y_i = (y_{i1},\dots,y_{ip}), \mbox{~where~} y_{ij} = \log |x_{ij}| \mbox{~for~} j=1,\dots,p.
		\end{align*}
	\end{itemize}
\end{example}

\begin{example}
	\label{exp4}
	Generate $i.i.d.$ samples from the following models for $i=1,\dots,n$.
	\begin{itemize}
		\item[(i)] Let $\circ$ denotes the Hadamard product,
		\begin{align*}
		& X_i =(x_{i1},\dots,x_{ip}) \overset{i.i.d.}{\sim}  U(-1, 1), \\
		&Y_i = X_i \circ X_i.
		\end{align*}
		\item[(ii)] \begin{align*}
		& X_i =(x_{i1},\dots,x_{ip}) \overset{i.i.d.}{\sim} U(0, 1), \\
		&Y_i = 4X_i \circ X_i \circ X_i - 3.6X_i +0.8.
		\end{align*}
		\item[(iii)]
		\begin{align*}
		& Z_i =(z_{i1},\dots,z_{ip}) \overset{i.i.d.}{\sim} U(0, 2 \pi), \\
		&X_i = \sin(Z_i), \quad Y_i =\cos(Z_i).
		\end{align*}	
	\end{itemize}
\end{example}

Notice that in the above two examples,
$\cov^2(x_i,y_j)=0$ but $dCov^2(x_i,y_j)$ $ \ne 0$ for all $(i,j)s$, that is, $(X,Y) \in H_{A_s}$. From Table \ref{tab3}, we can observe that for Example \ref{exp3}, the ``joint'' tests suffer substantial power loss as dimension increases for fixed sample size. The power loss is less severe in case (ii) than the ones in cases (i) and (iii), due to the dependence between the components. On the other hand, the powers corresponding to the marginal test statistics consistently outperform their  joint counterparts with very little to none power reduction as the dimension increases. Similar phenomenon can be observed for Example \ref{exp4}; see Table \ref{tab4}.  In addition, for all the cases in both Example \ref{exp3} and Example \ref{exp4}, the power loss corresponding to Laplacian kernel is consistently less than that for Gaussian kernel.
In general, we observe that the tests based on distance covariance, Hilbert-Schmidt covariance with Gaussian kernel, and Hilbert-Schmidt covariance with Laplacian kernel, are all admissible, as none of them dominate the others in all situations.

\begin{table}[!htbp] \centering
	\caption{Size comparison from Example \ref{exp1}}
	\label{tab1}
	\scalebox{0.6}{
		\begin{tabular}{@{\extracolsep{5pt}} cccccccccccccccc}
			\\[-1.8ex]\hline
			\hline \\[-1.8ex]
			&& & & \multicolumn{4}{c}{} & \multicolumn{4}{c}{Gaussian Kernel} & \multicolumn{4}{c}{Laplacian Kernel} \\ \cline{9-12} \cline{13-16}
			&$n$ & $p$ & $\alpha$ & $dCov$ & $mdCov$ & $T_{dCov}$ & $T_{mdCov}$ & $hCov$ & $mhCov$ & $T_{hCov}$ & $T_{mhCov}$ & $hCov$ & $mhCov$ & $T_{hCov}$ & $T_{mhCov}$  \\
			\hline \\[-1.8ex]
			\multirow{18}{*}{(i)}&$10$ & $5$ & $0.010$ & $0.017$ & $0.014$ & $0.020$ & $0.014$ & $0.016$ & $0.015$ & $0.020$ & $0.014$ & $0.014$ & $0.014$ & $0.017$ & $0.013$ \\
			&$10$ & $5$ & $0.050$ & $0.055$ & $0.055$ & $0.062$ & $0.061$ & $0.055$ & $0.060$ & $0.062$ & $0.061$ & $0.055$ & $0.050$ & $0.064$ & $0.050$ \\
			&$10$ & $5$ & $0.100$ & $0.105$ & $0.107$ & $0.110$ & $0.110$ & $0.103$ & $0.106$ & $0.109$ & $0.109$ & $0.102$ & $0.099$ & $0.105$ & $0.101$ \\
			&$10$ & $30$ & $0.010$ & $0.015$ & $0.015$ & $0.013$ & $0.011$ & $0.015$ & $0.016$ & $0.012$ & $0.012$ & $0.014$ & $0.014$ & $0.011$ & $0.011$ \\
			&$10$ & $30$ & $0.050$ & $0.054$ & $0.053$ & $0.050$ & $0.053$ & $0.053$ & $0.054$ & $0.050$ & $0.052$ & $0.052$ & $0.059$ & $0.050$ & $0.054$ \\
			&$10$ & $30$ & $0.100$ & $0.102$ & $0.104$ & $0.099$ & $0.102$ & $0.102$ & $0.105$ & $0.100$ & $0.103$ & $0.102$ & $0.107$ & $0.101$ & $0.105$ \\  \cline{2-16}
			&$30$ & $5$ & $0.010$ & $0.014$ & $0.016$ & $0.019$ & $0.018$ & $0.016$ & $0.016$ & $0.020$ & $0.017$ & $0.016$ & $0.015$ & $0.019$ & $0.015$ \\
			&$30$ & $5$ & $0.050$ & $0.052$ & $0.053$ & $0.062$ & $0.059$ & $0.052$ & $0.057$ & $0.061$ & $0.059$ & $0.054$ & $0.055$ & $0.061$ & $0.058$ \\
			&$30$ & $5$ & $0.100$ & $0.105$ & $0.104$ & $0.105$ & $0.107$ & $0.103$ & $0.107$ & $0.106$ & $0.106$ & $0.105$ & $0.104$ & $0.109$ & $0.104$ \\
			&$30$ & $30$ & $0.010$ & $0.014$ & $0.014$ & $0.011$ & $0.012$ & $0.014$ & $0.017$ & $0.010$ & $0.013$ & $0.014$ & $0.017$ & $0.011$ & $0.013$ \\
			&$30$ & $30$ & $0.050$ & $0.051$ & $0.053$ & $0.052$ & $0.051$ & $0.051$ & $0.056$ & $0.052$ & $0.053$ & $0.051$ & $0.058$ & $0.051$ & $0.052$ \\
			&$30$ & $30$ & $0.100$ & $0.097$ & $0.105$ & $0.096$ & $0.103$ & $0.097$ & $0.105$ & $0.095$ & $0.101$ & $0.099$ & $0.104$ & $0.100$ & $0.102$ \\ \cline{2-16}
			&$60$ & $5$ & $0.010$ & $0.013$ & $0.015$ & $0.018$ & $0.016$ & $0.014$ & $0.013$ & $0.019$ & $0.016$ & $0.014$ & $0.015$ & $0.017$ & $0.015$ \\
			&$60$ & $5$ & $0.050$ & $0.052$ & $0.055$ & $0.061$ & $0.057$ & $0.054$ & $0.061$ & $0.060$ & $0.064$ & $0.053$ & $0.057$ & $0.058$ & $0.058$ \\
			&$60$ & $5$ & $0.100$ & $0.103$ & $0.104$ & $0.109$ & $0.104$ & $0.107$ & $0.108$ & $0.110$ & $0.110$ & $0.102$ & $0.101$ & $0.103$ & $0.102$ \\
			&$60$ & $30$ & $0.010$ & $0.019$ & $0.017$ & $0.016$ & $0.012$ & $0.019$ & $0.015$ & $0.015$ & $0.013$ & $0.020$ & $0.016$ & $0.015$ & $0.014$ \\
			&$60$ & $30$ & $0.050$ & $0.060$ & $0.063$ & $0.057$ & $0.058$ & $0.060$ & $0.058$ & $0.057$ & $0.058$ & $0.061$ & $0.058$ & $0.058$ & $0.055$ \\
			&$60$ & $30$ & $0.100$ & $0.113$ & $0.112$ & $0.110$ & $0.107$ & $0.113$ & $0.109$ & $0.111$ & $0.105$ & $0.110$ & $0.111$ & $0.107$ & $0.107$ \\
			\hline \\[-1.8ex]
			\multirow{18}{*}{(ii)} &$10$ & $5$ & $0.010$ & $0.015$ & $0.015$ & $0.023$ & $0.023$ & $0.014$ & $0.016$ & $0.023$ & $0.019$ & $0.015$ & $0.017$ & $0.022$ & $0.021$ \\
			&$10$ & $5$ & $0.050$ & $0.051$ & $0.054$ & $0.064$ & $0.066$ & $0.053$ & $0.058$ & $0.064$ & $0.066$ & $0.054$ & $0.058$ & $0.066$ & $0.062$ \\
			&$10$ & $5$ & $0.100$ & $0.101$ & $0.105$ & $0.107$ & $0.111$ & $0.100$ & $0.109$ & $0.105$ & $0.113$ & $0.102$ & $0.110$ & $0.106$ & $0.109$ \\
			&$10$ & $30$ & $0.010$ & $0.014$ & $0.018$ & $0.013$ & $0.016$ & $0.014$ & $0.017$ & $0.014$ & $0.013$ & $0.017$ & $0.018$ & $0.017$ & $0.013$ \\
			&$10$ & $30$ & $0.050$ & $0.060$ & $0.061$ & $0.061$ & $0.061$ & $0.061$ & $0.056$ & $0.062$ & $0.056$ & $0.059$ & $0.060$ & $0.059$ & $0.056$ \\
			&$10$ & $30$ & $0.100$ & $0.105$ & $0.105$ & $0.110$ & $0.107$ & $0.105$ & $0.105$ & $0.109$ & $0.099$ & $0.106$ & $0.108$ & $0.111$ & $0.104$ \\ \cline{2-16}
			&$30$ & $5$ & $0.010$ & $0.012$ & $0.011$ & $0.022$ & $0.023$ & $0.012$ & $0.014$ & $0.021$ & $0.020$ & $0.013$ & $0.013$ & $0.019$ & $0.016$ \\
			&$30$ & $5$ & $0.050$ & $0.046$ & $0.048$ & $0.055$ & $0.056$ & $0.046$ & $0.052$ & $0.055$ & $0.059$ & $0.047$ & $0.053$ & $0.051$ & $0.059$ \\
			&$30$ & $5$ & $0.100$ & $0.094$ & $0.096$ & $0.094$ & $0.096$ & $0.096$ & $0.100$ & $0.097$ & $0.100$ & $0.093$ & $0.107$ & $0.097$ & $0.104$ \\
			&$30$ & $30$ & $0.010$ & $0.016$ & $0.016$ & $0.017$ & $0.015$ & $0.017$ & $0.015$ & $0.017$ & $0.011$ & $0.017$ & $0.015$ & $0.017$ & $0.012$ \\
			&$30$ & $30$ & $0.050$ & $0.061$ & $0.058$ & $0.060$ & $0.059$ & $0.061$ & $0.055$ & $0.060$ & $0.054$ & $0.058$ & $0.052$ & $0.060$ & $0.051$ \\
			&$30$ & $30$ & $0.100$ & $0.109$ & $0.105$ & $0.110$ & $0.107$ & $0.111$ & $0.101$ & $0.110$ & $0.098$ & $0.111$ & $0.102$ & $0.113$ & $0.097$ \\  \cline{2-16}
			&$60$ & $5$ & $0.010$ & $0.015$ & $0.013$ & $0.026$ & $0.022$ & $0.016$ & $0.014$ & $0.024$ & $0.020$ & $0.013$ & $0.015$ & $0.020$ & $0.018$ \\
			&$60$ & $5$ & $0.050$ & $0.055$ & $0.052$ & $0.062$ & $0.061$ & $0.055$ & $0.053$ & $0.061$ & $0.059$ & $0.055$ & $0.052$ & $0.061$ & $0.054$ \\
			&$60$ & $5$ & $0.100$ & $0.101$ & $0.100$ & $0.103$ & $0.100$ & $0.102$ & $0.100$ & $0.104$ & $0.099$ & $0.101$ & $0.097$ & $0.103$ & $0.099$ \\
			&$60$ & $30$ & $0.010$ & $0.013$ & $0.014$ & $0.013$ & $0.014$ & $0.013$ & $0.016$ & $0.014$ & $0.013$ & $0.014$ & $0.015$ & $0.013$ & $0.012$ \\
			&$60$ & $30$ & $0.050$ & $0.055$ & $0.051$ & $0.058$ & $0.051$ & $0.054$ & $0.054$ & $0.057$ & $0.053$ & $0.058$ & $0.053$ & $0.053$ & $0.052$ \\
			&$60$ & $30$ & $0.100$ & $0.105$ & $0.102$ & $0.105$ & $0.100$ & $0.106$ & $0.103$ & $0.105$ & $0.102$ & $0.107$ & $0.105$ & $0.107$ & $0.104$ \\
			\hline \\[-1.8ex]
			\multirow{18}{*}{(iii)}& $10$ & $5$ & $0.010$ & $0.012$ & $0.013$ & $0.025$ & $0.024$ & $0.012$ & $0.014$ & $0.024$ & $0.022$ & $0.016$ & $0.013$ & $0.025$ & $0.019$ \\
			& $10$ & $5$ & $0.050$ & $0.051$ & $0.051$ & $0.068$ & $0.069$ & $0.053$ & $0.051$ & $0.068$ & $0.062$ & $0.053$ & $0.049$ & $0.067$ & $0.056$ \\
			& $10$ & $5$ & $0.100$ & $0.100$ & $0.099$ & $0.107$ & $0.103$ & $0.100$ & $0.098$ & $0.105$ & $0.102$ & $0.100$ & $0.098$ & $0.104$ & $0.101$ \\
			& $10$ & $30$ & $0.010$ & $0.014$ & $0.015$ & $0.016$ & $0.014$ & $0.014$ & $0.015$ & $0.016$ & $0.013$ & $0.015$ & $0.015$ & $0.017$ & $0.013$ \\
			& $10$ & $30$ & $0.050$ & $0.055$ & $0.057$ & $0.061$ & $0.058$ & $0.053$ & $0.056$ & $0.061$ & $0.056$ & $0.057$ & $0.057$ & $0.064$ & $0.059$ \\
			&$10$ & $30$ & $0.100$ & $0.104$ & $0.105$ & $0.105$ & $0.107$ & $0.103$ & $0.105$ & $0.104$ & $0.107$ & $0.106$ & $0.110$ & $0.106$ & $0.112$ \\  \cline{2-16}
			&$30$ & $5$ & $0.010$ & $0.015$ & $0.014$ & $0.028$ & $0.029$ & $0.015$ & $0.014$ & $0.025$ & $0.024$ & $0.014$ & $0.014$ & $0.024$ & $0.019$ \\
			&$30$ & $5$ & $0.050$ & $0.052$ & $0.054$ & $0.060$ & $0.062$ & $0.051$ & $0.052$ & $0.062$ & $0.062$ & $0.048$ & $0.052$ & $0.058$ & $0.059$ \\
			&$30$ & $5$ & $0.100$ & $0.103$ & $0.103$ & $0.098$ & $0.099$ & $0.101$ & $0.101$ & $0.101$ & $0.098$ & $0.099$ & $0.099$ & $0.097$ & $0.098$ \\
			&$30$ & $30$ & $0.010$ & $0.017$ & $0.015$ & $0.019$ & $0.017$ & $0.016$ & $0.015$ & $0.019$ & $0.015$ & $0.013$ & $0.016$ & $0.018$ & $0.012$ \\
			&$30$ & $30$ & $0.050$ & $0.054$ & $0.055$ & $0.058$ & $0.058$ & $0.055$ & $0.055$ & $0.059$ & $0.057$ & $0.056$ & $0.057$ & $0.063$ & $0.056$ \\
			&$30$ & $30$ & $0.100$ & $0.102$ & $0.105$ & $0.105$ & $0.103$ & $0.101$ & $0.099$ & $0.103$ & $0.102$ & $0.104$ & $0.107$ & $0.105$ & $0.105$ \\  \cline{2-16}
			&$60$ & $5$ & $0.010$ & $0.012$ & $0.012$ & $0.029$ & $0.027$ & $0.014$ & $0.012$ & $0.028$ & $0.024$ & $0.016$ & $0.011$ & $0.023$ & $0.021$ \\
			&$60$ & $5$ & $0.050$ & $0.052$ & $0.052$ & $0.063$ & $0.064$ & $0.050$ & $0.048$ & $0.063$ & $0.059$ & $0.050$ & $0.052$ & $0.059$ & $0.061$ \\
			&$60$ & $5$ & $0.100$ & $0.100$ & $0.101$ & $0.098$ & $0.095$ & $0.098$ & $0.099$ & $0.097$ & $0.099$ & $0.099$ & $0.098$ & $0.100$ & $0.094$ \\
			&$60$ & $30$ & $0.010$ & $0.017$ & $0.015$ & $0.020$ & $0.019$ & $0.016$ & $0.017$ & $0.020$ & $0.017$ & $0.016$ & $0.015$ & $0.019$ & $0.014$ \\
			&$60$ & $30$ & $0.050$ & $0.052$ & $0.053$ & $0.058$ & $0.060$ & $0.055$ & $0.057$ & $0.061$ & $0.059$ & $0.057$ & $0.056$ & $0.062$ & $0.059$ \\
			&$60$ & $30$ & $0.100$ & $0.103$ & $0.106$ & $0.107$ & $0.103$ & $0.102$ & $0.106$ & $0.107$ & $0.105$ & $0.103$ & $0.102$ & $0.106$ & $0.101$ \\
			\hline \\[-1.8ex]
	\end{tabular} }
\end{table}

\begin{table}[!htbp] \centering
	\caption{Power comparision from Example \ref{exp2}}
	\label{tab2}
	\scalebox{0.6}{
		\begin{tabular}{@{\extracolsep{5pt}} cccccccccccccccc}
			\\[-1.8ex]\hline \hline \\[-1.8ex]
			&& & & \multicolumn{4}{c}{} & \multicolumn{4}{c}{Gaussian Kernel} & \multicolumn{4}{c}{Laplacian Kernel} \\ \cline{9-12} \cline{13-16}
			&$n$ & $p$ & $\alpha$ & $dCov$ & $mdCov$ & $T_{dCov}$ & $T_{mdCov}$ & $hCov$ & $mhCov$ & $T_{hCov}$ & $T_{mhCov}$ & $hCov$ & $mhCov$ & $T_{hCov}$ & $T_{mhCov}$  \\
			\hline \\[-1.8ex]
			\multirow{18}{*}{(i)}&$10$ & $5$ & $0.010$ & $0.635$ & $0.560$ & $0.691$ & $0.597$ & $0.629$ & $0.371$ & $0.685$ & $0.392$ & $0.516$ & $0.237$ & $0.585$ & $0.246$ \\
			&$10$ & $5$ & $0.050$ & $0.833$ & $0.774$ & $0.855$ & $0.792$ & $0.825$ & $0.598$ & $0.849$ & $0.610$ & $0.741$ & $0.450$ & $0.772$ & $0.458$ \\
			&$10$ & $5$ & $0.100$ & $0.910$ & $0.861$ & $0.914$ & $0.867$ & $0.906$ & $0.717$ & $0.912$ & $0.721$ & $0.839$ & $0.581$ & $0.851$ & $0.586$ \\
			&$10$ & $30$ & $0.010$ & $0.795$ & $0.654$ & $0.788$ & $0.634$ & $0.796$ & $0.410$ & $0.787$ & $0.379$ & $0.769$ & $0.247$ & $0.762$ & $0.219$ \\
			&$10$ & $30$ & $0.050$ & $0.936$ & $0.849$ & $0.937$ & $0.851$ & $0.935$ & $0.648$ & $0.937$ & $0.644$ & $0.921$ & $0.468$ & $0.924$ & $0.460$ \\
			&$10$ & $30$ & $0.100$ & $0.970$ & $0.914$ & $0.970$ & $0.916$ & $0.970$ & $0.767$ & $0.970$ & $0.768$ & $0.963$ & $0.604$ & $0.964$ & $0.603$ \\
			\cline{2-16}
			&$30$ & $5$ & $0.010$ & $1$ & $1$ & $1$ & $1$ & $1$ & $0.999$ & $1$ & $0.998$ & $1$ & $0.980$ & $1$ & $0.982$ \\
			&$30$ & $5$ & $0.050$ & $1$ & $1$ & $1$ & $1$ & $1$ & $1.000$ & $1$ & $1.000$ & $1$ & $0.996$ & $1$ & $0.996$ \\
			&$30$ & $5$ & $0.100$ & $1$ & $1$ & $1$ & $1$ & $1$ & $1.000$ & $1$ & $1.000$ & $1$ & $0.998$ & $1$ & $0.998$ \\
			&$30$ & $30$ & $0.010$ & $1$ & $1$ & $1$ & $1$ & $1$ & $1$ & $1$ & $1$ & $1$ & $0.996$ & $1$ & $0.996$ \\
			&$30$ & $30$ & $0.050$ & $1$ & $1$ & $1$ & $1$ & $1$ & $1$ & $1$ & $1$ & $1$ & $0.999$ & $1$ & $0.999$ \\
			&$30$ & $30$ & $0.100$ & $1$ & $1$ & $1$ & $1$ & $1$ & $1$ & $1$ & $1$ & $1$ & $1.000$ & $1$ & $1.000$ \\
			\cline{2-16}
			&$60$ & $5$ & $0.010$ & $1$ & $1$ & $1$ & $1$ & $1$ & $1$ & $1$ & $1$ & $1$ & $1$ & $1$ & $1$ \\
			&$60$ & $5$ & $0.050$ & $1$ & $1$ & $1$ & $1$ & $1$ & $1$ & $1$ & $1$ & $1$ & $1$ & $1$ & $1$ \\
			&$60$ & $5$ & $0.100$ & $1$ & $1$ & $1$ & $1$ & $1$ & $1$ & $1$ & $1$ & $1$ & $1$ & $1$ & $1$ \\
			&$60$ & $30$ & $0.010$ & $1$ & $1$ & $1$ & $1$ & $1$ & $1$ & $1$ & $1$ & $1$ & $1$ & $1$ & $1$ \\
			&$60$ & $30$ & $0.050$ & $1$ & $1$ & $1$ & $1$ & $1$ & $1$ & $1$ & $1$ & $1$ & $1$ & $1$ & $1$ \\
			&$60$ & $30$ & $0.100$ & $1$ & $1$ & $1$ & $1$ & $1$ & $1$ & $1$ & $1$ & $1$ & $1$ & $1$ & $1$ \\
			\hline \\[-1.8ex]
			\multirow{18}{*}{(ii)}& $10$ & $5$ & $0.010$ & $1.000$ & $0.999$ & $1.000$ & $0.999$ & $1.000$ & $0.986$ & $1.000$ & $0.989$ & $0.997$ & $0.935$ & $0.999$ & $0.942$ \\
			&$10$ & $5$ & $0.050$ & $1.000$ & $1.000$ & $1.000$ & $1.000$ & $1.000$ & $0.997$ & $1.000$ & $0.997$ & $0.999$ & $0.983$ & $1.000$ & $0.983$ \\
			&$10$ & $5$ & $0.100$ & $1.000$ & $1.000$ & $1.000$ & $1.000$ & $1.000$ & $0.999$ & $1.000$ & $0.999$ & $1.000$ & $0.992$ & $1.000$ & $0.993$ \\
			&$10$ & $30$ & $0.010$ & $1$ & $1$ & $1$ & $1.000$ & $1$ & $0.998$ & $1$ & $0.998$ & $1$ & $0.973$ & $1$ & $0.970$ \\
			&$10$ & $30$ & $0.050$ & $1$ & $1$ & $1$ & $1$ & $1$ & $1.000$ & $1$ & $1.000$ & $1$ & $0.995$ & $1$ & $0.995$ \\
			&$10$ & $30$ & $0.100$ & $1$ & $1$ & $1$ & $1$ & $1$ & $1.000$ & $1$ & $1.000$ & $1$ & $0.997$ & $1$ & $0.997$ \\ \cline{2-16}
			&$30$ & $5$ & $0.010$ & $1$ & $1$ & $1$ & $1$ & $1$ & $1$ & $1$ & $1$ & $1$ & $1$ & $1$ & $1$ \\
			&$30$ & $5$ & $0.050$ & $1$ & $1$ & $1$ & $1$ & $1$ & $1$ & $1$ & $1$ & $1$ & $1$ & $1$ & $1$ \\
			&$30$ & $5$ & $0.100$ & $1$ & $1$ & $1$ & $1$ & $1$ & $1$ & $1$ & $1$ & $1$ & $1$ & $1$ & $1$ \\
			&$30$ & $30$ & $0.010$ & $1$ & $1$ & $1$ & $1$ & $1$ & $1$ & $1$ & $1$ & $1$ & $1$ & $1$ & $1$ \\
			&$30$ & $30$ & $0.050$ & $1$ & $1$ & $1$ & $1$ & $1$ & $1$ & $1$ & $1$ & $1$ & $1$ & $1$ & $1$ \\
			&$30$ & $30$ & $0.100$ & $1$ & $1$ & $1$ & $1$ & $1$ & $1$ & $1$ & $1$ & $1$ & $1$ & $1$ & $1$ \\ \cline{2-16}
			&$60$ & $5$ & $0.010$ & $1$ & $1$ & $1$ & $1$ & $1$ & $1$ & $1$ & $1$ & $1$ & $1$ & $1$ & $1$ \\
			&$60$ & $5$ & $0.050$ & $1$ & $1$ & $1$ & $1$ & $1$ & $1$ & $1$ & $1$ & $1$ & $1$ & $1$ & $1$ \\
			&$60$ & $5$ & $0.100$ & $1$ & $1$ & $1$ & $1$ & $1$ & $1$ & $1$ & $1$ & $1$ & $1$ & $1$ & $1$ \\
			&$60$ & $30$ & $0.010$ & $1$ & $1$ & $1$ & $1$ & $1$ & $1$ & $1$ & $1$ & $1$ & $1$ & $1$ & $1$ \\
			&$60$ & $30$ & $0.050$ & $1$ & $1$ & $1$ & $1$ & $1$ & $1$ & $1$ & $1$ & $1$ & $1$ & $1$ & $1$ \\
			&$60$ & $30$ & $0.100$ & $1$ & $1$ & $1$ & $1$ & $1$ & $1$ & $1$ & $1$ & $1$ & $1$ & $1$ & $1$ \\
			\hline \\[-1.8ex] 	
			\multirow{18}{*}{(iii)}&$10$ & $5$ & $0.010$ & $0.635$ & $0.497$ & $0.685$ & $0.537$ & $0.633$ & $0.238$ & $0.681$ & $0.260$ & $0.525$ & $0.138$ & $0.584$ & $0.135$ \\
			&$10$ & $5$ & $0.050$ & $0.831$ & $0.728$ & $0.848$ & $0.748$ & $0.824$ & $0.460$ & $0.844$ & $0.477$ & $0.740$ & $0.311$ & $0.768$ & $0.323$ \\
			&$10$ & $5$ & $0.100$ & $0.903$ & $0.830$ & $0.911$ & $0.835$ & $0.899$ & $0.597$ & $0.905$ & $0.604$ & $0.835$ & $0.440$ & $0.844$ & $0.446$ \\
			&$10$ & $30$ & $0.010$ & $0.790$ & $0.583$ & $0.784$ & $0.555$ & $0.789$ & $0.273$ & $0.785$ & $0.247$ & $0.763$ & $0.147$ & $0.761$ & $0.122$ \\
			&$10$ & $30$ & $0.050$ & $0.928$ & $0.800$ & $0.930$ & $0.797$ & $0.928$ & $0.490$ & $0.930$ & $0.486$ & $0.915$ & $0.331$ & $0.919$ & $0.324$ \\
			&$10$ & $30$ & $0.100$ & $0.966$ & $0.888$ & $0.964$ & $0.889$ & $0.965$ & $0.628$ & $0.964$ & $0.626$ & $0.960$ & $0.460$ & $0.957$ & $0.453$ \\
			&$30$ & $5$ & $0.010$ & $1$ & $1.000$ & $1$ & $1$ & $1$ & $0.985$ & $1$ & $0.989$ & $1.000$ & $0.890$ & $1.000$ & $0.898$ \\
			&$30$ & $5$ & $0.050$ & $1$ & $1$ & $1$ & $1$ & $1$ & $0.996$ & $1$ & $0.997$ & $1$ & $0.971$ & $1$ & $0.971$ \\
			&$30$ & $5$ & $0.100$ & $1$ & $1$ & $1$ & $1$ & $1$ & $0.999$ & $1$ & $0.999$ & $1$ & $0.984$ & $1$ & $0.984$ \\
			&$30$ & $30$ & $0.010$ & $1$ & $1$ & $1$ & $1$ & $1$ & $0.998$ & $1$ & $0.999$ & $1$ & $0.950$ & $1$ & $0.948$ \\
			&$30$ & $30$ & $0.050$ & $1$ & $1$ & $1$ & $1$ & $1$ & $1.000$ & $1$ & $1.000$ & $1$ & $0.990$ & $1$ & $0.990$ \\
			&$30$ & $30$ & $0.100$ & $1$ & $1$ & $1$ & $1$ & $1$ & $1.000$ & $1$ & $1.000$ & $1$ & $0.997$ & $1$ & $0.997$ \\
			&$60$ & $5$ & $0.010$ & $1$ & $1$ & $1$ & $1$ & $1$ & $1$ & $1$ & $1$ & $1$ & $1.000$ & $1$ & $1$ \\
			&$60$ & $5$ & $0.050$ & $1$ & $1$ & $1$ & $1$ & $1$ & $1$ & $1$ & $1$ & $1$ & $1$ & $1$ & $1$ \\
			&$60$ & $5$ & $0.100$ & $1$ & $1$ & $1$ & $1$ & $1$ & $1$ & $1$ & $1$ & $1$ & $1$ & $1$ & $1$ \\
			&$60$ & $30$ & $0.010$ & $1$ & $1$ & $1$ & $1$ & $1$ & $1$ & $1$ & $1$ & $1$ & $1$ & $1$ & $1$ \\
			&$60$ & $30$ & $0.050$ & $1$ & $1$ & $1$ & $1$ & $1$ & $1$ & $1$ & $1$ & $1$ & $1$ & $1$ & $1$ \\
			&$60$ & $30$ & $0.100$ & $1$ & $1$ & $1$ & $1$ & $1$ & $1$ & $1$ & $1$ & $1$ & $1$ & $1$ & $1$ \\
			\hline \\[-1.8ex] 	
	\end{tabular} }
\end{table}

\begin{table}[!htbp] \centering
	\caption{Power comparision under $H_{A_s}$ from Example \ref{exp3}}
	\label{tab3}
	\scalebox{0.6}{
		\begin{tabular}{@{\extracolsep{5pt}} cccccccccccccccc}
			\\[-1.8ex]\hline \hline \\[-1.8ex]
			&& & & \multicolumn{4}{c}{} & \multicolumn{4}{c}{Gaussian Kernel} & \multicolumn{4}{c}{Laplacian Kernel} \\ \cline{9-12} \cline{13-16}
			&$n$ & $p$ & $\alpha$ & $dCov$ & $mdCov$ & $T_{dCov}$ & $T_{mdCov}$ & $hCov$ & $mhCov$ & $T_{hCov}$ & $T_{mhCov}$ & $hCov$ & $mhCov$ & $T_{hCov}$ & $T_{mhCov}$  \\
			\hline \\[-1.8ex]
			\multirow{18}{*}{(i)}& $10$ & $5$ & $0.010$ & $0.113$ & $0.285$ & $0.144$ & $0.321$ & $0.110$ & $0.493$ & $0.138$ & $0.516$ & $0.172$ & $0.801$ & $0.226$ & $0.813$ \\
			& $10$ & $5$ & $0.050$ & $0.231$ & $0.495$ & $0.254$ & $0.519$ & $0.236$ & $0.724$ & $0.256$ & $0.736$ & $0.356$ & $0.927$ & $0.398$ & $0.938$ \\
			& $10$ & $5$ & $0.100$ & $0.325$ & $0.618$ & $0.332$ & $0.628$ & $0.325$ & $0.828$ & $0.336$ & $0.834$ & $0.495$ & $0.968$ & $0.506$ & $0.969$ \\
			& $10$ & $30$ & $0.010$ & $0.032$ & $0.286$ & $0.028$ & $0.267$ & $0.032$ & $0.543$ & $0.030$ & $0.513$ & $0.044$ & $0.848$ & $0.041$ & $0.838$ \\
			& $10$ & $30$ & $0.050$ & $0.101$ & $0.526$ & $0.098$ & $0.523$ & $0.098$ & $0.769$ & $0.099$ & $0.763$ & $0.124$ & $0.945$ & $0.128$ & $0.947$ \\
			& $10$ & $30$ & $0.100$ & $0.158$ & $0.669$ & $0.162$ & $0.666$ & $0.160$ & $0.858$ & $0.160$ & $0.858$ & $0.203$ & $0.978$ & $0.205$ & $0.977$ \\ \cline{2-16}
			& $30$ & $5$ & $0.010$ & $0.440$ & $0.997$ & $0.499$ & $0.999$ & $0.518$ & $1$ & $0.583$ & $1$ & $0.924$ & $1$ & $0.956$ & $1$ \\
			& $30$ & $5$ & $0.050$ & $0.651$ & $1.000$ & $0.679$ & $1.000$ & $0.741$ & $1$ & $0.768$ & $1$ & $0.987$ & $1$ & $0.988$ & $1$ \\
			& $30$ & $5$ & $0.100$ & $0.766$ & $1.000$ & $0.773$ & $1$ & $0.836$ & $1$ & $0.845$ & $1$ & $0.994$ & $1$ & $0.995$ & $1$ \\
			& $30$ & $30$ & $0.010$ & $0.084$ & $1.000$ & $0.082$ & $1.000$ & $0.085$ & $1$ & $0.082$ & $1$ & $0.194$ & $1$ & $0.192$ & $1$ \\
			& $30$ & $30$ & $0.050$ & $0.190$ & $1$ & $0.187$ & $1$ & $0.192$ & $1$ & $0.192$ & $1$ & $0.365$ & $1$ & $0.365$ & $1$ \\
			& $30$ & $30$ & $0.100$ & $0.275$ & $1$ & $0.272$ & $1$ & $0.280$ & $1$ & $0.276$ & $1$ & $0.476$ & $1$ & $0.478$ & $1$ \\ \cline{2-16}
			& $60$ & $5$ & $0.010$ & $0.948$ & $1$ & $0.976$ & $1$ & $0.983$ & $1$ & $0.992$ & $1$ & $1$ & $1$ & $1$ & $1$ \\
			& $60$ & $5$ & $0.050$ & $0.994$ & $1$ & $0.996$ & $1$ & $0.998$ & $1$ & $0.999$ & $1$ & $1$ & $1$ & $1$ & $1$ \\
			& $60$ & $5$ & $0.100$ & $0.999$ & $1$ & $0.999$ & $1$ & $1.000$ & $1$ & $1.000$ & $1$ & $1$ & $1$ & $1$ & $1$ \\
			& $60$ & $30$ & $0.010$ & $0.185$ & $1$ & $0.173$ & $1$ & $0.194$ & $1$ & $0.183$ & $1$ & $0.587$ & $1$ & $0.587$ & $1$ \\
			& $60$ & $30$ & $0.050$ & $0.346$ & $1$ & $0.346$ & $1$ & $0.361$ & $1$ & $0.360$ & $1$ & $0.779$ & $1$ & $0.782$ & $1$ \\
			& $60$ & $30$ & $0.100$ & $0.462$ & $1$ & $0.459$ & $1$ & $0.475$ & $1$ & $0.473$ & $1$ & $0.861$ & $1$ & $0.864$ & $1$ \\
			\hline \\[-1.8ex]
			\multirow{18}{*}{(ii)} & $10$ & $5$ & $0.010$ & $0.167$ & $0.232$ & $0.237$ & $0.296$ & $0.192$ & $0.347$ & $0.263$ & $0.410$ & $0.279$ & $0.595$ & $0.391$ & $0.652$ \\
			& $10$ & $5$ & $0.050$ & $0.306$ & $0.386$ & $0.341$ & $0.421$ & $0.356$ & $0.570$ & $0.401$ & $0.606$ & $0.525$ & $0.806$ & $0.584$ & $0.832$ \\
			& $10$ & $5$ & $0.100$ & $0.401$ & $0.489$ & $0.409$ & $0.500$ & $0.479$ & $0.699$ & $0.487$ & $0.709$ & $0.674$ & $0.892$ & $0.689$ & $0.901$ \\
			& $10$ & $30$ & $0.010$ & $0.080$ & $0.202$ & $0.091$ & $0.210$ & $0.082$ & $0.376$ & $0.091$ & $0.366$ & $0.099$ & $0.646$ & $0.123$ & $0.634$ \\
			& $10$ & $30$ & $0.050$ & $0.178$ & $0.369$ & $0.191$ & $0.378$ & $0.179$ & $0.605$ & $0.192$ & $0.610$ & $0.229$ & $0.834$ & $0.252$ & $0.837$ \\
			& $10$ & $30$ & $0.100$ & $0.257$ & $0.492$ & $0.259$ & $0.492$ & $0.264$ & $0.728$ & $0.265$ & $0.730$ & $0.342$ & $0.906$ & $0.351$ & $0.909$ \\ \cline{2-16}
			& $30$ & $5$ & $0.010$ & $0.623$ & $0.847$ & $0.781$ & $0.950$ & $0.895$ & $0.999$ & $0.957$ & $1$ & $0.995$ & $1$ & $0.999$ & $1$ \\
			& $30$ & $5$ & $0.050$ & $0.872$ & $0.984$ & $0.902$ & $0.990$ & $0.982$ & $1$ & $0.990$ & $1$ & $1.000$ & $1$ & $1$ & $1$ \\
			& $30$ & $5$ & $0.100$ & $0.940$ & $0.996$ & $0.945$ & $0.995$ & $0.994$ & $1$ & $0.994$ & $1$ & $1$ & $1$ & $1$ & $1$ \\
			& $30$ & $30$ & $0.010$ & $0.251$ & $0.929$ & $0.277$ & $0.944$ & $0.307$ & $1$ & $0.336$ & $1$ & $0.629$ & $1$ & $0.686$ & $1$ \\
			& $30$ & $30$ & $0.050$ & $0.419$ & $0.982$ & $0.434$ & $0.985$ & $0.499$ & $1$ & $0.517$ & $1$ & $0.830$ & $1$ & $0.849$ & $1$ \\
			& $30$ & $30$ & $0.100$ & $0.532$ & $0.995$ & $0.532$ & $0.995$ & $0.613$ & $1$ & $0.622$ & $1$ & $0.905$ & $1$ & $0.909$ & $1$ \\ \cline{2-16}
			& $60$ & $5$ & $0.010$ & $0.999$ & $1$ & $1$ & $1$ & $1$ & $1$ & $1$ & $1$ & $1$ & $1$ & $1$ & $1$ \\
			& $60$ & $5$ & $0.050$ & $1$ & $1$ & $1$ & $1$ & $1$ & $1$ & $1$ & $1$ & $1$ & $1$ & $1$ & $1$ \\
			& $60$ & $5$ & $0.100$ & $1$ & $1$ & $1$ & $1$ & $1$ & $1$ & $1$ & $1$ & $1$ & $1$ & $1$ & $1$ \\
			& $60$ & $30$ & $0.010$ & $0.643$ & $1$ & $0.684$ & $1$ & $0.790$ & $1$ & $0.833$ & $1$ & $0.996$ & $1$ & $0.999$ & $1$ \\
			& $60$ & $30$ & $0.050$ & $0.824$ & $1$ & $0.836$ & $1$ & $0.918$ & $1$ & $0.930$ & $1$ & $1.000$ & $1$ & $1.000$ & $1$ \\
			& $60$ & $30$ & $0.100$ & $0.894$ & $1$ & $0.896$ & $1$ & $0.955$ & $1$ & $0.958$ & $1$ & $1$ & $1$ & $1$ & $1$ \\
			\hline \\[-1.8ex] 	
			\multirow{18}{*}{(iii)} & $10$ & $5$ & $0.010$ & $0.043$ & $0.233$ & $0.060$ & $0.257$ & $0.042$ & $0.434$ & $0.053$ & $0.447$ & $0.076$ & $0.768$ & $0.098$ & $0.785$ \\
			& $10$ & $5$ & $0.050$ & $0.121$ & $0.466$ & $0.141$ & $0.490$ & $0.119$ & $0.680$ & $0.137$ & $0.698$ & $0.191$ & $0.924$ & $0.214$ & $0.927$ \\
			& $10$ & $5$ & $0.100$ & $0.201$ & $0.616$ & $0.212$ & $0.624$ & $0.203$ & $0.808$ & $0.210$ & $0.810$ & $0.291$ & $0.963$ & $0.298$ & $0.964$ \\
			& $10$ & $30$ & $0.010$ & $0.017$ & $0.260$ & $0.013$ & $0.242$ & $0.017$ & $0.482$ & $0.012$ & $0.445$ & $0.021$ & $0.830$ & $0.017$ & $0.811$ \\
			& $10$ & $30$ & $0.050$ & $0.062$ & $0.488$ & $0.062$ & $0.487$ & $0.063$ & $0.729$ & $0.062$ & $0.727$ & $0.071$ & $0.941$ & $0.070$ & $0.940$ \\
			& $10$ & $30$ & $0.100$ & $0.120$ & $0.632$ & $0.116$ & $0.630$ & $0.118$ & $0.837$ & $0.115$ & $0.836$ & $0.131$ & $0.972$ & $0.130$ & $0.975$ \\ \cline{2-16}
			& $30$ & $5$ & $0.010$ & $0.146$ & $0.999$ & $0.191$ & $1$ & $0.153$ & $1$ & $0.187$ & $1$ & $0.464$ & $1$ & $0.529$ & $1$ \\
			& $30$ & $5$ & $0.050$ & $0.346$ & $1$ & $0.375$ & $1$ & $0.347$ & $1$ & $0.380$ & $1$ & $0.723$ & $1$ & $0.747$ & $1$ \\
			& $30$ & $5$ & $0.100$ & $0.484$ & $1$ & $0.497$ & $1$ & $0.496$ & $1$ & $0.501$ & $1$ & $0.835$ & $1$ & $0.840$ & $1$ \\
			& $30$ & $30$ & $0.010$ & $0.024$ & $1.000$ & $0.022$ & $1.000$ & $0.026$ & $1$ & $0.022$ & $1$ & $0.038$ & $1$ & $0.037$ & $1$ \\
			& $30$ & $30$ & $0.050$ & $0.088$ & $1$ & $0.085$ & $1$ & $0.086$ & $1$ & $0.085$ & $1$ & $0.117$ & $1$ & $0.115$ & $1$ \\
			& $30$ & $30$ & $0.100$ & $0.149$ & $1$ & $0.147$ & $1$ & $0.148$ & $1$ & $0.144$ & $1$ & $0.195$ & $1$ & $0.193$ & $1$ \\  \cline{2-16}
			& $60$ & $5$ & $0.010$ & $0.547$ & $1$ & $0.630$ & $1$ & $0.566$ & $1$ & $0.642$ & $1$ & $0.978$ & $1$ & $0.988$ & $1$ \\
			& $60$ & $5$ & $0.050$ & $0.802$ & $1$ & $0.835$ & $1$ & $0.808$ & $1$ & $0.836$ & $1$ & $0.997$ & $1$ & $0.998$ & $1$ \\
			& $60$ & $5$ & $0.100$ & $0.907$ & $1$ & $0.911$ & $1$ & $0.905$ & $1$ & $0.913$ & $1$ & $0.999$ & $1$ & $0.999$ & $1$ \\
			& $60$ & $30$ & $0.010$ & $0.038$ & $1$ & $0.030$ & $1$ & $0.038$ & $1$ & $0.029$ & $1$ & $0.089$ & $1$ & $0.080$ & $1$ \\
			& $60$ & $30$ & $0.050$ & $0.122$ & $1$ & $0.117$ & $1$ & $0.119$ & $1$ & $0.119$ & $1$ & $0.217$ & $1$ & $0.214$ & $1$ \\
			& $60$ & $30$ & $0.100$ & $0.198$ & $1$ & $0.196$ & $1$ & $0.199$ & $1$ & $0.197$ & $1$ & $0.326$ & $1$ & $0.325$ & $1$ \\
			\hline \\[-1.8ex]
	\end{tabular} }
\end{table}

\begin{table}[!htbp] \centering
	\caption{Power comparision under $H_{A_s}$ from Example \ref{exp4}}
	\label{tab4}
	\scalebox{0.6}{
		\begin{tabular}{@{\extracolsep{5pt}} cccccccccccccccc}
			\\[-1.8ex]\hline \hline \\[-1.8ex]
			&& & & \multicolumn{4}{c}{} & \multicolumn{4}{c}{Gaussian Kernel} & \multicolumn{4}{c}{Laplacian Kernel} \\ \cline{9-12} \cline{13-16}
			&$n$ & $p$ & $\alpha$ & $dCov$ & $mdCov$ & $T_{dCov}$ & $T_{mdCov}$ & $hCov$ & $mhCov$ & $T_{hCov}$ & $T_{mhCov}$ & $hCov$ & $mhCov$ & $T_{hCov}$ & $T_{mhCov}$  \\
			\hline \\[-1.8ex]
			\multirow{18}{*}{(i)} & $10$ & $5$ & $0.010$ & $0.044$ & $0.196$ & $0.055$ & $0.218$ & $0.042$ & $0.348$ & $0.052$ & $0.367$ & $0.074$ & $0.672$ & $0.098$ & $0.685$ \\
			& $10$ & $5$ & $0.050$ & $0.120$ & $0.390$ & $0.136$ & $0.416$ & $0.114$ & $0.582$ & $0.129$ & $0.604$ & $0.183$ & $0.859$ & $0.209$ & $0.870$ \\
			& $10$ & $5$ & $0.100$ & $0.201$ & $0.542$ & $0.209$ & $0.546$ & $0.191$ & $0.722$ & $0.197$ & $0.731$ & $0.292$ & $0.927$ & $0.304$ & $0.931$ \\
			& $10$ & $30$ & $0.010$ & $0.018$ & $0.212$ & $0.014$ & $0.194$ & $0.017$ & $0.387$ & $0.014$ & $0.362$ & $0.022$ & $0.722$ & $0.017$ & $0.706$ \\
			& $10$ & $30$ & $0.050$ & $0.066$ & $0.434$ & $0.064$ & $0.428$ & $0.066$ & $0.627$ & $0.064$ & $0.625$ & $0.075$ & $0.892$ & $0.077$ & $0.891$ \\
			& $10$ & $30$ & $0.100$ & $0.123$ & $0.571$ & $0.121$ & $0.568$ & $0.123$ & $0.749$ & $0.119$ & $0.750$ & $0.135$ & $0.944$ & $0.132$ & $0.946$ \\ \cline{2-16}
			& $30$ & $5$ & $0.010$ & $0.158$ & $0.988$ & $0.197$ & $0.996$ & $0.136$ & $1$ & $0.163$ & $1$ & $0.486$ & $1$ & $0.555$ & $1$ \\
			& $30$ & $5$ & $0.050$ & $0.341$ & $1.000$ & $0.369$ & $1$ & $0.303$ & $1$ & $0.328$ & $1$ & $0.725$ & $1$ & $0.756$ & $1$ \\
			& $30$ & $5$ & $0.100$ & $0.483$ & $1$ & $0.488$ & $1$ & $0.433$ & $1$ & $0.444$ & $1$ & $0.838$ & $1$ & $0.846$ & $1$ \\
			& $30$ & $30$ & $0.010$ & $0.026$ & $0.996$ & $0.023$ & $0.996$ & $0.027$ & $1.000$ & $0.022$ & $1.000$ & $0.043$ & $1$ & $0.038$ & $1$ \\
			& $30$ & $30$ & $0.050$ & $0.089$ & $1.000$ & $0.084$ & $0.999$ & $0.088$ & $1$ & $0.083$ & $1$ & $0.123$ & $1$ & $0.125$ & $1$ \\
			& $30$ & $30$ & $0.100$ & $0.153$ & $1.000$ & $0.152$ & $1.000$ & $0.151$ & $1$ & $0.152$ & $1$ & $0.209$ & $1$ & $0.204$ & $1$ \\ \cline{2-16}
			& $60$ & $5$ & $0.010$ & $0.559$ & $1$ & $0.637$ & $1$ & $0.461$ & $1$ & $0.539$ & $1$ & $0.989$ & $1$ & $0.996$ & $1$ \\
			& $60$ & $5$ & $0.050$ & $0.816$ & $1$ & $0.847$ & $1$ & $0.738$ & $1$ & $0.774$ & $1$ & $1.000$ & $1$ & $1$ & $1$ \\
			& $60$ & $5$ & $0.100$ & $0.916$ & $1$ & $0.925$ & $1$ & $0.861$ & $1$ & $0.870$ & $1$ & $1$ & $1$ & $1$ & $1$ \\
			& $60$ & $30$ & $0.010$ & $0.037$ & $1$ & $0.032$ & $1$ & $0.036$ & $1$ & $0.031$ & $1$ & $0.091$ & $1$ & $0.085$ & $1$ \\
			& $60$ & $30$ & $0.050$ & $0.125$ & $1$ & $0.119$ & $1$ & $0.122$ & $1$ & $0.115$ & $1$ & $0.231$ & $1$ & $0.228$ & $1$ \\
			& $60$ & $30$ & $0.100$ & $0.208$ & $1$ & $0.207$ & $1$ & $0.204$ & $1$ & $0.202$ & $1$ & $0.350$ & $1$ & $0.346$ & $1$ \\
			\hline \\[-1.8ex]
			\multirow{18}{*}{(ii)} & $10$ & $5$ & $0.010$ & $0.044$ & $0.217$ & $0.059$ & $0.242$ & $0.040$ & $0.393$ & $0.055$ & $0.413$ & $0.077$ & $0.713$ & $0.106$ & $0.732$ \\
			&$10$ & $5$ & $0.050$ & $0.124$ & $0.432$ & $0.141$ & $0.453$ & $0.117$ & $0.637$ & $0.131$ & $0.655$ & $0.202$ & $0.886$ & $0.224$ & $0.895$ \\
			&$10$ & $5$ & $0.100$ & $0.210$ & $0.577$ & $0.213$ & $0.583$ & $0.196$ & $0.771$ & $0.204$ & $0.775$ & $0.304$ & $0.942$ & $0.318$ & $0.942$ \\
			&$10$ & $30$ & $0.010$ & $0.020$ & $0.247$ & $0.013$ & $0.224$ & $0.019$ & $0.439$ & $0.013$ & $0.409$ & $0.022$ & $0.774$ & $0.018$ & $0.763$ \\
			&$10$ & $30$ & $0.050$ & $0.064$ & $0.474$ & $0.064$ & $0.474$ & $0.063$ & $0.677$ & $0.063$ & $0.676$ & $0.075$ & $0.913$ & $0.076$ & $0.913$ \\
			&$10$ & $30$ & $0.100$ & $0.126$ & $0.606$ & $0.125$ & $0.604$ & $0.126$ & $0.795$ & $0.126$ & $0.790$ & $0.141$ & $0.956$ & $0.138$ & $0.955$ \\ \cline{2-16}
			&$30$ & $5$ & $0.010$ & $0.178$ & $0.995$ & $0.221$ & $0.999$ & $0.148$ & $1$ & $0.186$ & $1$ & $0.544$ & $1$ & $0.608$ & $1$ \\
			&$30$ & $5$ & $0.050$ & $0.376$ & $1$ & $0.409$ & $1$ & $0.333$ & $1$ & $0.358$ & $1$ & $0.775$ & $1$ & $0.797$ & $1$ \\
			&$30$ & $5$ & $0.100$ & $0.518$ & $1$ & $0.526$ & $1$ & $0.468$ & $1$ & $0.478$ & $1$ & $0.871$ & $1$ & $0.880$ & $1$ \\
			&$30$ & $30$ & $0.010$ & $0.027$ & $0.998$ & $0.023$ & $0.998$ & $0.026$ & $1.000$ & $0.022$ & $1$ & $0.043$ & $1$ & $0.038$ & $1$ \\
			&$30$ & $30$ & $0.050$ & $0.088$ & $1.000$ & $0.087$ & $1.000$ & $0.088$ & $1$ & $0.086$ & $1$ & $0.128$ & $1$ & $0.128$ & $1$ \\
			&$30$ & $30$ & $0.100$ & $0.155$ & $1.000$ & $0.152$ & $1.000$ & $0.154$ & $1$ & $0.152$ & $1$ & $0.218$ & $1$ & $0.213$ & $1$ \\ \cline{2-16}
			&$60$ & $5$ & $0.010$ & $0.632$ & $1$ & $0.709$ & $1$ & $0.526$ & $1$ & $0.609$ & $1$ & $0.995$ & $1$ & $0.999$ & $1$ \\
			&$60$ & $5$ & $0.050$ & $0.870$ & $1$ & $0.895$ & $1$ & $0.792$ & $1$ & $0.826$ & $1$ & $1$ & $1$ & $1$ & $1$ \\
			&$60$ & $5$ & $0.100$ & $0.946$ & $1$ & $0.952$ & $1$ & $0.904$ & $1$ & $0.911$ & $1$ & $1$ & $1$ & $1$ & $1$ \\
			&$60$ & $30$ & $0.010$ & $0.044$ & $1$ & $0.037$ & $1$ & $0.043$ & $1$ & $0.036$ & $1$ & $0.105$ & $1$ & $0.096$ & $1$ \\
			&$60$ & $30$ & $0.050$ & $0.126$ & $1$ & $0.125$ & $1$ & $0.123$ & $1$ & $0.121$ & $1$ & $0.251$ & $1$ & $0.244$ & $1$ \\
			&$60$ & $30$ & $0.100$ & $0.213$ & $1$ & $0.211$ & $1$ & $0.211$ & $1$ & $0.206$ & $1$ & $0.368$ & $1$ & $0.366$ & $1$ \\
			\hline \\[-1.8ex] 	
			\multirow{18}{*}{(iii)} & $10$ & $5$ & $0.010$ & $0.019$ & $0.024$ & $0.023$ & $0.028$ & $0.017$ & $0.033$ & $0.022$ & $0.040$ & $0.023$ & $0.090$ & $0.029$ & $0.095$ \\
			& $10$ & $5$ & $0.050$ & $0.058$ & $0.079$ & $0.068$ & $0.089$ & $0.057$ & $0.111$ & $0.067$ & $0.115$ & $0.068$ & $0.232$ & $0.081$ & $0.242$ \\
			& $10$ & $5$ & $0.100$ & $0.113$ & $0.148$ & $0.117$ & $0.151$ & $0.114$ & $0.194$ & $0.118$ & $0.196$ & $0.124$ & $0.351$ & $0.129$ & $0.355$ \\
			& $10$ & $30$ & $0.010$ & $0.016$ & $0.026$ & $0.012$ & $0.020$ & $0.016$ & $0.037$ & $0.012$ & $0.030$ & $0.017$ & $0.089$ & $0.013$ & $0.076$ \\
			& $10$ & $30$ & $0.050$ & $0.059$ & $0.086$ & $0.057$ & $0.083$ & $0.060$ & $0.112$ & $0.058$ & $0.105$ & $0.061$ & $0.233$ & $0.060$ & $0.225$ \\
			& $10$ & $30$ & $0.100$ & $0.111$ & $0.156$ & $0.108$ & $0.153$ & $0.112$ & $0.199$ & $0.108$ & $0.193$ & $0.112$ & $0.357$ & $0.109$ & $0.346$ \\ \cline{2-16}
			& $30$ & $5$ & $0.010$ & $0.019$ & $0.051$ & $0.021$ & $0.068$ & $0.017$ & $0.141$ & $0.021$ & $0.170$ & $0.026$ & $0.673$ & $0.032$ & $0.724$ \\
			& $30$ & $5$ & $0.050$ & $0.061$ & $0.166$ & $0.070$ & $0.188$ & $0.058$ & $0.339$ & $0.066$ & $0.360$ & $0.083$ & $0.889$ & $0.091$ & $0.903$ \\
			& $30$ & $5$ & $0.100$ & $0.117$ & $0.283$ & $0.117$ & $0.288$ & $0.117$ & $0.488$ & $0.116$ & $0.497$ & $0.153$ & $0.953$ & $0.153$ & $0.955$ \\
			& $30$ & $30$ & $0.010$ & $0.017$ & $0.074$ & $0.012$ & $0.065$ & $0.017$ & $0.182$ & $0.012$ & $0.165$ & $0.017$ & $0.754$ & $0.012$ & $0.742$ \\
			& $30$ & $30$ & $0.050$ & $0.061$ & $0.202$ & $0.058$ & $0.198$ & $0.061$ & $0.378$ & $0.059$ & $0.373$ & $0.063$ & $0.913$ & $0.061$ & $0.913$ \\
			& $30$ & $30$ & $0.100$ & $0.112$ & $0.309$ & $0.110$ & $0.307$ & $0.113$ & $0.518$ & $0.110$ & $0.517$ & $0.117$ & $0.960$ & $0.114$ & $0.959$ \\ \cline{2-16}
			& $60$ & $5$ & $0.010$ & $0.019$ & $0.174$ & $0.024$ & $0.219$ & $0.017$ & $0.580$ & $0.022$ & $0.666$ & $0.034$ & $1.000$ & $0.041$ & $1$ \\
			& $60$ & $5$ & $0.050$ & $0.066$ & $0.421$ & $0.073$ & $0.458$ & $0.061$ & $0.853$ & $0.069$ & $0.883$ & $0.108$ & $1$ & $0.119$ & $1$ \\
			& $60$ & $5$ & $0.100$ & $0.123$ & $0.600$ & $0.128$ & $0.612$ & $0.119$ & $0.941$ & $0.122$ & $0.949$ & $0.179$ & $1$ & $0.183$ & $1$ \\
			& $60$ & $30$ & $0.010$ & $0.013$ & $0.251$ & $0.009$ & $0.233$ & $0.013$ & $0.680$ & $0.010$ & $0.665$ & $0.014$ & $1.000$ & $0.010$ & $1$ \\
			& $60$ & $30$ & $0.050$ & $0.053$ & $0.485$ & $0.051$ & $0.484$ & $0.052$ & $0.869$ & $0.050$ & $0.871$ & $0.056$ & $1$ & $0.055$ & $1$ \\
			& $60$ & $30$ & $0.100$ & $0.105$ & $0.620$ & $0.101$ & $0.619$ & $0.106$ & $0.930$ & $0.101$ & $0.929$ & $0.107$ & $1$ & $0.106$ & $1$ \\
			\hline \\[-1.8ex]
	\end{tabular} }
\end{table}

\section{Technical Details}

\subsection{Proof of Proposition \ref{prop:taylor}}
%\begin{proof}
%	Denote $f^{(2)}(t) = \frac{1}{4} (1+t)^{-\frac{3}{2}}$. By the definition of Lagrange's form of the remainder term from Taylor expansion, we have
%	\begin{align*}
%	\left| R_X(X,X') \right|=& \left| \int_{0}^{L_X(X,X')} f^{(2)}(t) (L_X(X,X')-t)dt \right| \\
%	= & \mathbb{I}_{\{L_X(X,X')> 0 \}} \left| \int_{0}^{L_X(X,X')} f^{(2)}(t) (L_X(X,X')-t)dt \right|    \\
%	& +  \mathbb{I}_{\{L_X(X,X')\leq 0 \} } \left|\int_{0}^{L_X(X,X')} f^{(2)}(t) (L_X(X,X')-t)dt \right| \\
%	\le & \mathbb{I}_{\{L_X(X,X') > 0\}} \int_{0}^{L_X(X,X')} f^{(2)}(t)\left| L_X(X,X')-t \right| dt    \\
%	& + \mathbb{I}_{\{L_X(X,X') \leq 0 \}} \int_{L_X(X,X')}^{0} f^{(2)}(t)\left|  L_X(X,X')-t  \right|  dt   \\
%	\le &  |L_X(X,X')|  \int_{0}^{L_X(X,X')} f^{(2)}(t) 1_{\{L_X(X,X') > 0\} } dt  \\
%	& + |L_X(X,X')| \int_{L_X(X,X')}^{0} f^{(2)}(t) 1_{\{L_X(X,X') \le 0 \}} dt   \\
%	=  &   \frac{|L_X(X,X')|}{2} \left| 1- \frac{1}{\sqrt{1+L_X(X,X')} } \right| \\
%	=  &  \frac{|L_X(X,X')|}{2} \frac{|L_X(X,X')|}{1+L_X(X,X')+\sqrt{1+L_X(X,X')}}\\
%	\le  & \frac{L_X(X,X')^2 }{ 2[1+ L_X(X,X')] }.
%	\end{align*}
%	Thus $R_X(X,X')=O_p( L_X(X,X')^2)$ as long as $L_X(X,X') =o_p(1)$. Similar proof applies to $R_Y(Y,Y')$.
%\end{proof}

\begin{proof}
Denote $f^{(2)}(t) = - \frac{1}{4} (1+t)^{-\frac{3}{2}}$. The remainder term can be written as
\begin{align*}
R_{X}(X_{s},X_{t}) = \int_{0}^1 \int_{0}^{1} v f^{(2)} \left(  uv L_{X}(X_s, X_t) \right) dudv  \times \left( L_{X}(X_s, X_t) \right)^2.
\end{align*}
Set $ \varphi (x) = \int_{0}^1 \int_{0}^{1} v f^{(2)} \left( uv x \right) dudv $. Then $\varphi (x)$ is continuous at $0$. Next, by the continuous mapping theorem, we have
\begin{align*}
\int_{0}^1 \int_{0}^{1} v f^{(2)} \left(  uv L_{X}(X_s, X_t) \right) dudv \overset{p}{\rightarrow}
\int_{0}^1 \int_{0}^{1} v f^{(2)} \left( 0 \right) dudv.
\end{align*}
So, $ R_{X}(X_{s},X_{t})  \asymp_p  \left( L_{X}(X_s, X_t)  \right)^2 $. Similar arguments hold for $R_{Y}(Y_{s},Y_{t})$.
\end{proof}

\subsection{Proof of Remark \ref{remark:mainthm}}
\label{App:proofRemark}
\begin{proof}
%	(i) We have
%	\begin{align*}
%	 & \sqrt{var[ L(X_{k}, X_{l}) ]} \\
%	\asymp &  \sqrt{ \frac{ var[ \sum_{i=1}^{p} ( x_{ki} - x_{li} )^2 ] }{p^{2}} } \\
%	  = & \sqrt{\frac{ var[ (x_{k1} - x_{l1})^2 ] + 2 \sum_{i=2}^{p} (1 - \frac{i-1}{p}) cov[ (x_{k1} - x_{l1})^2, (x_{ki} - x_{li})^2 ] }{p}}
%	\end{align*}
%	Here, we can show that
%	\begin{align*}
%	 & var[ (x_{k1} - x_{l1})^2 ] + 2 \sum_{i=2}^{p} \left| cov[ (x_{k1} - x_{l1})^2, (x_{ki} - x_{li})^2 ] \right| \\
%	\leq & 2\left[  var ( x_{k1}^2 ) + 2 \sum_{i=2}^{\infty} \left| cov( x_{k1}^2, x_{ki}^2 ) \right|\right] + 4 \left[  var^2 ( x_{k1} ) + 2 \sum_{i=2}^{\infty} \left| cov^2( x_{k1}, x_{ki} ) \right|\right] \\
%	= & O(1).
%	\end{align*}
%	Therefore, we have  $a_{p} = 1/ \sqrt{p} \text{ and } b_q = 1/ \sqrt{q}$.
	
	(i) Notice that
	\begin{align*}
	& \sqrt{var[ L(X_{k}, X_{l}) ]} \\
	 \asymp & \sqrt{ \frac{ var[ (X_{k} - X_{l})^{T}( X_{k} -X_{l} ) ] }{p^{2}} } \\
	 = & \sqrt{ \frac{ var \{[ \mathbf{A} (U_{k} - U_{l}) + (\Phi_{k} - \Phi_{l}) ]^{T}[ \mathbf{A} (U_{k} - U_{l}) + (\Phi_{k} - \Phi_{l}) ] \}  }{p^{2}} }.
	\end{align*}
	Denote $\mathbf{C} = (c_{ij}) = \mathbf{A}^{T} \mathbf{A}$. We obtain that
	\begin{align*}
	& var \{[ \mathbf{A} (U_{k} - U_{l}) + (\Phi_{k} - \Phi_{l}) ]^{T}[ \mathbf{A} (U_{k} - U_{l}) + (\Phi_{k} - \Phi_{l}) ] \} \\
	= & var \big[ (U_{k} - U_{l})^{T} \mathbf{A}^{T} \mathbf{A} (U_{k} - U_{l}) \\
	& \quad \quad +  (\Phi_{k} - \Phi_{l})^{T}(\Phi_{k} - \Phi_{l}) +2 (U_{k} - U_{l})^{T}\mathbf{A}^{T}(\Phi_{k} - \Phi_{l}) \big] \\
	= & var \bigg[ \sum\limits_{i=1}^{s_{1}} \sum\limits_{j=1}^{s_{1}} c_{ij} (u_{ki} - u_{li})(u_{kj} - u_{lj}) + \sum\limits_{i=1}^{p}( \phi_{ki} - \phi_{li} )^{2} \\
	& \quad\quad + 2 \sum\limits_{i=1}^{s_{1}} \sum\limits_{j=1}^{p} a_{ji} (u_{ki} - u_{li})( \phi_{kj} - \phi_{lj} )  \bigg] \\
	\leq & 2 \sum\limits_{i=1}^{s_{1}} \sum\limits_{j=1}^{s_{1}} c_{ij}^{2} var[  (u_{ki} - u_{li})(u_{kj} - u_{lj})] + \sum\limits_{i=1}^{p} var [ ( \phi_{ki} - \phi_{li} )^{2}] \\
	& \quad\quad + 4 \sum\limits_{i=1}^{s_{1}} \sum\limits_{j=1}^{p} a_{ji}^{2} var[ (u_{ki} - u_{li})( \phi_{kj} - \phi_{lj} )].
	\end{align*}
	Since the 4th moment is bounded uniformly for each $u_{ki}$ and $\phi_{ki}$, $ var[  (u_{ki} - u_{li})(u_{kj} - u_{lj})]$, $var [ ( \phi_{ki} - \phi_{li} )^{2}]$ and $var[ (u_{ki} - u_{li})( \phi_{kj} - \phi_{lj} )]$ are all uniformly bounded by a constant. As $\|A\|_F^2=O(p^{1/2})$,
we have $\|A^TA\|_F^2=O(p)$ by the Cauchy-Schwarz inequality. It follows that
	\begin{align*}
	var \{[ \mathbf{A} (U_{k} - U_{l}) + (\Phi_{k} - \Phi_{l}) ]^{T}[ \mathbf{A} (U_{k} - U_{l}) + (\Phi_{k} - \Phi_{l}) ] \} = O(p).
	\end{align*}
	Thus, we have $a_{p} = 1/ \sqrt{p} \text{ and } b_q = 1/ \sqrt{q}.$
\end{proof}

%\subsection{Proof of Remark \ref{remark:mainthm}}
%First of all, the following lemma is useful.
%\begin{lemma}
%	Suppose $X_n \overset{d}{\rightarrow} Z$, where $Z$ is standard normal, then $X_n \asymp_p 1$.
%\end{lemma}
%\begin{proof}
%	It is well known that $X_n = O_p(1)$, here we show that $ 1/X_n = O_p (1) $. Since $X_n \overset{d}{\rightarrow} Z$, there exists $N$ such that for all $n > M$, we have
%	\begin{align*}
%	& P(X_n < 1/M) < P(Z< 1/M) + \epsilon/3, \\
%	&  P(X_n \leq - 1/M) > P(Z \leq -1/M) - \epsilon/3. \\
%	\end{align*}
%	Then, we have
%	\begin{align*}
%	& P \left( \left| \frac{1}{X_n} \right| >M \right) \\
%	= & P \left( X_n < \frac{1}{M} \right) - P \left( X_n \leq - \frac{1}{M} \right) \\
%	\leq & P \left( Z < \frac{1}{M} \right) - P \left( Z \leq - \frac{1}{M} \right) + \frac{2}{3} \epsilon \\
%	\leq & P \left( |Z| < \frac{1}{M} \right) + \frac{2}{3} \epsilon.
%	\end{align*}
%	Choose $M$ large enough such that $ P \left( |Z| < \frac{1}{M} \right) < \frac{\epsilon}{3} $, we have
%	\begin{align*}
%	P \left( \left| \frac{1}{X_n} \right| >M \right) < \epsilon .
%	\end{align*}
%\end{proof}
%
%Proof of (a):

\subsection{Proof of Theorem \ref{thm:decomp}}
\label{proof:decomp}
\begin{proof}
	(i) Recall that
	$
	dCov_n^2(\mathbf{X},\mathbf{Y}) = (\widetilde{\mathbf{A}} \cdot \widetilde{\mathbf{B}}  ).
	$
	Using the approximation of $b_{st}$ in Proposition \ref{prop:taylor}, we have
	\begin{align*}
	\frac{1}{\tau_X}\widetilde{\mathbf{A}} = \widetilde{\mathbf{1}}_{n \times n} + \frac{1}{2} \widetilde{ \mathbf L}_X + \widetilde{ \mathbf R}_X = \frac{1}{2} \widetilde{ \mathbf L}_X + \widetilde{ \mathbf R}_X,
	\end{align*}
	where $\mathbf{L}_X= (L_{X}(X_{s}, X_{t} ))_{s,t=1}^n$ and $\mathbf{R}_X= (R_{X}(X_{s}, X_{t} ))_{s,t=1}^n$. Similarly, $\frac{1}{\tau_Y} \widetilde{\mathbf{B}} = \frac{1}{2}\widetilde{ \mathbf L}_Y + \widetilde{ \mathbf R}_Y $. Then, we have
	\begin{align*}
	\frac{ dCov_n^2(\mathbf{X},\mathbf{Y})}{\tau} & =( (\frac{1}{2}\widetilde{ \mathbf L}_X + \widetilde{ \mathbf R}_X) \cdot ( \frac{1}{2}\widetilde{ \mathbf L}_Y + \widetilde{ \mathbf R}_Y)) \\
	& = \frac{1}{4} (\widetilde{ \mathbf L}_X \cdot \widetilde{ \mathbf L}_Y  ) + \frac{1}{2} (\widetilde{ \mathbf L}_X \cdot \widetilde{ \mathbf R}_Y  ) + \frac{1}{2} (\widetilde{ \mathbf R}_X \cdot \widetilde{ \mathbf L}_Y  ) + (\widetilde{ \mathbf R}_X \cdot \widetilde{ \mathbf R}_Y  ).
	\end{align*}
	Let $R_n = \frac{1}{2} (\widetilde{ \mathbf L}_X \cdot \widetilde{ \mathbf R}_Y  ) + \frac{1}{2} (\widetilde{ \mathbf R}_X \cdot \widetilde{ \mathbf L}_Y  ) + (\widetilde{ \mathbf R}_X \cdot \widetilde{ \mathbf R}_Y  )$. We show that $\frac{1}{4} (\widetilde{ \mathbf L}_X \cdot \widetilde{ \mathbf L}_Y  )$ can be written as sum of sample component-wise cross-covariances up to a constant factor in the following Lemma.
	\begin{lemma}
		\label{PropEquality}
		\begin{align*}
		 \frac{1}{4}(\widetilde{ \mathbf L}_X \cdot \widetilde{ \mathbf L}_Y  ) = \frac{1}{\tau^2} \sum\limits_{i=1}^{p} \sum\limits_{j=1}^{q} \text{cov}_{n}^2 (x_{i}, y_{j}) .
		\end{align*}
	\end{lemma}
\begin{proof}
	By Lemma A.1. of \cite{park2015}, since all diagonal entries of distance matrices $\mathbf A$ and $\mathbf B$ are equal to 0, we have
	$
	(\widetilde{\mathbf{A}} \cdot \widetilde{\mathbf{B}} ) = (\mathbf{A} \cdot \widetilde{ \widetilde{\mathbf{B}} } ).
	$
	Then, it can be directly verified that for any $1\leq s,t \leq n$, $\sum_{u =1 }^{n} \widetilde{b}_{ut} = \sum_{v = 1}^{n} \widetilde{b}_{sv} = 0$ and it further implies that
	\begin{itemize}
		\item[(i)] $ \widetilde{ \widetilde{\mathbf{B}} } = \widetilde{ \mathbf B}$ as long as the diagonal elements of $\mathbf{B}$ are 0;
		\item[(ii)] $\widetilde{ \mathbf B} = 0$ if $\mathbf{B} = \mathbf{a} \mathbf{1}_n^T \text{ or } \mathbf{B} =   \mathbf{1}_n \mathbf{a}^T = 0$ for any vector $\mathbf{a} \in \mathbb{R}^{n}$.
	\end{itemize}
	Direct calculation shows that
	\begin{multline}
	\label{eq:defDcov}
	(\mathbf{A} \cdot  \widetilde{\mathbf{B}}  )
	=   \frac{1}{\binom{n}{2}} \frac{1}{2!} \sum\limits_{(s,t) \in \mathbf i_{2}^{n}} a_{st}b_{st}
	 \\ + \frac{1}{\binom{n}{4}} \frac{1}{4!} \sum\limits_{(s,t,u,v) \in \mathbf i_{4}^{n}} a_{st}b_{uv}   	
	- \frac{2}{\binom{n}{3}} \frac{1}{3!} \sum\limits_{(s,t,u) \in \mathbf i_{3}^{n}}  a_{st}b_{su},
	\end{multline}
	where $\mathbf{i}_{m}^{n}$ denotes the set of all $m$-tuples drawn without replacement from $\{ 1,2, \cdots, n \}$. Equation \eqref{eq:defDcov} can be used as equivalent definition of the sample distance covariance.
	Notice that
	\begin{align*}
	\widetilde{\mathbf{L}}_X = \widetilde{ \mathbf D}_X - \widetilde{ \mathbf 1}_{n \times n} = \widetilde{ \mathbf D}_X,
	\end{align*}
	where $ \mathbf{D}_X = \frac{1}{\tau_{X}^2}(|X_s - X_t|^2)_{s,t=1}^n$. Similarly, $\widetilde{\mathbf{L}}_Y = \widetilde{ \mathbf D}_Y$. Then, we can further decompose $ \widetilde{ \mathbf D}_X $ as follows,
	\begin{align*}
	\widetilde{ \mathbf D}_X = \widetilde{ \mathbf D}_{X, 1} + \widetilde{ \mathbf D}_{X, 2} + \widetilde{ \mathbf D}_{X, 3} = \widetilde{ \mathbf D}_{X, 2},
	\end{align*}
	where $ \mathbf {D}_{X, 1} = \frac{1}{\tau_X^2} (X_{s}^T X_s)_{s,t=1}^n$, $\mathbf {D}_{X, 2} =-2 \frac{1}{\tau_X^2} (X_{s}^T X_t)_{s,t=1}^n$ and $\mathbf {D}_{X, 3} = \frac{1}{\tau_X^2} (X_{t}^T X_t)_{s,t=1}^n$. Similarly, $ \widetilde{ \mathbf D}_Y = \widetilde{ \mathbf D}_{Y, 2}$. Next,
	using Equation \eqref{eq:defDcov}, we have
	\begin{align*}
	 & \tau^2  \times (\widetilde{ \mathbf L}_X \cdot \widetilde{ \mathbf L}_Y  ) \\
	 = & \tau^2  \times (\widetilde{ \mathbf D}_{X, 2} \cdot \widetilde{ \mathbf D}_{Y , 2}  ) \\
	=  & 4 \left\lbrace  \frac{1}{\binom{n}{2}} \frac{1}{2!} \sum\limits_{(s,t) \in i_{2}^{n}}  X_s^T X_t Y_s^T Y_t + \right.  \\
	 & \left.  \frac{1}{\binom{n}{4}} \frac{1}{4!} \sum\limits_{(s,t,u,v) \in \mathbf i_{4}^{n}} X_s^T X_t Y_u^T Y_v
	  -\frac{2}{\binom{n}{3}} \frac{1}{3!} \sum\limits_{(s,t,u) \in \mathbf i_{3}^{n}} X_s^T X_t Y_s^T Y_u \right\rbrace \\
	  =  & 4 \sum\limits_{i=1}^p \sum\limits_{j=1}^q  \left\lbrace  \frac{1}{\binom{n}{2}} \frac{1}{2!} \sum\limits_{(s,t) \in i_{2}^{n}}  x_{si} x_{ti} y_{sj} y_{tj} + \right.  \\
	  & \quad\quad \left.  \frac{1}{\binom{n}{4}} \frac{1}{4!} \sum\limits_{(s,t,u,v) \in \mathbf i_{4}^{n}} x_{si} x_{ti} y_{uj} y_{vj}
	  -\frac{2}{\binom{n}{3}} \frac{1}{3!} \sum\limits_{(s,t,u) \in \mathbf i_{3}^{n}} x_{si} x_{ti} y_{sj} y_{uj} \right\rbrace \\
	  = & 4 \sum_{i=1}^p \sum_{j=1}^q  \left\lbrace   \frac{ 1 }{\binom{n}{4}}    \sum_{k< l< s < t  }  \frac{1}{4!} \sum_{  * }^{(k,  l , s, t )}  \frac{(x_{ki} - x_{li}) (y_{kj} - y_{lj})(x_{si} - x_{ti}) (y_{sj} - x_{tj})}{4}    \right\rbrace  \\
	  = & 4 \sum_{i=1}^p \sum_{j=1}^q cov_n^2(\mathcal X_i, \mathcal Y_j).
	\end{align*}
\end{proof}

Therefore, by Lemma \ref{PropEquality}, we have the following decomposition,
\begin{align*}
dCov^2_n(\mathbf X, \mathbf Y)  =  \frac{1}{\tau} \sum_{i=1}^p \sum_{j=1}^q cov_n^2(\mathcal X_i, \mathcal Y_j)  + \mathcal{R}_n,
\end{align*}
where ${\cal R}_n = \tau R_n$.

	(ii) Note $L_X(X_s, X_t) = O_p (a_p)=o_p(1)$ and $L_Y(Y_s, Y_t)=O_p(b_q)=o_p(1)$ for $s \ne t \in \{1,\dots,n\}$. We can then apply Proposition \ref{prop:taylor}, obtain that $R_X(X_s, X_t) = O_p (L_X(X_s, X_t)^2)$  and $R_Y(Y_s, Y_t) = O_p (L_Y(Y_s, Y_t)^2)$. For the leading term $\tau (\widetilde{ \mathbf L}_X \cdot \widetilde{ \mathbf L}_Y  )$, it can be easily seen from Equation \eqref{eq:defDcov} that $(\widetilde{ \mathbf L}_X \cdot \widetilde{ \mathbf L}_Y  ) = O_p (a_p b_q) $. Similarly, for the remainder terms,
%	\begin{align*}
% 	 (\widetilde{ \mathbf L}_X \cdot \widetilde{ \mathbf R}_Y  )
%	 =  &   \frac{1}{\binom{n}{2}} \frac{1}{2!} \sum\limits_{(s,t) \in \mathbf i_{2}^{n}} L_X(X_s, X_t) R_Y(Y_s, Y_t)   \\
%	 &  + \frac{1}{\binom{n}{4}} \frac{1}{4!} \sum\limits_{(s,t,u,v) \in \mathbf i_{4}^{n}} L_X(X_s, X_t) R_Y(Y_u, Y_v) \\
%	 & -\frac{2}{\binom{n}{3}} \frac{1}{3!} \sum\limits_{(s,t,u) \in \mathbf i_{3}^{n}} L_X(X_s, X_t) R_Y(Y_s, Y_u)  \\
%	 \asymp_p  &   \frac{1}{\binom{n}{2}} \frac{1}{2!} \sum\limits_{(s,t) \in \mathbf i_{2}^{n}} L_X(X_s, X_t) (L_Y(Y_s, Y_t))^2    \\
%	 &  + \frac{1}{\binom{n}{4}} \frac{1}{4!} \sum\limits_{(s,t,u,v) \in \mathbf i_{4}^{n}} L_X(X_s, X_t) (L_Y^2(Y_u, Y_v))^2  \\
%	 & -\frac{2}{\binom{n}{3}} \frac{1}{3!} \sum\limits_{(s,t,u) \in \mathbf i_{3}^{n}} L_X(X_s, X_t) (L_Y^2(Y_s, Y_u))^2  \\
%	 \asymp_p & a_p b_q^2
%	 \end{align*}
	$(\widetilde{ \mathbf L}_X \cdot \widetilde{ \mathbf R}_Y  ) = O_p ( a_p b_q^2 )$, $(\widetilde{ \mathbf R}_X \cdot \widetilde{ \mathbf L}_Y  ) = O_p(a_p^2b_q) $ and $(\widetilde{ \mathbf R}_X \cdot \widetilde{ \mathbf R}_Y  ) =O_p ( a_p^2b_q^2) $. Thus, we have $R_n = O_p (a_p^2b_q + a_p b_q^2 )$ and ${\cal R}_n = \tau  R_n  = O_p (\tau a_p^2b_q + \tau a_p b_q^2 ) = o_p(1) $. Therefore the remainder terms are negligible comparing to the leading term.
\end{proof}

\subsection{Proof of Theorem \ref{thm:decompHsic}}
\begin{proof}
	(i) We first show that $\gamma_{\mathbf X}$ is asymptotically equal to $\tau_{X}$ (similar result applies to $\gamma_{\mathbf Y}$ and $\tau_{Y}$). Recall that for all $s \neq t$,
	\begin{align*}
	L_{X}(X_{s}, X_{t}) = \frac{|X_{s} - X_{t}|^{2} - \tau_{X}^{2}}{ \tau_{X}^{2}}.
	\end{align*}
	Since $ L_{X}(X_{s}, X_{t}) = O_{p} (a_{p}) =o_p(1)$, we have
	$\frac{|X_{s} - X_{t}|^{2} }{ \tau_{X}^{2} } \overset{p}{\rightarrow} 1$. Then
	$$
	\frac{\text{median}\{|X_{s} - X_{t}|^{2}\} }{ \tau_{X}^{2} } \overset{p}{\rightarrow} 1
	$$
	and thus
	\begin{align*}
	\frac{\tau_{X}}{\gamma_{\mathbf X}} = \sqrt{ \frac{ \tau_{X}^{2} }{\text{median}\{|X_{i} - X_{j}|^{2}\} } } \overset{p}{\rightarrow} 1.
	\end{align*}
	Similar arguments can also be used to show that $ \frac{\tau_{Y}}{\gamma_{\mathbf Y}} \overset{p}{\rightarrow} 1$. Next, under Proposition \ref{prop:taylor}, we can deduce that
	\begin{align*}
	 & f \left( \frac{ |X_{s} - X_{t}|}{ \gamma_{\mathbf X}} \right) \\
	= &  f \left( \frac{|X_{s} - X_{t}|}{ \tau_{X} }  \frac{\tau_{X}}{\gamma_{\mathbf X}} \right) \\
	= &  f\left( \left\lbrace  1+ \frac{ L_{X}(X_{s}, X_{t})}{2} + R_{X}(X_{s}, X_{t}) \right\rbrace  \frac{\tau_{X}}{\gamma_{\mathbf X}} \right) \\
	= &  f\left(\frac{\tau_{X}}{\gamma_{\mathbf X}}\right) + f^{(1)} \left( \frac{\tau_{X}}{ \gamma_{\mathbf X}} \right) \left\lbrace  \frac{ L_{X}(X_{s}, X_{t})}{2} + R_{X}(X_{s}, X_{t}) \right\rbrace  \frac{\tau_{X}}{\gamma_{\mathbf X}} + R_{f}(X_{s},X_{t}),
	\end{align*}
	where $ R_{f}(X_{s},X_{t}) $ is the remainder term. Similarly,
	\begin{multline*}
	g \left( \frac{|Y_{s} - Y_{t}|}{\gamma_{\mathbf Y}} \right) = g\left(\frac{\tau_{Y}}{\gamma_{\mathbf Y}}\right)  + \\ g^{(1)} \left( \frac{\tau_{Y}}{\gamma_{\mathbf Y}} \right) \left\lbrace  \frac{L_{Y}(Y_{s}, Y_{t})}{2} +  R_{Y}(Y_{s}, Y_{t}) \right\rbrace  \frac{\tau_{Y}}{\gamma_{\mathbf Y}} + R_{g}(Y_{s},Y_{t}).
	\end{multline*}	
	Similar to the proof of Theorem \ref{thm:decomp},
	\begin{multline*}
	hCov_n^2(\mathbf X, \mathbf Y) = ( \widetilde{\mathbf R} \cdot \widetilde{\mathbf H}) \\= \frac{1}{4} f^{(1)} \left( \frac{\tau_{X}}{ \gamma_{\mathbf X}} \right) g^{(1)} \left( \frac{\tau_{Y}}{\gamma_{\mathbf Y}} \right)\frac{\tau_{X}}{ \gamma_{\mathbf X}} \frac{\tau_{Y}}{\gamma_{\mathbf Y}}  (\widetilde{ \mathbf L}_X \cdot \widetilde{ \mathbf L}_Y  )  + \frac{1}{2}f^{(1)} \left( \frac{\tau_{X}}{ \gamma_{\mathbf X}} \right)\frac{\tau_{X}}{ \gamma_{\mathbf X}} (\widetilde{ \mathbf L}_X \cdot \widetilde{ \mathbf R}_Y  )\\ + \frac{1}{2}g^{(1)} \left( \frac{\tau_{Y}}{\gamma_{\mathbf Y}} \right)\frac{\tau_{Y}}{\gamma_{\mathbf Y}} (\widetilde{ \mathbf R}_X \cdot \widetilde{ \mathbf L}_Y  ) + (\widetilde{ \mathbf R}_X \cdot \widetilde{ \mathbf R}_Y  ),
	\end{multline*}
	where $\mathbf{L}_X= (L_{X}(X_{s}, X_{t} ))_{s,t=1}^n$, $\mathbf{L}_Y= (L_{Y}(Y_{s}, Y_{t} ))_{s,t=1}^n$ and
	\begin{align*}
	& \mathbf{R}_X= \left(f^{(1)} \left( \frac{\tau_{X}}{ \gamma_{\mathbf X}} \right) \frac{\tau_{X}}{\gamma_{\mathbf X}} R_{X}(X_{s}, X_{t}) + R_{f}(X_{s}, X_{t})\right)_{s,t=1}^n, \\
	& \mathbf{R}_Y= \left(g ^{(1)} \left( \frac{\tau_{Y}}{ \gamma_{\mathbf Y}} \right) \frac{\tau_{Y}}{\gamma_{\mathbf Y}} R_{Y}(Y_{s}, Y_{t}) + R_{g}(Y_{s}, Y_{t})\right)_{s,t=1}^n.
	\end{align*}
	Denote $R_n = \frac{1}{2}f^{(1)} \left( \frac{\tau_{X}}{ \gamma_{\mathbf X}} \right)\frac{\tau_{X}}{ \gamma_{\mathbf X}} (\widetilde{ \mathbf L}_X \cdot \widetilde{ \mathbf R}_Y  )+ \frac{1}{2}g^{(1)} \left( \frac{\tau_{Y}}{\gamma_{\mathbf Y}} \right)\frac{\tau_{Y}}{\gamma_{\mathbf Y}} (\widetilde{ \mathbf R}_X \cdot \widetilde{ \mathbf L}_Y  ) + (\widetilde{ \mathbf R}_X \cdot \widetilde{ \mathbf R}_Y  )$ and $\mathcal R_{n} = \tau R_n$. By Lemma \ref{PropEquality}, we have
	\begin{multline*}
	\tau \times hCov^2_{n}(\mathbf X, \mathbf Y) = \\ f^{(1)} \left( \frac{\tau_{X}}{\gamma_{\mathbf X}} \right)  g^{(1)} \left( \frac{\tau_{Y}}{ \gamma_{\mathbf Y}} \right) \frac{\tau_{X}}{\gamma_{\mathbf X}}  \frac{\tau_{Y}}{\gamma_{ \mathbf Y}}  \frac{1}{\tau} \sum\limits_{i=1}^{p} \sum\limits_{j=1}^{q} cov_{n}^2 (\mathcal X_{i}, \mathcal Y_{j}) + \mathcal{R}_{n}.
	\end{multline*}
	
	(ii) We present the following lemma which would be useful in subsequent arguments.
	\begin{lemma}
		\label{TailOder}
		Suppose $f^{(2)}$ and $g^{(2)}$ are continuous on some open interval containing 1. Then under the assumptions of Theorem \ref{thm:decompHsic},
		$$
		R_{f}(X_{s},X_{t}) = O_{p}( L_{X}(X_{s}, X_{t})^2 ), \quad R_{g}(Y_{s},Y_{t}) = O_{p}( L_{Y}(Y_{s}, Y_{t})^2 ).
		$$
	\end{lemma}
	\begin{proof}
%		The remainder term can be written as
%		$$
%		R_{f}(X_{s},X_{t}) = \frac{1}{2} f^{(2)}(x^{*}) \left( \frac{\tau_{X}}{\gamma_{\mathbf X}}\right)^2 \left( \frac{1}{2} L_{X}(X_{s}, X_{t}) + R_{X}(X_{s}, X_{t}) \right)^2,
%		$$
%		where $x^{*}$ is between $\frac{\tau_{X}}{\gamma_{\mathbf X}} $ and $  \frac{\tau_{X}}{\gamma_{\mathbf X}} + \left\lbrace  \frac{1}{2} L_{X}(X_{s}, X_{t}) + R_{X}(X_{s}, X_{t}) \right\rbrace \frac{ \tau_{X}}{\gamma_{\mathbf X}} .
%		$
%		By Proposition \ref{gamma}, $ c_{1}:= \tau_{X}/\gamma_{\mathbf X} = O_{p}( 1)$ and
%		\begin{align*}
%		c_{2}:= \frac{\tau_{X}}{\gamma_{\mathbf X}} + \left\lbrace  \frac{1}{2} L_{X}(X_{s}, X_{t}) + R_{X}(X_{s}, X_{t}) \right\rbrace \frac{ \tau_{X}}{\gamma_{\mathbf X}} = O_{p} (1) .
%		\end{align*}
%		Then, since $|x^{*}| < \max \{ |c_{1}|, |c_{2}| \}$, $x^{*} = O_{p}(1)$, Lemma \ref{bigo3}  implies that
%		$$
%		\frac{1}{2} f^{(2)}(x^{*}) \left( \frac{\tau_{X}}{\gamma_{\mathbf X}}\right)^2 = O_{p}(1).
%		$$
%		On the other hand, since $R_{X}(X_{s}, X_{t}) = O_{p}( L_{X}(X_{s}, X_{t})^2 ),$
%		$$
%		\left( \frac{1}{2} L_{X}(X_{s}, X_{t}) + R_{X}(X_{s}, X_{t}) \right)^2 =  O_{p}( L_{X}(X_{s}, X_{t})^2 ).
%		$$
%		Thus, $R_{f}(X_{s},X_{t}) = O_{p}( L_{X}(X_{s}, X_{t})^2 )$. Similarly, 	$R_{g}(Y_{s},Y_{t}) = O_{p}( L_{Y}(Y_{s}, Y_{t})^2 ).$
		The remainder term can be written as
		\begin{multline}\label{eq:rf}
        %\begin{split}
		R_{f}(X_{s},X_{t}) =\\  \int_{0}^1 \int_{0}^{1} v f^{(2)} \left( \frac{\tau_{X}}{\gamma_{\mathbf X}} + uv \left\lbrace \frac{L_{X}(X_s, X_t) }{2} + R_X (X_s, X_t) \right\rbrace \frac{\tau_{X}}{\gamma_{\mathbf X}} \right) dudv
        \\  \times \left(\frac{\tau_{X}}{\gamma_{\mathbf X}}\right)^2\left( \frac{L_{X}(X_s, X_t) }{2} + R_X (X_s, X_t) \right)^2.
		%\end{split}
        \end{multline}
		Set $ \varphi (x,y) = \int_{0}^1 \int_{0}^{1} v f^{(2)} \left( x + uv y \right) dudv $. Then $\varphi (x,y)$ is continuous at $(1,0)$. By the continuous mapping theorem, we have
		\begin{multline*}
		\int_{0}^1 \int_{0}^{1} v f^{(2)} \left( \frac{\tau_{X}}{\gamma_{\mathbf X}} + uv \left\lbrace \frac{L_{X}(X_s, X_t) }{2} + R_X (X_s, X_t) \right\rbrace \frac{\tau_{X}}{\gamma_{\mathbf X}} \right) dudv \\ \overset{p}{\rightarrow}
		\int_{0}^1 \int_{0}^{1} v f^{(2)} \left( 1 \right) dudv.
		\end{multline*}
		So $ R_{f}(X_{s},X_{t})  = O_p (1) \left( \frac{L_{X}(X_s, X_t) }{2} + R_X (X_s, X_t) \right)^2 = O_p ( L_{X}(X_{s}, X_{t})^2 ) $. Similar argument holds for $R_{g}(Y_{s},Y_{t})$.
	\end{proof}
	Both the Gaussian and Laplacian kernel have continuous second order derivatives. From Lemma \ref{TailOder}, we know
	\begin{align*}
	& f^{(1)} \left( \frac{\tau_{X}}{ \gamma_{\mathbf X}} \right) \frac{\tau_{X}}{\gamma_{\mathbf X}} R_{X}(X_{s}, X_{t}) + R_{f}(X_{s}, X_{t}) =O_p ( L_{X}(X_{s}, X_{t})^2 ), \\
	& g^{(1)} \left( \frac{\tau_{Y}}{ \gamma_{\mathbf Y}} \right) \frac{\tau_{Y}}{\gamma_{\mathbf Y}} R_{Y}(Y_{s}, Y_{t}) + R_{g}(Y_{s}, Y_{t}) = O_p( L_{Y}(Y_{s}, Y_{t})^2 ).
	\end{align*}
	Thus, similar arguments in Theorem \ref{thm:decomp} can be used to show that $ {\cal R}_n = O_p( \tau a_p^2b_q + \tau a_pb_q^2 ) = o_p(1)$.
\end{proof}

\subsection{Proof of Proposition \ref{prop:ind}}
\begin{proof} Clearly, $E\left[ k_{st}(i) l_{uv}(j) \right]=0$ when $\{s,t\} \cap \{ u,v \} = \emptyset$. For any $1 \leq i, i'\leq p, 1 \leq j, j' \leq q,$
	\begin{align*}
	& E\left[ k_{st}(i) l_{su}(j) \right] \\
	= & E \left[ E[ k_{st}(i) l_{su}(j)  | x_{si}, y_{sj} ] \right] \\
	= & E \left[ E[ k_{st}(i)  | x_{si}, y_{sj} ] E[ l_{su}(j)  | x_{si}, y_{sj} ] \right].
	\end{align*}
	Notice that
	\begin{align*}
	& E[ k_{st}(i)  | x_{si}, y_{sj} ] \\
	= & E \big\{   k(x_{si},x_{ti})-E[k(x_{si},x_{ti})|x_{si}]  -E[k(x_{si},x_{ti})|x_{ti}]+E[k(x_{si},x_{ti})] |  x_{si}, y_{sj} \big\}  \\
	= &  E[k(x_{si},x_{ti})|x_{si},y_{sj}] - E[k(x_{si},x_{ti})|x_{si}] - E[k(x_{si},x_{ti})] + E[k(x_{si},x_{ti})]  \\
	%= & E[k(x_{si},x_{ti})|x_{si}, y_{sj}] - E[k(x_{si},x_{ti})|x_{si}] \\
	= & 0.
	\end{align*}
	Thus $ E\left[ k_{st}(i) l_{su}(j) \right]=0 $. Similarly, $E\left[ k_{st}(i) k_{su}(i') \right] = E\left[ l_{st}(j) l_{su}(j') \right] =0.$
\end{proof}

\subsection{Proof of Theorem \ref{thm:key} }
\begin{proof}
Let $\widetilde{ \mathbf K} = (\tilde{k}_{st})_{s,t=1}^n$ and $\widetilde{ \mathbf L} = (\tilde{l}_{st})_{s,t=1}^n$. Notice that
\begin{align*}
uCov^2_{n}(\mathbf X,\mathbf Y) & = (pq)^{-1/2}\sum^{p}_{i=1}\sum^{q}_{j=1} \frac{1}{n(n-3)} \sum_{s\neq
	t}\tilde{k}_{st}(i)\tilde{l}_{st}(j) \\
& = \frac{1}{n(n-3)} \sum_{s\neq t} \left(
	 p^{-1/2} \sum^{p}_{i=1} \tilde{k}_{st}(i)\right) \left( q^{-1/2} \sum^{q}_{j=1} \tilde{l}_{st}(j) \right).
\end{align*}	
Under Assumption \ref{D3}, we have
\begin{align*}
&p^{-1/2}\sum^{p}_{i=1}\tilde{k}_{st}(i)
\\=&p^{-1/2}\sum^{p}_{i=1}k_{st}(i)-\frac{1}{n-2}\sum_{u\neq
	t}p^{-1/2}\sum^{p}_{i=1}k_{ut}(i)
\\& -\frac{1}{n-2}\sum_{ v\neq
	s}p^{-1/2}\sum^{p}_{i=1}k_{sv}(i)
+\frac{1}{(n-1)(n-2)}\sum_{u\neq
	v}p^{-1/2}\sum^{p}_{i=1}k_{uv}(i)
\\ \overset{d }{ \rightarrow } & c_{st}-\frac{1}{n-2}\sum_{u\neq t}c_{ut}-\frac{1}{n-2}\sum_{v\neq s}c_{vs}+\frac{1}{(n-1)(n-2)}\sum_{u\neq
	v}c_{uv}.
\end{align*}
Then we get
\begin{multline*}
n(n-3) \times uCov^2_{n}(\mathbf X, \mathbf Y)  \overset{d}{\rightarrow} \\ \sum_{s\neq t}\left(c_{st}-\frac{1}{n-2}\sum_{u\neq
	t}c_{ut}-\frac{1}{n-2}\sum_{v\neq
	s}c_{sv}+\frac{1}{(n-1)(n-2)}\sum_{u\neq v}c_{uv}\right)
\\ \times \left(d_{st}-\frac{1}{n-2}\sum_{u\neq
	t}d_{ut}-\frac{1}{n-2}\sum_{v\neq
	s}d_{sv}+\frac{1}{(n-1)(n-2)}\sum_{u\neq v}d_{uv}\right).
\end{multline*}
Set
\begin{align*}
& \mathbf{c} = \left( c_{12}, c_{13}, \cdots, c_{1n}, c_{23}, \cdots,c_{2n}, c_{34}, \cdots, c_{n(n-1) } \right)^T, \\
& \mathbf{d} = \left( d_{12}, d_{13}, \cdots, d_{1n}, d_{23}, \cdots,d_{2n}, d_{34}, \cdots, d_{n(n-1) } \right)^T.
\end{align*}
Under Assumption \ref{D3} and by Proposition \ref{prop:ind}, we know that
\begin{align*} \left(
\begin{array}{c}
\mathbf{c} \\
\mathbf{d}
\end{array} \right)  \sim N \left(\mathbf 0, \left(
\begin{array}{cc}
\sigma_{x}^2 \mathbf{I}_{n(n-1)/2} & \sigma_{xy}^2 \mathbf{I}_{n(n-1)/2}  \\
\sigma_{xy}^2 \mathbf{I}_{n(n-1)/2} & \sigma_{y}^2 \mathbf{I}_{n(n-1)/2}
\end{array} \right)
\right).
\end{align*}
Define $\mathbf C=(c_{st})_{s,t=1}^n$ such that $c_{st}=c_{ts}$ and $\widetilde{\mathbf  C}=(\widetilde{c}_{st})_{s,t=1}^n$. Here we assume that $c_{ss}=0$. From the proof of Lemma A.1 of Park et al. (2015), we have
\begin{align*}
\text{vec}( \widetilde{\mathbf  C} )= \mathbf F \mathbf S \text{vec}(\mathbf C)  = \mathbf F \mathbf S \mathbf F \text{vec}(\mathbf C),
\end{align*}
where $\text{vec}(\mathbf C)$ is the usual vectorization of matrix $\mathbf C$; $\mathbf F$ is the matrix of the linear operator that sets the diagonal of a matrix to be 0, i.e., $\text{vec}(\mathbf B_{-D}) =\mathbf  F \text{vec}(\mathbf B)$, $\mathbf B_{-D}$ is $\mathbf B$ with its diagonal set to be 0; Letting $\mathbf J = \mathbf{1}_n \mathbf{1}_n^T$, we define $\mathbf S$ as
\begin{align*}
\mathbf S = \mathbf I_n \otimes \mathbf I_n - \frac{1}{n-2} \mathbf{J} \otimes \mathbf{I}_n - \frac{1}{n-2} \mathbf{I}_n \otimes \mathbf{J} + \frac{1}{(n-1)(n-2)} \mathbf{J} \otimes \mathbf{J}.
\end{align*}
Next, to simplify the following proof, we will use a different vectorization operator, which will align the upper triangular elements frist, then the lower triangular elements and lastly the diagonal elements, i.e., define
\begin{align*}
\widetilde{\text{vec}}(\mathbf C) & = \left( \mathbf{c}_{u}^{T}, \mathbf{c}_{l}^{T}, \mathbf{c}_{d}^{T}  \right)^T, \\
\mathbf{c}_{u}^{T} &  = \left(  c_{12}, c_{13}, \cdots, c_{1n}, c_{23}, \cdots,c_{2n}, c_{34}, \cdots, c_{(n-1)n} \right)^T, \\
\mathbf{c}_{l}^{T} &  =  \left(  c_{21}, c_{31}, \cdots, c_{n1}, c_{32}, \cdots,c_{n2}, c_{43}, \cdots, c_{n(n-1)} \right)^T, \\
\mathbf{c}_{d}^{T} & = \left( c_{11}, c_{22}, \cdots , c_{nn}  \right)^T.
\end{align*}
Notice that there is a permutation matrix $\mathbf{P}_1$ such that $\widetilde{\text{vec}}(\mathbf C) = \mathbf{P}_1 \text{vec}(\mathbf C)$. Then
\begin{align*}
\widetilde{\text{vec}}( \widetilde{ \mathbf C} )= \mathbf{P}_1 \mathbf F \mathbf S \mathbf{F} \mathbf{P}_1^{T} \widetilde{\text{vec}}(\mathbf C).
\end{align*}
Observe that for any matrix $\mathbf C$, both the column sum and row sum of $\widetilde{ \mathbf C}$ are 0. We can verify that $\widetilde{\widetilde{\mathbf C}}=\widetilde{\mathbf C}$.
Set $\mathbf U = \mathbf{P}_1 \mathbf F \mathbf S \mathbf{F} \mathbf{P}_1^T$. It follows that $ \mathbf U^2 \widetilde{\text{vec}}(\mathbf C)=\mathbf U \widetilde{ \text{vec}}(\mathbf C)$ and thus
%I have used this property in Section 1.1 of the supplement of Zhang et al. (2017). We can provide a rigorous proof for this result here.
\begin{align}\label{eq12}
(\mathbf U^2- \mathbf U) \widetilde{\text{vec}}( \mathbf C)=0.
\end{align}
Equation (\ref{eq12}) still holds if we replace $c_{ss}$ by some nonzero elements. Since Equation (\ref{eq12}) holds for any $\widetilde{ \text{vec}}( \mathbf C)$, we must have
$\mathbf U^2= \mathbf U$ which implies that $\mathbf U$ is an idempotent matrix. Next, let $\mathbf C^u$ ($\mathbf C^l$) be the matrix by setting the lower (upper) triangular and diagonal elements in $\mathbf C$ to be zero. Denote
\begin{align*}
\mathbf P_2=\begin{pmatrix}
\mathbf 0 & \mathbf I & \mathbf 0 \\
\mathbf I & \mathbf 0 & \mathbf 0 \\
\mathbf 0 & \mathbf 0 & \mathbf I
\end{pmatrix}, \quad
\mathbf D=\begin{pmatrix}
\mathbf I \\
\mathbf 0 \\
\mathbf 0
\end{pmatrix}.
\end{align*}
Then, we see that $ \widetilde{\text{vec}}(\mathbf C^l) = \mathbf{P}_2 \widetilde{\text{vec}}(\mathbf C^u)$ and
\begin{align*}
\mathbf U\widetilde{\text{vec}}(\mathbf C)= & \mathbf U \widetilde{\text{vec}}(\mathbf C^u) + \mathbf U \mathbf P_2 \widetilde{\text{vec}}(\mathbf C^u)= \mathbf U(\mathbf I+ \mathbf P_2) \widetilde{\text{vec}}(\mathbf C^u)=\mathbf  U (\mathbf I+ \mathbf P_2)\mathbf D \mathbf c.
\end{align*}
%where $D_1$ is a diagonal matrix with ones corresponding to the position for $c_{ij}$ with $i<j$ ($i>j$) and zeros corresponding to other positions.
%Then we have
%\begin{align*}
%U\text{vec}(C)=(2U-I + P)D\text{vech}(C^u),
%\end{align*}
We note that
\begin{align*}
\mathbf W:=&\mathbf D^T (\mathbf I+ \mathbf P_2)\mathbf U \mathbf U( \mathbf I+ \mathbf P_2) \mathbf D
=\mathbf D^T (\mathbf U+ \mathbf U \mathbf P_2+ \mathbf P_2 \mathbf U+\mathbf P_2 \mathbf U \mathbf P_2) \mathbf D.
\end{align*}
We partition $\mathbf U$ into three blocks corresponding to the upper triangular, lower triangular and diagonal elements respective, i.e., we write
\begin{align*}
\mathbf U=\begin{pmatrix}
\mathbf U_{1} & \mathbf U_2 & \mathbf 0 \\
\mathbf U_2 & \mathbf U_{1} & \mathbf  0 \\
\mathbf 0 & \mathbf 0 & \mathbf 0
\end{pmatrix},
\end{align*}
where we have used the symmetry for $\mathbf U$. Then we have
$$\mathbf W=2(\mathbf U_1+ \mathbf U_2).$$
Now we argue that $\mathbf W^2= 2 \mathbf W$.
Recall that $\mathbf U$ is an idempotent matrix. Thus
\begin{align*}
& \mathbf U_1^2+ \mathbf U_2^2=\mathbf U_1, \quad \mathbf U_1 \mathbf U_2+ \mathbf U_2 \mathbf U_1= \mathbf U_2.
\end{align*}
Therefore, we get
$$\mathbf W^2=4( \mathbf U_1+ \mathbf U_2)^2=4( \mathbf U_1^2+ \mathbf U_2^2+\mathbf U_1 \mathbf U_2+ \mathbf U_2 \mathbf U_1)=4(\mathbf U_1+\mathbf U_2)=2 \mathbf W,$$ which indicates that $\mathbf W$ has eigenvalues which are either equal to two or zero.
It remains to show that the rank of $\mathbf W$ is $n(n-3)/2$ or equivalently, the trace of $\mathbf W /2=\mathbf U_1+\mathbf U_2$ is $n(n-3)/2.$
Note that
\begin{align*}
\text{Tr}(\mathbf W/2)=\text{Tr}(\mathbf U_1+\mathbf U_2)=\sum^{n(n-1)/2}_{i=1} \frac{\mathbf r^T_i \mathbf U \mathbf r_i }{ 2} =\frac{ n(n-1)}{4} \widetilde{\text{vec}}(\widetilde{\mathbf E}_1)^T\widetilde{\text{vec}}(\widetilde{\mathbf E}_1),
\end{align*}
where $\mathbf r_i=(\mathbf e_{i}^T,\mathbf e_{i}^T, \mathbf 0 ^T)^T$ and $\mathbf e_i$ is a $n(n-1)/2$-dimensional vector with $1$ on the $i$th position and zero otherwise; $\widetilde{\mathbf E}_i$ denotes the $\mathcal U$-centering version of the matrix $\mathbf E_i$ such that $\widetilde{\text{vec}}(\mathbf E_i) = \mathbf r_{i}$. Direct calculation shows that
\begin{align*}
\text{vec}(\widetilde{\mathbf E}_1)^T \text{vec}(\widetilde{\mathbf E}_1)=&\frac{2(n-3)^2}{(n-1)^2}+4(n-2)\frac{(n-3)^2}{(n-1)^2(n-2)^2}
\\&+(n-2)(n-3)\frac{4}{(n-1)^2(n-2)^2}=\frac{2(n-3)}{n-1},
\end{align*}
which implies that
$4^{-1}n(n-1)\widetilde{\text{vec}}(\widetilde{\mathbf E}_1)^T \widetilde{\text{vec}}(\widetilde{\mathbf E}_1)=n(n-3)/2.$ Using the above results and setting $\mathbf M = \mathbf W/2$, we have
\begin{multline*}
\text{vec}(\widetilde{\mathbf C})^T\text{vec}(\widetilde{\mathbf C})= \widetilde{\text{vec}}(\widetilde{\mathbf C})^T \widetilde{\text{vec}}(\widetilde{ \mathbf C}) = \widetilde{\text{vec}}(\mathbf C)^T \mathbf U\widetilde{\text{vec}}(\mathbf C) \\ =2 \mathbf c^T \mathbf M \mathbf c \sim 2\sigma^2_x \chi^2_{n(n-3)/2}.
\end{multline*}
Thus,
$$
 Cov^2_{n}(\mathbf X, \mathbf X)   \overset{d}{\rightarrow} \frac{2}{n(n-3)} \mathbf{c}^T \mathbf M \mathbf c \overset{d}{=} \frac{2}{n(n-3)}\sigma^{2}_{x}\chi^2_{n(n-3)/2}.
$$
Similarly,
\begin{align*}
& uCov^2_{n}(\mathbf X,\mathbf Y) \overset{d}{\rightarrow}  \frac{2}{n(n-3)} \mathbf{c}^T \mathbf{M} \mathbf{d},\\
& uCov^2_{n}(\mathbf Y, \mathbf Y) \overset{d}{\rightarrow}  \frac{2}{n(n-3)} \mathbf{d}^T \mathbf M \mathbf{d} \overset{d}{=} \frac{2}{n(n-3)} \sigma^2_y \chi^2_{n(n-3)/2}.
\end{align*}
\end{proof}

\subsection{Proof of Proposition \ref{prop:exactT}}
%\begin{lemma}
%	\label{prop:t}
%	Define the test statistic $T_u$ and random variable $\varUpsilon$ as
%	\begin{align*}
%	T_u  = \sqrt{v-1} \frac{uCor_n^2(\mathbf X, \mathbf Y)}{ \sqrt{1 - (uCor_n^2(\mathbf X, \mathbf Y))^{2} } } \text{ and } \varUpsilon  = \sqrt{v-1} \frac{\varUpsilon^{*}}{ \sqrt{1 - (\varUpsilon^{*})^{2} } },
%	\end{align*}
%	where
%	\begin{align*}
%	uCor_n^2(\mathbf X,\mathbf Y) = \frac{ uCov_n^2(\mathbf X,\mathbf Y) }{\sqrt{ uCov_n^2(\mathbf X,\mathbf X) uCov_n^2(\mathbf Y,\mathbf Y)} } \text{, } \varUpsilon^{*} = \frac{ \mathbf{c}^T \tilde{\mathbf M} \mathbf{d} }{\sqrt{ \left( \mathbf{c}^T \tilde{\mathbf M} \mathbf{c} \right) \left( \mathbf{d}^T \tilde{\mathbf M} \mathbf{d} \right) } }
%	\end{align*}
%	and $\mathbf{c}, \mathbf{d}$ are jointly normal, i.e.,
%	\begin{align*}
%	\left(
%	\begin{array}{c}
%	\mathbf{c} \\
%	\mathbf{d}
%	\end{array} \right)  \sim N \left(0, \left(
%	\begin{array}{cc}
%	\sigma_{x}^2 \mathbf{I}_{n(n-1)/2} & \sigma_{xy}^2 \mathbf{I}_{n(n-1)/2}  \\
%	\sigma_{xy}^2 \mathbf{I}_{n(n-1)/2} & \sigma_{y}^2 \mathbf{I}_{n(n-1)/2}
%	\end{array} \right)
%	\right).
%	\end{align*}	
%	Then, Fix $n$ and let $p \wedge q \rightarrow \infty$, under assumption \ref{D1} and \ref{D2}, we have
%	\begin{align*}
%	P(T_u \leq t ) \rightarrow P( \varUpsilon \leq t ) =  E \left[ P \left( t_{v-1, W} \leq t | \mathbf{d} \right)\right],
%	\end{align*}
%	where $W | \mathbf{d} \sim c \chi_{v}, c =  \sigma_{xy}^2 / \sqrt{ \sigma_{x}^2 \sigma_{y}^2 - \sigma_{xy}^4}, v = n(n-3)/2 ,$ and $t_{v-1,W}$ denotes the non-central $t$-distribution with non-central parameter $W$.
%\end{lemma}
\begin{proof}
	Since
	\begin{align*} \left(
	\begin{array}{c}
	\mathbf{c} \\
	\mathbf{d}
	\end{array} \right)  \overset{d}{=} N \left(\mathbf 0, \left(
	\begin{array}{cc}
	\sigma_{x}^2 \mathbf{I}_{n(n-1)/2} & \sigma_{xy}^2 \mathbf{I}_{n(n-1)/2}  \\
	\sigma_{xy}^2 \mathbf{I}_{n(n-1)/2} & \sigma_{y}^2 \mathbf{I}_{n(n-1)/2}
	\end{array} \right)
	\right),
	\end{align*}
	we have
	\begin{align*}
	\mathbf{c} | \mathbf{d} \overset{d}{=} N\left(\mu \mathbf{d}, \sigma^2 \mathbf{I}_{n(n-1)/2} \right),
	\end{align*}
	where $\mu = \sigma_{xy}^2/\sigma_{y}^2, \sigma^2 = (\sigma_{x}^2 \sigma_{y}^2 - \sigma_{xy}^4)/ \sigma_{y}^2$. Set
	\begin{align*}
	\mathbf{z} = \frac{\mathbf{M} \mathbf{d}}{\sqrt{ \left( \mathbf{d}^T \mathbf{M} \mathbf{d} \right) }}.
	\end{align*}
	It can be easily seen that conditional on $\mathbf d$,
	\begin{align*}
	 \mathbf{c}^T \mathbf{z} / \sigma     \sim  N(\mu \mathbf{z}^T \mathbf d/ \sigma, 1) ,
	\end{align*}
	which implies that $\left. (\mathbf{c}^T \mathbf{z})^2 / \sigma^2  \right| \mathbf{d}  \sim  \chi^{2}_{1} (W^2)$ , where $\chi_{1}^{2}(W^2)$ is the non-central chi-squared distribution and $W^2=  \frac{\mu^2}{\sigma^2} \mathbf{d}^T \mathbf{M} \mathbf{d}$ is the non-centrality parameter. Note that conditioned on $\mathbf d$,
	\begin{align*}
	\mathbf M (\mathbf I - \mathbf{z} \mathbf{z}^T) \mathbf c /\sigma \sim N(\mathbf 0, \mathbf M (\mathbf I - \mathbf{z} \mathbf{z}^T) \mathbf M ),
	\end{align*}
	where we have used the fact that $\mathbf M (\mathbf I - \mathbf{z} \mathbf{z}^T) \mathbf d =0$. As $\mathbf M (\mathbf I - \mathbf{z} \mathbf{z}^T) \mathbf M = \mathbf M - \frac{\mathbf M \mathbf d \mathbf{d}^T \mathbf M}{\mathbf{d}^T \mathbf M \mathbf{d}}$ is a projection matrix with rank $v-1$, it is easy to see that conditioned on $\mathbf d$,
	 $$
	 \mathbf c^T (\mathbf I - \mathbf{z}\mathbf{z}^T) \mathbf M (\mathbf I - \mathbf{z}\mathbf{z}^T) \mathbf c / \sigma^{2} \sim \chi_{v -1}^2.
	 $$
	 Next, conditioned on $\mathbf d$, as $ \mathbf{z}^T \mathbf{c} $ and $ (\mathbf I - \mathbf{z}\mathbf{z}^T) \mathbf c $ are independent, we have  $(\mathbf{c}^T \mathbf{z})^2 / \sigma^2$ and $ \mathbf c^T (\mathbf I - \mathbf{z}\mathbf{z}^T) \mathbf M (\mathbf I - \mathbf{z}\mathbf{z}^T) \mathbf c $ are independent.
	Then,
	\begin{align*}
	P_{H_A}(T_u < t) & \rightarrow P \left(  \sqrt{v-1} \frac{ \frac{ \mathbf{c}^T \mathbf{z} }{ \sqrt{ \left( \mathbf{c}^T \mathbf{M} \mathbf{c} \right)  }} }{ \sqrt{1 - \left( \frac{ \mathbf{c}^T  \mathbf{z} }{ \sqrt{ \left( \mathbf{c}^T \mathbf{M} \mathbf{c} \right)  }} \right)^{2} } } <t  \right) \\
	&= E \left[ P \left( \left. \sqrt{v-1} \frac{ \frac{ \mathbf{c}^T \mathbf{z} }{ \sqrt{ \left( \mathbf{c}^T \mathbf{M} \mathbf{c} \right)  }} }{ \sqrt{1 - \left( \frac{ \mathbf{c}^T  \mathbf{z} }{ \sqrt{ \left( \mathbf{c}^T \mathbf{M} \mathbf{c} \right)  }} \right)^{2} } } <t \right| \mathbf{d} \right)\right] \\
	& = E \left[ P \left( \left. \sqrt{v-1} \frac{  \mathbf{c}^T \mathbf{z}  }{ \sqrt{\mathbf{c}^T \mathbf{M} \mathbf{c} - \left(  \mathbf{c}^T  \mathbf{z}  \right)^{2} } } <t \right| \mathbf{d} \right)\right] \\
	& = E \left[ P \left( \left. \frac{ \mathbf{c}^T \mathbf{z}  }{ \sqrt{ \frac{1}{v-1} \mathbf c^T (\mathbf I - \mathbf{z}\mathbf{z}^T) \mathbf M (\mathbf I - \mathbf{z}\mathbf{z}^T) \mathbf c } } <t \right| \mathbf{d} \right)\right] \\
	& = E \left[ P \left( t_{v-1, W} <t  \right)\right]
	\end{align*}
	where $t_{v-1, W}$ is a noncentral $t$-distribution with $v-1$ degrees of freedom and noncentrality parameter $W = \frac{\mu}{\sigma} \sqrt{\mathbf{d}^T \mathbf{M} \mathbf{d}} \overset{d}{=} c \chi_{v}$ for $c = \frac{ \sigma_{xy}^2}{\sqrt{ \sigma_{x}^2 \sigma_{y}^2 - \sigma_{xy}^4}}.$ By setting $c=0$, we get
	$
	P_{H_0}(T_u < t) \rightarrow P \left( t_{v-1} <t \right).
	$
\end{proof}

\subsection{Proof of Proposition \ref{prop:LarA}}
\begin{proof}
Notice that
$$
\phi = \frac{\phi_0}{\sqrt{v}} \Rightarrow c = \frac{\phi_{0}}{\sqrt{v - \phi_0^2 }} = \frac{\phi_{0}}{\sqrt{v}}\left( 1 + O\left( \frac{1}{v} \right) \right).
$$

Next, by the definition of non-central $t$-distribution,
\begin{align*}
P \left( t_{v-1, u} <t  \right) = & P \left(  \frac{Z + u}{\sqrt{\chi^2_{v-1}/(v-1)}} <t  \right) \\
= &  P \left(  Z  <t \sqrt{\chi^2_{v-1}/(v-1)} - u  \right) \\
= & E \left[ P \left( \left. Z  <t \sqrt{\chi^2_{v-1}/(v-1)} - u \right|  \chi^2_{v-1} \right)\right] \\
= & E \left[ \Phi \left(   t \sqrt{\frac{\chi^2_{v-1}}{v-1}} - u  \right)  \right],
\end{align*}
where $\Phi$ is the cdf of standard normal. For notational convenience, set
$$
g(u) = E \left[ \Phi \left(   t \sqrt{\frac{\chi^2_{v-1}}{v-1}} - u  \right)  \right].
$$
Notice that $ P \left( t_{v-1, W} <t  \right) =  g(W) $. By the following asymptotic series [see \cite{laforgia2012asymptotic,tricomi1951asymptotic}],
\begin{align*}
\frac{\Gamma (J+1/2)}{\Gamma (J)} & = \sqrt{J} \left( 1 - \frac{1}{8J} + \frac{1}{128J^2} + \frac{5}{1024 J^3} - \frac{21}{ 32768J^4 } + \cdots \right)\\
& = \sqrt{J} \left(1 + O \left(\frac{1}{J} \right) \right),
\end{align*}
we can get,
\begin{align*}
& E \left[ (W- \phi_{0}) \right] \\
 = & \frac{\phi_{0}}{\sqrt{v}}\left( 1 + O\left( \frac{1}{v} \right) \right) \sqrt{2} \frac{\Gamma((v+1)/2)}{\Gamma (v/2)} - \phi_{0} \\
 = & \phi_{0} \left(1 + O \left(\frac{1}{v} \right) \right) - \phi_{0} \\
 = & O \left(\frac{1}{v} \right),
\end{align*}
as well as
\begin{align*}
& E \left[ (W - \phi_{0})^{2} \right] \\
 = & \phi_{0}^2 E \left[  \left(\frac{\chi_{v}}{ \sqrt{v} }\left( 1 + O\left( \frac{1}{v} \right) \right) - 1 \right)^{2} \right] \\
 = & \phi_{0}^2 E \left[ \frac{\chi_{v}^2}{ v }\left( 1 + O\left( \frac{1}{v} \right) \right) - 2 \frac{\chi_{v}}{\sqrt{v}}\left( 1 + O\left( \frac{1}{v} \right) \right) + 1 \right] \\
 = & \phi_{0}^2 \left\lbrace \left( 1 + O\left( \frac{1}{v} \right) \right) - 2 \left(1 + O \left(\frac{1}{v} \right) \right) + 1 \right\rbrace   \\
 = & O \left(\frac{1}{v} \right),
\end{align*}
and
\begin{align*}
&E \left[ W (W - \phi_{0})^{2} \right] \\
  = & \phi_{0}^3 E \left[ \frac{\chi_{v}}{\sqrt{v}}\left( 1 + O\left( \frac{1}{v} \right) \right) \left(\frac{\chi_{v}}{ \sqrt{v} }\left( 1 + O\left( \frac{1}{v} \right) \right) - 1 \right)^{2} \right] \\
 = & \phi_{0}^3 E \left[ \frac{\chi_{v}^3}{ v^{3/2} } - 2 \frac{\chi_{v}^2}{v} + \frac{\chi_{v}}{\sqrt{v}} \right]\left( 1 + O\left( \frac{1}{v} \right) \right) \\
 = & \phi_{0 }^3 \Bigg\{ \frac{(v+1)}{v^{3/2}} \sqrt{v} \left(1 + O \left(\frac{1}{v} \right) \right)  - 2 + 1 + O \left(\frac{1}{v} \right) \Bigg\} \left( 1 + O\left( \frac{1}{v} \right) \right) \\
 = & O \left(\frac{1}{v} \right).
\end{align*}
We note that
\begin{align*}
\frac{\partial }{ \partial u } \Phi \left( t \sqrt{\frac{\chi^2_{v-1}}{v-1}} - u \right) & = - \phi \left(  t \sqrt{\frac{\chi^2_{v-1}}{v-1}} - u  \right)  \\
\frac{\partial^2 }{ \partial u^2 } \Phi \left(  t \sqrt{\frac{\chi^2_{v-1}}{v-1}} - u  \right) & = -\left(  t \sqrt{\frac{\chi^2_{v-1}}{v-1}} - u  \right) \phi \left(  t \sqrt{\frac{\chi^2_{v-1}}{v-1}} - u  \right).
\end{align*}
Thus,
\begin{align*}
|g^{(2)}(u)|
& = \left| \frac{\partial^2 }{ \partial u^2 } E \left[ \Phi \left( t \sqrt{\frac{\chi^2_{v-1}}{v-1}} - u  \right)  \right] \right| \\
& = \left|  E \left[ \frac{\partial^2 }{ \partial u^2 } \Phi \left(  t \sqrt{\frac{\chi^2_{v-1}}{v-1}} - u  \right) \right] \right| \\
& = \left|  E \left[ -\left(   t \sqrt{\frac{\chi^2_{v-1}}{v-1}} - u  \right) \phi \left(   t \sqrt{\frac{\chi^2_{v-1}}{v-1}} - u   \right)  \right] \right| \\
& \leq   E \left[\left| -\left(   t \sqrt{\frac{\chi^2_{v-1}}{v-1}} - u  \right) \right| \phi \left(  t \sqrt{\frac{\chi^2_{v-1}}{v-1}} - u  \right)  \right]  \\
&  \leq   E \left[ \left(  \left| t \right|  \sqrt{\frac{\chi^2_{v-1}}{v-1}} + u  \right) \phi \left(  t \sqrt{\frac{\chi^2_{v-1}}{v-1}} - u  \right)    \right]  \\
& <  E \left[ \left(  |t| \sqrt{\frac{\chi^2_{v-1}}{v-1}} + |u|  \right)  \right]  \\
& \leq    \left(  |t| E\sqrt{\frac{\chi^2_{v-1}}{v-1}} + |u| \right)      \\
& \leq  \sqrt{2} |t| + |u|.
\end{align*}
Next, we can bound the following integral,
\begin{align*}
&\left| \int_{0}^{1} \int_{0}^{1} a g^{(2)}(\phi_{0} + ab(W - \phi_{0})) db da \right| \\
\leq &  \int_{0}^{1}  \int_{0}^{1} \left|  a g^{(2)}(\phi_{0} + ab(W - \phi_{0})) \right| db da \\
\leq & \int_{0}^{1}  \int_{0}^{1} \sqrt{2} |t| + | \phi_{0} + ab(W - \phi_{0}) | db da \\
\leq & \int_{0}^{1}  \int_{0}^{1} \sqrt{2}|t| +  \phi_{0} + |W|  db da \\
= & \sqrt{2}|t| +  \phi_{0} + W .
\end{align*}
To calculate $ E \left[ P \left( t_{v-1, W} <t \right)\right] = E \left[ g(W) \right] $, taking the Taylor expansion of $g(W)$ around $\phi_{0}$, the asymptotic mean of $W$, we get
\begin{align*}
=&  E \left[ g(W)  \right] \\
=&   g(\phi_{0})
+  g^{(1)}(\phi_{0})E \left[   \left( W - \phi_{0} \right) \right] \\
& + E \left[ \int_{0}^{1} \int_{0}^{1} a g^{(2)}(\phi_{0} + ab(W - \phi_{0})) db da \left( W - \phi_{0} \right)^2  \right] \\
=&  P \left( t_{v-1, \phi_{0}} <t  \right) + O\left(\frac{1}{v}\right) \\
& + E \left[  \int_{0}^{1} \int_{0}^{1} a g^{(2)}(\phi_{0} +   ab(W - \phi_{0})) db da \left( W - \phi_{0} \right)^2  \right].
\end{align*}
Notice that,
\begin{align*}
& \left| E \left[ \int_{0}^{1} \int_{0}^{1} a g^{(2)}(\phi_{0} + ab(W - \phi_{0})) db da \left( W - \phi_{0} \right)^2  \right] \right| \\
\leq & E \left[ \left|  \int_{0}^{1} \int_{0}^{1} a g^{(2)}(\phi_{0} + ab(W - \phi_{0})) db da \left( W - \phi_{0} \right)^2 \right|   \right] \\
\leq & E \left[ (\sqrt{2} |t| + \phi_{0}+ W) \left( W - \phi_{0} \right)^2   \right] \\
\leq &  (\sqrt{2}|t| + \phi_{0}) E \left[  \left( W - \phi_{0} \right)^2  \right] +  E \left[ W \left( W - \phi_{0} \right)^2  \right] \\
= & O \left(\frac{1}{v} \right).
\end{align*}
In conclusion, we have $E \left[ P \left( t_{v-1, W} <t \right)\right] = P \left( t_{v-1, \phi_{0 }} <t \right) +  O \left(\frac{1}{v} \right).$ Since $t_{v-1}^{(\alpha)} \rightarrow Z^{(\alpha)}$ as $n \rightarrow \infty$, where $Z^{(\alpha)}$ is the $(1-\alpha)$th percentile of standard normal, $t_{v-1}^{(\alpha)}$ is bounded. Then, all the above analysis still holds if we replace $t$ with $t_{v-1}^{\alpha}$.
\end{proof}

Let $B(\cdot, \cdot)$ denote the beta function and $I_y ( \cdot , \cdot)$ denote the regularized incomplete beta function. In the following, we express $ E \left[ P \left( t_{v-1, W} \leq t \right)\right] $ as a sum of infinite series.
\begin{lemma}
	\label{lem:exact}
	$E \left[ P \left( t_{v-1, W} \leq t \right)\right]$ can be calculated exactly as
	\begin{multline*}
	E \left[  P \left( t_{v-1, W} <t  \right) \right]
	= \left(\frac{1}{c^2+1} \right)^{v/2}  \Bigg\{
	P(t_{v-1} \leq t) + \Bigg. \\ \Bigg. \sum\limits_{j=1}^{\infty}  \left( \frac{c^2 }{c^2 + 1} \right)^{j/2}  \frac{1}{jB(j/2, v/2)} \left( (-1)^j + I_{\frac{t^2}{t^2+v-1}} (\frac{j+1}{2}, \frac{v-1}{2}) \right)
	\Bigg\} .
	\end{multline*}
\end{lemma}
\begin{proof}
	Notice that from \cite{walck1996hand}, the CDF of non-central $t$-distribution for $t \geq 0$ can be written as
	\begin{multline*}
	P \left( t_{v-1, W} <t \right)   = \frac{1}{2 \sqrt{\pi}} \times \\ \sum\limits_{j=0}^{\infty} \frac{2^{\frac{j}{2}}}{j!} W^j \exp \left\lbrace- \frac{W^2}{2}  \right\rbrace \Gamma \left(\frac{j+1}{2}\right) \left( (-1)^j + I_{z} \left(\frac{j+1}{2}, \frac{v-1}{2}\right) \right),
	\end{multline*}
	where
	\begin{align*}
	z & =\frac{t^2}{t^2+v-1}, \; v= \frac{n(n-3)}{2},\\
	I_{y} & (\cdot, \cdot)  \text{ is the regularized incomplete beta function} ,\\
	W & = \frac{\mu}{\sigma} \sqrt{\mathbf{d}^T \mathbf{M} \mathbf{d}} \overset{d}{=} c \chi_{v}, c = \frac{ \sigma_{xy}^2}{\sqrt{ \sigma_{x}^2 \sigma_{y}^2 - \sigma_{xy}^4}}.
	\end{align*}
	Next, we calculate the expectation by constructing a generalized gamma distribution,
	\begin{align*}
	& E \left[ W^j \exp \left\lbrace- \frac{W^2}{2} \right\rbrace \right] \\
	= & \int_{0}^{\infty} w^j \exp \left\lbrace- \frac{w^2}{2}  \right\rbrace \frac{1}{c} \frac{1}{2^{v/2 -1} \Gamma (v/2)} \left(\frac{w}{c}\right)^{v-1} \exp \left\lbrace - \frac{w^2}{2 c^2}  \right\rbrace dw \\
	= &  \frac{1}{c^{v}} \frac{1}{2^{v/2 -1} \Gamma (v/2)} \int_{0}^{\infty} \exp \left\lbrace - \left( \frac{w}{\sqrt{2c^2 /(c^2 + 1)}} \right)^2 \right\rbrace w^{j + v -1} dw \\
	= & \frac{1}{c^{v}} \frac{1}{2^{v/2 -1} \Gamma (v/2)} \frac{\Gamma ( j/2 + v/2 ) (\sqrt{2c^2 /(c^2 + 1)})^{j + v}}{2} \\
	= & \frac{(\sqrt{2c^2 /(c^2 + 1)})^{j + v}}{c^{v}} \frac{1}{2^{v/2 }} \frac{\Gamma ( j/2 + v/2 )}{\Gamma (v/2)}.
	\end{align*}
	Then,
	\begin{multline*}
	E \left[  P \left( t_{v-1, W} <t  \right) \right]
	=   \frac{1}{2 \sqrt{\pi}} \left( \sqrt{\frac{1}{c^2+1}} \right)^{v} \times \\ \sum\limits_{j=0}^{\infty} \left( \frac{4c^2 }{c^2 + 1} \right)^{\frac{j}{2}} \frac{\Gamma((j+1)/2)\Gamma ( j/2 + v/2 )}{j!\Gamma (v/2)}\left( (-1)^j + I_{z} (\frac{j+1}{2}, \frac{v-1}{2}) \right).
	\end{multline*}
	According to the gamma duplicate formula,
	$$
	\Gamma \left( \frac{j+1}{2} \right) = \frac{\sqrt{\pi}}{2^j} \frac{\Gamma (j+1)}{\Gamma (j/2 + 1)},
	$$
	which further implies that
	\begin{align*}
	\frac{\Gamma((j+1)/2)\Gamma ( j/2 + v/2 )}{j!\Gamma (v/2)} & =  \frac{\sqrt{\pi}}{2^j} \frac{\Gamma (j+1)}{\Gamma (j/2 + 1)} \frac{\Gamma ( j/2 + v/2 )}{j!\Gamma (v/2)}  \\
	& = \left\lbrace
	\begin{array}{lc}
	\sqrt{\pi}, & j=0 \\
	\frac{\sqrt{\pi}}{2^{j-1}} \frac{1}{j \Gamma (j/2 )} \frac{\Gamma ( j/2 + v/2 )}{\Gamma (v/2)}, & j \geq 1
	\end{array}
	\right. \\
	&  = \left\lbrace
	\begin{array}{lc}
	\sqrt{\pi}, & j=0 \\
	\frac{\sqrt{\pi}}{j2^{j-1}} \frac{1}{B(j/2, v/2)}, & j \geq 1
	\end{array}
	\right.
	\end{align*}
	where $B(\cdot, \cdot)$ is the beta function. Then, the expectation can be further simplified as
	\begin{multline*}
	E \left[  P \left( t_{v-1, W} <t \right) \right]
	= \frac{1}{2 } \left(\frac{1}{c^2+1} \right)^{v/2}  \left(1 + I_{z} (\frac{1}{2}, \frac{v-1}{2}) \right) + \\ \left( \frac{1}{c^2+1} \right)^{v/2} \sum\limits_{j=1}^{\infty}  \left( \frac{c^2 }{c^2 + 1} \right)^{\frac{j}{2}}  \frac{1}{j B(j/2, v/2)} \left( (-1)^j + I_{z} (\frac{j+1}{2}, \frac{v-1}{2}) \right).
	\end{multline*}
	Notice that
	$$
	\frac{1}{2} \left(1 + I_{z} (\frac{1}{2}, \frac{v-1}{2}) \right)  = P(t_{v-1} \leq t).
	$$
	Thus,
	\begin{multline*}
	E \left[  P \left( t_{v-1, W} <t  \right) \right]
	= \left(\frac{1}{c^2+1} \right)^{v/2}  \Bigg\{
	P(t_{v-1} \leq t) + \Bigg. \\ \Bigg. \sum\limits_{j=1}^{\infty}  \left( \frac{c^2 }{c^2 + 1} \right)^{j/2}  \frac{1}{jB(j/2, v/2)} \left( (-1)^j + I_{z} (\frac{j+1}{2}, \frac{v-1}{2}) \right)
	\Bigg\} .
	\end{multline*}
%	Let $t = 0$, by the definition of noncental $t$ distribution, the left hand side of equation \eqref{distri} can be written as
%	\begin{align*}
%	E \left[  P \left(  t_{v-1, W} <0 \right) \right] & = E \left[  P \left(  Z +  W <0 | W \right) \right] \\
%	& =  P \left(Z +  W <0  \right) \\
%	& =  P \left(Z +  c \chi_{v} <0  \right) \\
%	& =  P \left(\frac{Z}{\sqrt{\chi_{v}^2 /v } } < - c \sqrt{v}  \right) \\
%	& = P( t_{v} < -c \sqrt{v} ),
%	\end{align*}
%	where $Z$ is standard normal. Also apply $t=0$ to the right hand side of equation \eqref{distri}, we have
%	\begin{align}
%	\label{dis2}
%	P( t_{v} < -c \sqrt{v} )\left( c^2+1 \right)^{v/2}-\frac{1}{2} =  \sum\limits_{j =1}^{\infty}  \left( \frac{c^2 }{c^2 + 1} \right)^{j/2}  \frac{1}{jB(j/2, v/2)} (-1)^j.
%	\end{align}
%	Plug equation \eqref{dis2} back to equation \eqref{distri}, we get
%	\begin{multline*}
%	E \left[  P \left( t_{v-1, W} <t  \right) \right]
%	= P( t_{v} < -c \sqrt{v} ) +
%	\left(\frac{1}{c^2+1} \right)^{v/2}\left( P(t_{v-1} \leq t) -\frac{1}{2} \right)   \\   +\left(\frac{1}{c^2+1} \right)^{v/2}  \sum\limits_{j=1}^{\infty}  \left( \frac{c^2 }{c^2 + 1} \right)^{j/2}  \frac{1}{jB(j/2, v/2)}   I_{y} (\frac{j+1}{2}, \frac{v-1}{2}) .
%	\end{multline*}
%	Replace $c$ with $\frac{\phi}{\sqrt{1-\phi^2}}$, we get the result.
%	
\end{proof}

\subsection{Proof of Proposition \ref{signal:normal}}
\begin{proof}
	Since we have
	\begin{align*} \left(
	\begin{array}{c}
	X \\
	Y
	\end{array} \right)  \sim N \left(\mathbf 0, \left(
	\begin{array}{cc}
	\mathbf{I}_{p} & \bm{\Sigma}_{XY}  \\
	\bm{\Sigma}_{XY}^T  & \mathbf{I}_{q}
	\end{array} \right)
	\right),
	\end{align*}
	from Theorem 7 in \cite{szekely2007}, by setting $c = \frac{1}{4 (\pi/3 - \sqrt{3} +1)}$, we obtain
	\begin{align*}
	c \leq \frac{dCor^2(x_{i}, y_{j})}{ cor^2(x_{i}, y_{j}) } \leq 1,
	\end{align*}
	$ cov^2(x_{i}, y_{j}) = cor^2(x_{i}, y_{j})$ and $dCor^2 (x_{i}, y_{j}) = dCov^2 (x_{i}, y_{j})\pi/ c.$
	Combine these results, we have
	\begin{align*}
	c \leq \frac{dCov^2(x_{i}, y_{j}) \pi /c}{ cov^2(x_{i}, y_{j}) } \leq 1.
	\end{align*}
	Notice also that $ dCov^2(x_{i}, x_{i}) =dCov^2(y_{j}, y_{j}) = c/\pi $ and $ cov^2(x_{i}, x_{i}) =cov^2(y_{j}, y_{j}) = 1 $. We finally get $
	0.89^2 \phi_{2}  \leq  \phi_{1} \leq  \phi_{2}.$
\end{proof}

\subsection{Proof of Proposition \ref{prop:uni}}
\begin{proof}
	(i)
	When $k(x,y) = l(x,y) = |x-y|^2$,
	\begin{align*}
	& k_{st}(i) = -2 (x_{si} - E(x_{si}))(x_{ti} - E(x_{ti})), \\
	& l_{st}(j) = -2 (y_{sj} - E(y_{sj}))(y_{tj} - E(y_{tj})).
	\end{align*}
	Thus, letting $\mathbf D_X (i) = (x_{si}x_{ti})_{s,t=1}^n$ and $\mathbf D_Y(j) = (y_{sj}y_{tj})_{s,t=1}^n$, we have
	\begin{align*}
	uCov^2_{n}(\mathbf X, \mathbf Y) &  = \frac{1}{\sqrt{pq}}\sum^{p}_{i=1}\sum^{q}_{j=1} (  \widetilde{\mathbf K}(i) \cdot \widetilde{\mathbf L}(j) ) \\
	& = \frac{1}{\sqrt{pq}}\sum^{p}_{i=1}\sum^{q}_{j=1} 4 (  \widetilde{\mathbf D}_X(i) \cdot \widetilde{\mathbf D}_Y(j) ) \\
	=  & 4 \frac{1}{\sqrt{pq}}\sum\limits_{i=1}^p \sum\limits_{j=1}^q  \left\lbrace  \frac{1}{\binom{n}{2}} \frac{1}{2!} \sum\limits_{(s,t) \in i_{2}^{n}}  x_{si} x_{ti} y_{sj} y_{tj} + \right.  \\
	& \quad  \left.  \frac{1}{\binom{n}{4}} \frac{1}{4!} \sum\limits_{(s,t,u,v) \in \mathbf i_{4}^{n}} x_{si} x_{ti} y_{uj} y_{vj}
	-\frac{2}{\binom{n}{3}} \frac{1}{3!} \sum\limits_{(s,t,u) \in \mathbf i_{3}^{n}} x_{si} x_{ti} y_{sj} y_{uj} \right\rbrace \\
	= & 4\frac{1}{\sqrt{pq}} \sum_{i=1}^p \sum_{j=1}^q cov_n^2(\mathcal X_i, \mathcal Y_j).
	\end{align*}
	Thus,
	\begin{align*}
		dCov^2_n(\mathbf X, \mathbf Y) & =  \frac{1}{\tau} \sum_{i=1}^p \sum_{j=1}^q cov_n^2(\mathcal X_i,\mathcal Y_j)  + {\cal R}'_n = \frac{1}{ 4} \frac{\sqrt{pq}}{\tau} uCov^2_n(\mathbf X, \mathbf Y) + {\cal R}'_n
	\end{align*}
	and
	\begin{align*}
	& \tau \times hCov^2_n(\mathbf X, \mathbf Y) \\
	 = &  f^{(1)} \left( \frac{\tau_{X}}{\gamma_{\mathbf X}} \right) g^{(1)} \left( \frac{\tau_{Y}}{ \gamma_{\mathbf Y}} \right)   \frac{\tau_{X}}{\gamma_{\mathbf X}} \frac{\tau_{Y}}{\gamma_{\mathbf Y}}  \frac{1}{\tau} \sum\limits_{i=1}^{p} \sum\limits_{j=1}^{q} cov_{n}^2 (\mathcal X_{i}, \mathcal Y_{j}) + \mathcal{R}''_{n} \\
	=&  f^{(1)} \left( \frac{\tau_{X}}{\gamma_{\mathbf X}} \right)  g^{(1)} \left( \frac{\tau_{Y}}{ \gamma_{\mathbf Y}} \right)   \frac{\tau_{X}}{\gamma_{\mathbf X}} \frac{\tau_{Y}}{\gamma_{\mathbf Y}}  \frac{1}{ 4 } \frac{\sqrt{pq}}{\tau} uCov^2_n(\mathbf X, \mathbf Y) + {\cal R}''_n.
	\end{align*}

(ii) 	When $k(x,y) = l(x,y) = |x-y|$, we have
\begin{align*}
\widetilde{\mathbf K}(i) = \widetilde{\mathbf K}_1(i) - \widetilde{\mathbf K}_2(i) - \widetilde{\mathbf K}_3(i) + \widetilde{\mathbf K}_4(i) = \widetilde{\mathbf K}_1(i),
\end{align*}
where
\begin{align*}
& \mathbf{K}_1(i) = (k(x_{si}, x_{ti}))_{s,t=1}^n, \mathbf{K}_2(i) = (E[k(x_{si},x_{ti})|x_{si}])_{s,t=1}^n, \\
& \mathbf{K}_3(i) = (E[k(x_{si},x_{ti})|x_{ti}])_{s,t=1}^n, \mathbf{K}_4(i) = (E[k(x_{si},x_{ti})])_{s,t=1}^n.
\end{align*}
Similarly, $\widetilde{\mathbf L}(j) = \widetilde{\mathbf L}_1(j)$ with $ \mathbf{L}_1(j) = (l(y_{sj}, l_{tj}))_{s,t=1}^n .$
Then, we have
\begin{align*}
uCov^2_{n}(\mathbf X, \mathbf Y) &  = \frac{1}{\sqrt{pq}}\sum^{p}_{i=1}\sum^{q}_{j=1} (  \widetilde{\mathbf K}_1(i) \cdot \widetilde{\mathbf L}_1(j) ) \\
& = \frac{1}{\sqrt{pq}}\sum^{p}_{i=1}\sum^{q}_{j=1} dCov_{n}^2(\mathcal X_i, \mathcal Y_j) \\
&=\frac{1}{\sqrt{pq}}\frac{1}{\sqrt{\binom{n}{2}}}  mdCov_{n}^2(\mathbf X, \mathbf Y).
\end{align*}
\end{proof}

\subsection{Proof of Corollary \ref{cor:uni}}
\begin{proof}
	For any fixed $t$ and each $R \in \{ dCov, hCov, mdCov \}$, Proposition \ref{prop:uni} and Theorem \ref{thm:key} imply that
	\begin{align*}
	T_R \overset{d}{\rightarrow}   \sqrt{v-1} \frac{\varUpsilon}{ \sqrt{1 - (\varUpsilon)^{2} }}, 	\text{ where } \varUpsilon=\frac{ \mathbf{c}^T \mathbf M \mathbf{d} }{ \sqrt{ \left( \mathbf{c}^T \mathbf M \mathbf{c} \right) \left( \mathbf{d}^T \mathbf M \mathbf{d} \right) }}.
	\end{align*}
	Then the results follow similarly from the proof of Proposition \ref{prop:exactT}.
\end{proof}

\subsection{Proof of Remark \ref{rem:proof2}}
\label{App:proofRemk2}
\begin{proof}
For notational convenience, set $z_{i} = (x_{i} - x_{i}')^{2} - E[ (x_{i} - x_{i}')^{2} ] $. Since $\sup_{i} E(x_{i}^{8}) < \infty $, we get $\sup_{i}E(z_{i}^4) < \infty$. Then, we have
\begin{align*}
\alpha_{p}^2  & \asymp \frac{E \left[ \left(\sum_{i=1}^{p} z_{i} \right)^2 \right]}{p^2} \\
& =  \frac{E \left[ \sum_{s=1}^{p} \sum_{t \in [ s - m, s+ m]} z_{s}z_{t}  \right]}{p^2} \\
& \leq \frac{(2m+1)p}{p^2} \sup_{i}E(z_{i}^2) \\
& = O \left( \frac{m}{p} \right)
\end{align*}
and
\begin{align*}
\gamma_{p}^2  & \asymp \frac{E \left[ \left(\sum_{i=1}^{p} z_{i} \right)^4 \right]}{p^4} \\
%& =  \frac{E \left[ \sum\limits_{s=1}^{p} \sum\limits_{t,u,v \in [ s - 4m, s+ 4m]} z_{s}z_{t}z_{u}z_{v} + \binom{4}{2} \sum\limits_{s=1}^{p} \sum\limits_{u \in [s+4m+1, p]}^{s = t,  u =v , s<u} z_{s}z_{t}z_{u}z_{v} \right]}{p^4} \\
& \asymp \frac{m^3p + m^2 p^2}{p^4} \sup_{i}E(z_{i}^4) \\
& = O \left( \frac{m^2}{p^2}\right).
\end{align*}
Similarly, we can show that
\begin{align*}
\beta_{q}^2 = O \left( \frac{m'}{q} \right) \text{ and } \lambda_{q}^2 = O \left( \frac{m'^2}{q^2}\right).
\end{align*}
Next, it follows that
\begin{align*}
\tau \alpha_{p} \lambda_{q} = O \left(\frac{m'\sqrt{m}}{\sqrt{q}} \right) = o(1) .
\end{align*}
The other results can be proved in a similar fashion.
\end{proof}

\subsection{Proof of Theorem \ref{thm:decomp2}}
\begin{proof} (i)\&(ii) Following the proof of Theorem \ref{thm:decomp}, we only need to check that $\mathcal{R}_{n} = o_p(1)$ still holds as $ n \wedge p \wedge q \rightarrow \infty $. Recall that the leading term is $\tau  \times (\widetilde{ \mathbf L}_X \cdot \widetilde{ \mathbf L}_Y  )$ and the remainder term is given as
	\begin{align*}
	\mathcal{R}_n = \frac{1}{2} \tau (\widetilde{ \mathbf L}_X \cdot \widetilde{ \mathbf R}_Y  ) +  \frac{1}{2} \tau (\widetilde{ \mathbf R}_X \cdot \widetilde{ \mathbf L}_Y  ) + \tau (\widetilde{ \mathbf R}_X \cdot \widetilde{ \mathbf R}_Y  ).
	\end{align*}
	Then, using Equation \eqref{eq:defDcov}, we have
	\begin{align*}
	(\widetilde{ \mathbf L}_X \cdot \widetilde{ \mathbf R}_Y  )
	=  &   \frac{1}{\binom{n}{2}} \frac{1}{2!} \sum\limits_{(s,t) \in \mathbf i_{2}^{n}} L_X(X_s, X_t) R_Y(Y_s, Y_t)   \\
	&  + \frac{1}{\binom{n}{4}} \frac{1}{4!} \sum\limits_{(s,t,u,v) \in \mathbf i_{4}^{n}} L_X(X_s, X_t) R_Y(Y_u, Y_v) \\
	& -\frac{2}{\binom{n}{3}} \frac{1}{3!} \sum\limits_{(s,t,u) \in \mathbf i_{3}^{n}} L _X(X_s, X_t) R_Y(Y_s, Y_u),
	\end{align*}
	and
	\begin{align*}
	(\widetilde{ \mathbf R}_X \cdot \widetilde{ \mathbf R}_Y  )
	=  &   \frac{1}{\binom{n}{2}} \frac{1}{2!} \sum\limits_{(s,t) \in \mathbf i_{2}^{n}} R_X(X_s, X_t) R_Y(Y_s, Y_t)   \\
	&  + \frac{1}{\binom{n}{4}} \frac{1}{4!} \sum\limits_{(s,t,u,v) \in \mathbf i_{4}^{n}} R_X(X_s, X_t) R_Y(Y_u, Y_v) \\
	& -\frac{2}{\binom{n}{3}} \frac{1}{3!} \sum\limits_{(s,t,u) \in \mathbf i_{3}^{n}} R_X(X_s, X_t) R_Y(Y_s, Y_u).
	\end{align*}
	To show that $ \mathcal{R}_n $ is asymptotically negligible, we consider the events $B_{\mathbf{X}},B_{\mathbf{Y}}$ and their complements $ B_{\mathbf{X}}^{c}, B_{\mathbf{Y}}^{c} $, where
	\begin{align*}
	B_{\mathbf{Y}} = \left\lbrace \min\limits_{1\leq s < t \leq n} \frac{|Y_{s} - Y_{t}|^2}{\tau_X^2} \leq \frac{1}{2} \text{ or } \max\limits_{1\leq s < t \leq n} \frac{|Y_{s} - Y_{t}|^2}{\tau_X^2}  \geq  \frac{3}{2} \right\rbrace .
	\end{align*}
	Then, under Assumption \ref{D4}, $ \text{as } n\wedge p \wedge q \rightarrow \infty $
	\begin{align*}
	P(B_{\mathbf{Y}}) & = P \left( \min\limits_{1\leq s < t \leq n} L_{Y}(Y_{s}, Y_{t}) \leq -\frac{1}{2} \text{ or } \max\limits_{1\leq s < t \leq n} L_{Y}(Y_{s}, Y_{t}) \geq  \frac{1}{2} \right) \\
	& = P\left( \bigcup\limits_{1\leq s < t \leq n} \left\lbrace L_{Y}(Y_{s}, Y_{t}) \leq -\frac{1}{2} \text{ or } L_{Y}(Y_{s}, Y_{t}) \geq \frac{1}{2} \right\rbrace \right) \\
	& \leq \sum\limits_{1\leq s < t \leq n} P \left( | L_{Y}(Y_{s}, Y_{y}) | \geq \frac{1}{2} \right) \\
	& < n^2 P \left( | L_{Y}(Y_{1}, Y_{2}) | \geq \frac{1}{2} \right) \\
	& \leq 4 n^2 E\left[ L_{Y}(Y_{1}, Y_{2})^2 \right] \\
	& = o(1) .
	\end{align*}
	Also notice that $P(B_{\mathbf{Y}}B_{\mathbf{X}}^{c}) \leq P(B_{\mathbf{Y}}) = o(1)$. Similarly, we have $ P(B_{\mathbf{X}}) = o(1), P(B_{\mathbf{X}}B_{\mathbf{Y}}^{c}) = o(1) $ and $ P(B_{\mathbf{Y}}B_{\mathbf{X}}) = o(1)$. By the proof of Proposition \ref{prop:taylor}, the remainder term can be written as
	\begin{align*}
	R_{X}(X_{s},X_{t}) = \int_{0}^1 \int_{0}^{1} v f^{(2)} \left(  uv L_{X}(X_s, X_t) \right) dudv  \times \left( L_{X}(X_s, X_t) \right)^2,
	\end{align*}
	where $f^{(2)}(t) = - \frac{1}{4} (1+t)^{-\frac{3}{2}}$ and similar formula holds for $Y$. Conditioned on the event $B_{\mathbf{X}}^{c}B_{\mathbf{Y}}^{c} $, we can easily show that
	\begin{align}
	\label{eq:2}
	| R_{X}(X_{s},X_{t})  | \leq \frac{\sqrt{2}}{4} \left( L_{X}(X_s, X_t) \right)^2, | R_{Y}(Y_{s},Y_{t})  | \leq \frac{\sqrt{2}}{4} \left( L_{Y}(Y_s, Y_t) \right)^2.
	\end{align}
	Notice that
	\begin{align*}
	&\frac{1}{\binom{n}{2}} \frac{1}{2!} \sum\limits_{(s,t) \in \mathbf i_{2}^{n}} R_X(X_s, X_t) R_Y(Y_s, Y_t) \\
	= & \frac{1}{\binom{n}{2}} \frac{1}{2!} \sum\limits_{(s,t) \in \mathbf i_{2}^{n}} R_X(X_s, X_t) R_Y(Y_s, Y_t) \mathbb{I}_{ \{  B_{\mathbf{X}}^{c}B_{\mathbf{Y}}^{c} \} } \\
	& + \frac{1}{\binom{n}{2}} \frac{1}{2!} \sum\limits_{(s,t) \in \mathbf i_{2}^{n}} R_X(X_s, X_t) R_Y(Y_s, Y_t) \mathbb{I}_{ \{  B_{\mathbf{X}} B_{\mathbf{Y}}^{c} \} } \\
	& + \frac{1}{\binom{n}{2}} \frac{1}{2!} \sum\limits_{(s,t) \in \mathbf i_{2}^{n}} R_X(X_s, X_t) R_Y(Y_s, Y_t) \mathbb{I}_{ \{  B_{\mathbf{X}}^{c} B_{\mathbf{Y}} \} } \\
	& + \frac{1}{\binom{n}{2}} \frac{1}{2!} \sum\limits_{(s,t) \in \mathbf i_{2}^{n}} R_X(X_s, X_t) R_Y(Y_s, Y_t) \mathbb{I}_{ \{  B_{\mathbf{X}} B_{\mathbf{Y}} \} } \\
	= & \rmnum{1} + \rmnum{2} + \rmnum{3 } + \rmnum{4}.
	\end{align*}
	For any $\epsilon >0$,
	$
	P( |\tau \times \rmnum{2}| > \epsilon) \leq P( B_{\mathbf{X}} B_{\mathbf{Y}}^{c} ) = o(1),
	$
	which implies that $ \tau \times \rmnum{2} = o_{p}(1) $. Similarly, $ \tau \times \rmnum{3} = o_{p}(1) $ and $ \tau \times \rmnum{4} = o_{p}(1) $. For term $\rmnum{1}$, by Equation \eqref{eq:2}, we have
	\begin{align*}
	| \rmnum{1}|  \leq &   \left\lbrace   \frac{1}{\binom{n}{2}} \frac{1}{2!} \sum\limits_{(s,t) \in \mathbf i_{2}^{n}} | R_X(X_s, X_t) R_Y(Y_s, Y_t)| \right\rbrace  B_{\mathbf{X}}^{c}B_{\mathbf{Y}}^{c} \\
	\leq & \frac{1}{8} \frac{1}{\binom{n}{2}}  \frac{1}{2!} \sum\limits_{(s,t) \in \mathbf i_{2}^{n}}  L_X(X_s, X_t)^2 L_Y(Y_s, Y_t)^2 \\
	\leq &  \frac{1}{8} \left\lbrace \left( \frac{1}{\binom{n}{2}}  \frac{1}{2!} \sum\limits_{(s,t) \in \mathbf i_{2}^{n}}  L_X(X_s, X_t)^4  \right) \left(\frac{1}{\binom{n}{2}}  \frac{1}{2!} \sum\limits_{(s,t) \in \mathbf i_{2}^{n}}  L_Y(Y_s, Y_t)^4\right) \right\rbrace^{\frac{1}{2}}.
	\end{align*}
	Next, by the Markov's inquality
	\begin{align*}
	P\left( \frac{1}{\binom{n}{2}}  \frac{1}{2!} \sum\limits_{(s,t) \in \mathbf i_{2}^{n}}  L_X(X_s, X_t)^4 > \epsilon \right) & \leq \frac{E \left[\frac{1}{\binom{n}{2}} \frac{1}{2!} \sum\limits_{(s,t) \in \mathbf i_{2}^{n}}   L_X(X_s, X_t)^4 \right]}{\epsilon} \\
	& = \frac{1}{\epsilon}  E \left[ L_X(X_1, X_2)^4 \right] \\
	& = \frac{1}{\epsilon}  \gamma_{p}^2.
	\end{align*}
	Thus, we have $ \frac{1}{\binom{n}{2}}  \frac{1}{2!} \sum\limits_{(s,t) \in \mathbf i_{2}^{n}}  L_X(X_s, X_t)^4 = O_p( \gamma_{p}^2 ) $ and similar proof shows that
	$
	\frac{1}{\binom{n}{2}}  \frac{1}{2!} \sum\limits_{(s,t) \in \mathbf i_{2}^{n}}  L_Y(Y_s, Y_t)^4 = O_p( \lambda_{q}^2 ).
	$
	So, we have $\tau \rmnum{1} = O_{p}( \tau \gamma_{p} \lambda_{q} )
	$
	and
	\begin{align*}
	\tau  \frac{1}{\binom{n}{2}} \frac{1}{2!} \sum\limits_{(s,t) \in \mathbf i_{2}^{n}} R_X(X_s, X_t) R_Y(Y_s, Y_t) = O_{p}( \tau \gamma_{p} \lambda_{q} ).
	\end{align*}
	Similarly, it can be shown that
	\begin{align*}
	&  \tau \frac{1}{\binom{n}{4}} \frac{1}{4!} \sum\limits_{(s,t,u,v) \in \mathbf i_{4}^{n}} R_X(X_s, X_t) R_Y(Y_u, Y_v) =  O_{p}( \tau \gamma_{p} \lambda_{q} ), \\
	& \tau \frac{2}{\binom{n}{3}} \frac{1}{3!} \sum\limits_{(s,t,u) \in \mathbf i_{3}^{n}} R_X(X_s, X_t) R_Y(Y_s, Y_u) = O_{p}( \tau \gamma_{p} \lambda_{q} ).
	\end{align*}
	In conclusion, we have $ \tau (\widetilde{ \mathbf R}_X \cdot \widetilde{ \mathbf R}_Y  ) = O_{p}( \tau \gamma_{p} \lambda_{q} ) $. Similarly, it can be shown that $ \tau (\widetilde{ \mathbf L}_X \cdot \widetilde{ \mathbf L}_Y  ) = O_{p}( \tau \alpha_{p} \beta_{q} ) $, $ \tau (\widetilde{ \mathbf L}_X \cdot \widetilde{ \mathbf R}_Y  ) = O_{p}( \tau \alpha_{p} \lambda_{q} ) $ and $ \tau (\widetilde{ \mathbf R}_X \cdot \widetilde{ \mathbf L}_Y  ) = O_{p}( \tau \gamma_{p} \beta_{q} ) $.
\end{proof}

\subsection{Proof of Theorem \ref{thm:decompHsic2}}
\begin{proof}
(i)\&(ii)  Continuing with the proof of Theorem \ref{thm:decompHsic}, we need to show that $\mathcal{R}_n = o_p (1)$ and $\gamma_{\mathbf X}$ is asymptotically euqal to $\tau_{X}$ as $n \wedge p \wedge q \rightarrow  \infty$ (similar result applies to $\gamma_{\mathbf Y}$ and $\tau_{Y}$). Recall that for all $s \neq t$,
\begin{align*}
L_{X}(X_{s}, X_{t}) = \frac{|X_{s} - X_{t}|^{2} - \tau_{X}^{2}}{ \tau_{X}^{2} }.
\end{align*}
Since for any $\epsilon > 0 $, under Assumption \ref{D4},
\begin{align*}
& P \left( \left| \frac{\text{median}\{|X_{s} - X_{t}|^{2}\} }{ \tau_{X}^{2} } -1  \right| > \epsilon \right) \\
   \leq & P \left( \min\limits_{1\leq s < t \leq n} L_{X}(X_{s}, X_{t}) \leq -\epsilon \text{ or } \max\limits_{1\leq s < t \leq n} L_{X}(X_{s}, X_{t}) \geq  \epsilon \right) \\
 = & P\left( \bigcup\limits_{1\leq s < t \leq n} \left\lbrace L_{X}(X_{s}, X_{t}) \leq -\epsilon \text{ or } L_{X}(X_{s}, X_{t}) \geq \epsilon \right\rbrace \right) \\
 \leq & \sum\limits_{1\leq s < t \leq n} P \left( | L_{X}(X_{s}, X_{t}) | \geq \epsilon \right) \\
 < & n^2 P \left( | L_{X}(X_{1}, X_{2}) | \geq \epsilon \right) \\
 \leq & \frac{1}{\epsilon^2} n^2 E\left[ L_{X}(X_{1}, X_{2})^2 \right] \\
 = & o(1).
\end{align*}
Thus, we have $\frac{\text{median}\{|X_{s} - X_{t}|^{2}\} }{ \tau_{X}^{2} } \overset{p}{\rightarrow} 1$ and
$
\frac{\tau_{X}}{\gamma_{\mathbf X}} = \sqrt{ \frac{ \tau_{X}^{2} }{\text{median}\{|X_{i} - X_{j}|^{2}\} } } \overset{p}{\rightarrow} 1.
$ Similar arguments can also be used to show that $ \frac{\tau_{Y}}{\gamma_{\mathbf Y}} \overset{p}{\rightarrow} 1$.

Notice that conditioned on $B_{\mathbf{X}}^{c}B_{\mathbf{Y}}^{c}$, for all $1 \leq s < t \leq n$, we have
\begin{align}
\label{in:bounded}
|L_{X}(X_{s}, X_{t})| < 1/2 \text{ and }
\frac{1}{2} < \frac{|X_{s} - X_{t}|^2}{\tau_X^2} < \frac{3}{2}.
\end{align}
Next, Inequalities \eqref{eq:2} and \eqref{in:bounded} together imply that
$$
\left| \frac{\tau_{X}}{\gamma_{\mathbf X}} + uv \left\lbrace \frac{L_{X}(X_s, X_t) }{2} + R_X (X_s, X_t) \right\rbrace \frac{\tau_{X}}{\gamma_{\mathbf X}}\right| \leq c, $$
where $c$ is some constant. Since we choose kernels $k$ and $l$ to be the Gaussian or Laplacian kernel, it can be shown that
\begin{align*}
\left| \int_{0}^1 \int_{0}^{1} v f^{(2)} \left( \frac{\tau_{X}}{\gamma_{\mathbf X}} + uv \left\lbrace \frac{L_{X}(X_s, X_t) }{2} + R_X (X_s, X_t) \right\rbrace \frac{\tau_{X}}{\gamma_{\mathbf X}} \right) dudv \right| \leq c',
\end{align*}
where $c'$ is some constant. Then, we can easily see from Equation \eqref{eq:rf} that $| R_{f}(X_{s},X_{t}) | \leq c' L_{X}(X_{s}, X_{t})^2$. Similar result holds for $Y$. Finally, Theorem \ref{thm:decompHsic2} can be shown using similar arguments as in the proof of Theorem \ref{thm:decomp2}.
\end{proof}

\subsection{Proof of Remark \ref{rem:D7}}
\begin{proof}
	When $k(x,y) = l(x,y) = |x-y|^2$,
	\begin{align*}
	& k_{st}(i) = -2 (x_{si} - E(x_{si}))(x_{ti} - E(x_{ti})), \\
	& l_{st}(j) = -2 (y_{sj} - E(y_{sj}))(y_{tj} - E(y_{tj})).
	\end{align*}
	Thus, we have
	\begin{align*}
	& E [U(X_s, X_t)^2] \\
	= & E \left[  \frac{1}{p} \sum\limits_{i=1}^{p} \sum\limits_{j=1}^{p} k_{st}(i) k_{st}(j)  \right] \\
	= & \frac{4}{p}  \sum\limits_{i=1}^{p} \sum\limits_{j=1}^{p} E \left[  (x_{si} - E[x_{si}])(x_{ti} - E[x_{ti}])  (x_{sj} - E[x_{sj}])(x_{tj} - E[x_{tj}]) \right] \\
	= &  \frac{4}{p}  \sum\limits_{i=1}^{p} \sum\limits_{j=1}^{p} cov^2(x_{i}, x_{j}) \\
	= & \frac{4}{p} Tr ( \bm{\Sigma}_{X}^2 ),
	\end{align*}
	and
	\begin{align*}
	& E [U(X_s, X_t)^4] \\
	= & E \left[  \frac{1}{p^2} \sum\limits_{i,j,r,w=1}^{p} k_{st}(i) k_{st}(j) k_{st}(r) k_{st}(w)  \right] \\
	= & \frac{16}{p^2}  \sum\limits_{i,j,r,w=1}^{p} E \Big[  (x_{si} - E[x_{si}])(x_{ti} - E[x_{ti}])  (x_{sj} - E[x_{sj}])(x_{tj} - E[x_{tj}]) \\
	& \quad \quad \quad \quad  \quad \quad  (x_{sr} - E[x_{sr}])(x_{tr} - E[x_{tr}])  (x_{sw} - E[x_{sw}])(x_{tw} - E[x_{tw}])  \Big] \\
	= &  \frac{16}{p^2}  \sum\limits_{i,j,r,w=1}^{p} E^2 \left[ (x_{i} - E[x_{i}]) (x_{j} - E[x_{j}]) (x_{r} - E[x_{r}]) (x_{w} - E[x_{w}]) \right] \\
	\asymp & \frac{m^3p + m^2p^2}{p^2} \sup\limits_{i} E^2(x_{i}^4) \\
	= & O\left( m^2 \right).
	\end{align*}
	Also,
	\begin{align*}
	& E [U(X_s, X_t)U(X_t, X_u)U(X_u, X_v)U(X_v, X_s)] \\
	= & E \left[  \frac{1}{p^2} \sum\limits_{i,j,r,w=1}^{p}  k_{st}(i) k_{tu}(j) k_{uv}(r) k_{vs}(w)  \right] \\
	= & \frac{16}{p^2}  \sum\limits_{i,j,r,w=1}^{p}  E \Big[  (x_{si} - E[x_{si}])(x_{ti} - E[x_{ti}])  (x_{tj} - E[x_{tj}])(x_{uj} - E[x_{uj}]) \\
	& \quad \quad \quad \quad  \quad \quad  (x_{ur} - E[x_{ur}])(x_{vr} - E[x_{vr}])  (x_{vw} - E[x_{vw}])(x_{sw} - E[x_{sw}])  \Big] \\
	= &  \frac{16}{p^2}  \sum\limits_{i,j,r,w=1}^{p} cov(x_{i}, x_{j})cov(x_{j}, x_{r}) cov(x_{r}, x_{w}) cov(x_{w}, x_{i}) \\
	= & \frac{16}{p^2}  Tr (\bm{\Sigma}_{X}^4) \\
	\asymp & \frac{m^3p}{p^2} \sup\limits_{i} E^4(x_{i}^2) \\
	= & O \left(\frac{m^3}{p}\right).
	\end{align*}
\end{proof}
%\begin{proof}
%	(i)  By the proof of Proposition 2.1 in \cite{zhang2018conditional}, we have the results.
%	(ii) Consequence of Proposition 2.2 in \cite{zhang2018conditional}.
%\end{proof}

\subsection{Proof of Theorem \ref{thm:key2}}
\begin{proof}
Firstly, the following lemma would be useful.
	\begin{lemma}
		\label{lem:hoeff}
		Under null, we have
		\begin{align*}
		\frac{1}{\mathcal{S}} uCov^2_{n}(\mathbf X, \mathbf Y)   = \frac{1}{\binom{n}{2} \mathcal{S} } \sum\limits_{1\leq s < t \leq n} H \left(Z_{s}, Z_{t} \right) + \mathcal{R}_{n},
		\end{align*}
		where $\sqrt{\binom{n}{2}}\mathcal{R}_{n,p,q} = o_{p}(1)$ as $n \wedge p \wedge q \rightarrow \infty$, $Z_{s} = (X_{s}, Y_{s}) $ and $H( \cdot , \cdot )$ is defined as
		$$
		H \left(Z_{s}, Z_{t} \right) := U(X_{s}, X_{t})V(Y_{s}, Y_{t}) .
		$$
	\end{lemma}
	
	\begin{proof}
		Firstly, sample $uCov$ can be written as
		\begin{align*}
		uCov^2_{n}(\mathbf X, \mathbf Y)  & = \frac{1}{\sqrt{pq}}\sum^{p}_{i=1}\sum^{q}_{j=1} (  \widetilde{\mathbf K}(i) \cdot \widetilde{\mathbf L}(j) ) \\
		& = (\frac{1}{\sqrt{p}}\sum^{p}_{i=1}  \widetilde{\mathbf K}(i) \cdot \frac{1}{\sqrt{q}}\sum^{q}_{j=1}\widetilde{\mathbf L}(j) ) \\
		& = (  \widetilde{\overline{\mathbf K}} \cdot \widetilde{\overline{\mathbf L}}),
		\end{align*}
		where $ \overline{\mathbf K} = (\overline{k}_{st})_{s,t=1}^{n} $, $\overline{\mathbf L} = (\overline{l}_{st})_{s,t=1}^{n}$,
		$
		\overline{k}_{st} = \frac{1}{\sqrt{p}}\sum^{p}_{i=1} k(x_{si},x_{ti})$ and $\overline{l}_{st} = \frac{1}{\sqrt{q}}\sum^{q}_{i=1} l(y_{si},y_{ti}).
		$
		Thus, $ uCov^2_{n}(\mathbf X, \mathbf Y) $ is just $ dCov^2_{n}(\mathbf X, \mathbf Y) $ with kernel $\overline{K}$ defines as $ \overline{K}(X_{s}, X_{t}) = \overline{k}_{st}$ and $\overline{L}(Y_{s}, Y_{t}) = \overline{l}_{st}$. Notice that
		\begin{align*}
		& \overline{K}(X_{s}, X_{t}) - E[\overline{K}(X_{s}, X_{t})| X_{s}] - E[\overline{K}(X_{s}, X_{t})|X_{t}] + E[ \overline{K}(X_{s}, X_{t}) ] = \frac{1}{\sqrt{p}}\sum^{p}_{i=1} k_{st}(i), \\
		& \overline{L}(Y_{s}, Y_{t}) - E[\overline{L}(Y_{s}, Y_{t}) | Y_{s}] - E[\overline{L}(Y_{s}, Y_{t}) | Y_{t}] + E[ \overline{L}(Y_{s}, Y_{t})  ] = \frac{1}{\sqrt{q}}\sum^{q}_{i=1} l_{st}(i),
		\end{align*}
		where $k_{st}(i)$ and $l_{st}(i)$ are the double centered kernel distance defined in Section \ref{sec:uni}. By Proposition 2.1 of \cite{yao2017testing}, we have
		\begin{align*}
		\frac{1}{\mathcal{S}}(  \widetilde{\overline{\mathbf K}} \cdot \widetilde{\overline{\mathbf L}}) & = \frac{1}{\binom{n}{2} \mathcal{S} } \sum\limits_{1\leq s < t \leq n} U(X_{s,n},X_{t,n}) V(Y_{s,n},Y_{t,n}) + \mathcal{R}_{n,p,q} \\
		& = \frac{1}{\binom{n}{2} \mathcal{S} } \sum\limits_{1\leq s < t \leq n} \frac{1}{\sqrt{p}}\sum^{p}_{ i=1} k_{st}(i) \frac{1}{\sqrt{q}}\sum^{q}_{i=1} l_{st}(i) + \mathcal{R}_{n,p,q},
		\end{align*}
		where $ \sqrt{\binom{n}{2}} \mathcal{R}_{n,p,q} = o_{p}(1)$ as $n \wedge p \wedge q \rightarrow \infty$.
	\end{proof}
	By Lemma \ref{lem:hoeff}, we have
	\begin{align*}
	& \sqrt{\binom{n}{2}} \frac{uCov_{n}^{2}(\mathbf{X}, \mathbf{Y})}{\mathcal{S}}
	= \frac{1}{\sqrt{\binom{n}{2}} \mathcal{S}} \sum\limits_{1\leq s < t \leq n} H \left(Z_{s}, Z_{t} \right) + \sqrt{\binom{n}{2}} \mathcal{R}_{n,p,q},
	\end{align*}
	where $ \sqrt{\binom{n}{2}} \mathcal{R}_{n,p,q} = o_{p}(1) .$ By similar proof of Theorem 2.1 in \cite{zhang2018conditional}, under $H_{0}$, we have
	\begin{align*}
	\frac{1}{\sqrt{\binom{n}{2}} \mathcal{S}} \sum\limits_{1\leq s < t \leq n} H \left(Z_{s}, Z_{t} \right) \overset{d}{\rightarrow} N(0,1).
	\end{align*}
\end{proof}

\subsection{Proof of Proposition \ref{prop:exactT2}}
\begin{proof}
	Notice that by the proof of Theorem 2.2 in \cite{zhang2018conditional}, under null
	\begin{align}
	\label{eq:uni}
	\frac{ uCov_{n}^{2}(\mathbf{X}, \mathbf{X}) }{E[ U(X, X')^2 ]} \overset{p}{\rightarrow} 1, \quad \frac{ uCov_{n}^{2}(\mathbf{Y}, \mathbf{Y}) }{E[ V(Y, Y')^2 ]} \overset{p}{\rightarrow} 1.
	\end{align}
	So, by Theorem \ref{thm:key2}
	\begin{align*}
	\sqrt{\binom{n}{2}} \frac{uCov_{n}^{2}(\mathbf{X}, \mathbf{Y})}{\sqrt{ uCov_{n}^{2}(\mathbf{X}, \mathbf{X}) uCov_{n}^{2}(\mathbf{Y}, \mathbf{Y}) }} \overset{d}{\rightarrow} N(0,1),
	\end{align*}
	and also
	\begin{align*}
	\frac{uCov_{n}^{2}(\mathbf{X}, \mathbf{Y})}{\sqrt{ uCov_{n}^{2}(\mathbf{X}, \mathbf{X}) uCov_{n}^{2}(\mathbf{Y}, \mathbf{Y}) }} \overset{p}{\rightarrow} 0.
	\end{align*}
	As a consequence, we have $T_u \overset{d}{\rightarrow} N(0,1).$
\end{proof}

\subsection{Proof of Proposition \ref{prop:uni2}}
\begin{proof}
Based on Theorem \ref{thm:decomp2} and \ref{thm:decompHsic2}, the results follow similarly from the proof of Proposition \ref{prop:uni}.
\end{proof}

\subsection{Proof of Corollary \ref{cor:uni2}}
\begin{proof}
	(i) If $ R = mdCov $, the result follows from Proposition \ref{prop:exactT2} and the following observation
	\begin{align*}
	\sqrt{\binom{n}{2}}  \frac{R_{n}^{2}(\mathbf{X}, \mathbf{Y})}{\sqrt{ R_{n}^{2}(\mathbf{X}, \mathbf{X}) R_{n}^{2}(\mathbf{Y}, \mathbf{Y})}} = \sqrt{\binom{n}{2}} \frac{uCov_{n}^{2}(\mathbf{X}, \mathbf{Y})}{\sqrt{ uCov_{n}^{2}(\mathbf{X}, \mathbf{X}) uCov_{n}^{2}(\mathbf{Y}, \mathbf{Y}) }}.
	\end{align*}
	
	(ii) Recall that when $k(x,y) = l(x,y) = |x-y|^2$,
	$E [U(X_s, X_t)^2] = \frac{4}{p}  Tr ( \bm{\Sigma}_{X}^2 )$ and
	 $E [V(Y_s, Y_t)^2] = \frac{4}{q} Tr ( \bm{\Sigma}_{Y}^2 ).$
	 If $R = hCov$, by  Proposition \ref{prop:uni2}, we have
	 \begin{align*}
	  \sqrt{\binom{n}{2}} \tau \times \frac{R_{n}^{2}(\mathbf{X}, \mathbf{Y})}{\mathcal{S}}  =A_{p}B_{q} \sqrt{\binom{n}{2}}\frac{  uCov^2_n(\mathbf X, \mathbf Y) }{\mathcal{S}} + \sqrt{\binom{n}{2}}\frac{ \mathcal{R}_{n}'' }{\mathcal{S}},
	 \end{align*}
	 where $ A_{p} =  \frac{\sqrt{p}}{ 2 \tau_{X} }  f^{(1)} \left( \frac{\tau_{X}}{\gamma_{\mathbf X}} \right)\frac{\tau_{X}}{\gamma_{\mathbf X}}$ and $B_{q} = \frac{\sqrt{q}}{2 \tau_{Y}} g^{(1)} \left( \frac{\tau_{Y}}{ \gamma_{\mathbf Y}} \right) \frac{\tau_Y}{\gamma_{\mathbf Y}}$. By Theorem \ref{thm:key2},
	\begin{align*}
	A_{p}B_{q} \sqrt{\binom{n}{2}}  \frac{ uCov^2_n(\mathbf X, \mathbf Y) }{\mathcal{S}} \overset{d}{\rightarrow} c N(0,1),
	\end{align*}
	where $c$ is some constant. Also notice that
	\begin{align*}
	\left| \sqrt{\binom{n}{2}}\frac{ \mathcal{R}_{n}'' }{\mathcal{S}} \right| \leq \left| \frac{n \mathcal{R}_{n}'' }{4\sqrt{  \frac{1}{p}  Tr ( \bm{\Sigma}_{X}^2 )  \frac{1}{q}  Tr ( \bm{\Sigma}_{Y}^2 ) }} \right| = o_p(1).
	\end{align*}
	Thus, we have
	\begin{align*}
	\sqrt{\binom{n}{2}} \tau \times \frac{R_{n}^{2}(\mathbf{X}, \mathbf{Y})}{\mathcal{S}} \overset{d}{\rightarrow} c N(0,1).
	\end{align*}
	Next, under Assumption \ref{D4}, by Equation \eqref{eq:uni} and Proposition \ref{prop:uni2}
	\begin{align*}
	& \tau \times \frac{\sqrt{ R_{n}^{2}(\mathbf{X}, \mathbf{X}) R_{n}^{2}(\mathbf{Y}, \mathbf{Y})} }{\mathcal{S}} \\
	= & \sqrt{ \left( \frac{A_{p}^2 uCov^2_n(\mathbf X, \mathbf X) + \mathcal{R}''' }{ E[ U(X, X')^2 ] } \right)\left( \frac{B_{q}^2 uCov^2_n(\mathbf Y, \mathbf Y) + \mathcal{R}'''' }{ E[ U(Y, Y')^2 ] } \right) } \\
	 \overset{p}{\rightarrow} & c.
	\end{align*}
    Notice that Under Assumptions \ref{D1} and \ref{D4}, Proposition \ref{prop:uni2} also holds similarly when $\mathbf{X} = \mathbf{Y}$ or $ \mathbf{Y} = \mathbf{X} $. So $\mathcal{R}'''$ and $\mathcal{R}'''' $ are both negligible. Thus, we have
	\begin{align*}
	\sqrt{\binom{n}{2}}  \frac{R_{n}^{2}(\mathbf{X}, \mathbf{Y})}{\sqrt{ R_{n}^{2}(\mathbf{X}, \mathbf{X}) R_{n}^{2}(\mathbf{Y}, \mathbf{Y})}} \overset{d}{\rightarrow} N(0,1)
	\end{align*}
	and consequently $T_{R} \overset{d}{\rightarrow} N(0,1) $. Similarly, it can be proved for $R = dCov$.
\end{proof}
\end{document}